\documentclass[11pt,a4paper]{amsart}
\usepackage[utf8]{inputenc}
\usepackage{amsmath}
\usepackage{amsfonts}
\usepackage{amssymb}
\usepackage{graphicx}
\usepackage{booktabs}
\usepackage{multirow}
\usepackage{adjustbox}
\usepackage{subfigure}
\usepackage[ruled,vlined]{algorithm2e}
\usepackage{amsaddr}
\usepackage[left=2cm,right=2cm,top=2cm,bottom=1cm]{geometry}  
\usepackage{mathtools}
\usepackage{accents}
\usepackage{mathrsfs}
\usepackage{tikz}
\usepackage[ruled,vlined]{algorithm2e}
\usepackage{enumitem}
\usepackage{nicefrac}
\usetikzlibrary{decorations.pathreplacing}
\usetikzlibrary{fadings}

\newcommand{\D}{{\mathop{}\!\mathrm{d}}} 

\newcommand{\R}{\mathbb{R}}

\newcommand{\N}{\mathbb{N}}

\newcommand{\loc}{\operatorname{loc}}
\newcommand{\PP}{\mathbb{P}}

\newcommand{\E}{\mathbb{E}}

\newcommand{\T }{\mathcal{T}}

\newcommand{\cA}{\mathcal{A}}
\newcommand{\one}{ 1 \hspace{-3pt} \mathrm{l}} %

\newcommand{\aloc}{{a_{\operatorname{loc}}}}
\newcommand{\Omloc}{{\Omega_{\operatorname{loc}}}}
\newcommand{\ab}{{\mathbf{a}}}

\numberwithin{equation}{section}  
\newtheorem{defn}{Definition}[section]

\newtheorem{rem}[defn]{Remark}
\newtheorem{thm}[defn]{Theorem}
\newtheorem{prop}[defn]{Proposition}
\newtheorem{cor}[defn]{Corollary}
\newtheorem{lem}[defn]{Lemma}

\newtheorem{asu}[defn]{Assumption}

\usepackage{hyperref}
\usepackage[textwidth=3.2cm]{todonotes}
\usepackage{xcolor}
\usepackage{etoolbox}

\definecolor{bb}{RGB}{0,0,255}

\hypersetup{
    colorlinks,
    linkcolor={red!50!black},
    citecolor={blue!50!black},
    urlcolor={blue!80!black}
}
\makeatletter
\def\@settitle{%
  \begin{center}%
    \baselineskip14\p@\relax
    \normalfont\LARGE\bfseries 
    \@title
  \end{center}%
}
\makeatother

\title[Non-concave control under model uncertainty]{Non-concave stochastic optimal control in finite discrete time under model uncertainty}

\author[A. Neufeld, J. Sester]{ Ariel Neufeld$^{1}$, Julian Sester$^{2}$}

\begin{document}

\vspace*{-1.0cm}
\maketitle

\begin{center}
\normalsize{\today} \\ \vspace{0.5cm}
\small\textit{$^{1}$NTU Singapore, Division of Mathematical Sciences,\\ 21 Nanyang Link, Singapore 637371.\\
$^{2}$National University of Singapore, Department of Mathematics,\\ 21 Lower Kent Ridge Road, 119077.}                                                                                                                              
\end{center}

\begin{abstract}~
In this article we present a general framework for non-concave 
robust stochastic control problems 
{under model uncertainty} 
in a discrete time finite horizon setting. 
Our framework allows to consider a variety of different path-dependent ambiguity sets of probability measures comprising, as a natural example, the ambiguity set defined via Wasserstein-balls around path-dependent reference measures {with path-dependent radii}, as well as parametric classes of probability distributions. We establish a dynamic programming principle which allows to derive both optimal control and worst-case measure by solving recursively a sequence of one-step optimization problems. 
{Moreover, we derive upper bounds for the difference of the values of the robust and non-robust stochastic control problem in the Wasserstein uncertainty and parameter uncertainty case.}

As a concrete application, we study the robust hedging problem of {financial derivatives} under an asymmetric (and non-convex) loss function accounting for different preferences of sell- and buy side when it comes to the hedging of financial derivatives. As our entirely data-driven ambiguity set of probability measures, we consider Wasserstein-balls around the empirical measure derived from real financial data. We demonstrate that during adverse scenarios such as a financial crisis, our robust approach outperforms typical model-based hedging strategies such as the classical \emph{Delta-hedging} strategy as well as the hedging strategy obtained in the non-robust setting with respect to the empirical measure and therefore overcomes the problem of model misspecification in such critical periods. 
\\

\end{abstract}

{
\section{Introduction}
Consider an agent executing a sequence of actions based on the observation she makes in her random environment 
with the goal to maximize her expected cumulative payoff at a predetermined future maturity date.  
The agent may take decisions based on the current and past observations of the environment,  
as well as based on her past executions. 
Sequential decision making problems of this type are usually referred to as (non-Markovian)
\emph{stochastic control} problems and 
  have found a wide range of applications including many 
  in physics 
  (e.g., \cite{freund2000stochastic, guerra1983quantization, milstein2004stochastic, yasue1981quantum}), 
  economics 
  (e.g., \cite{chow1970optimal,fleming2004application, kendrick2005stochastic, leonard1992optimal, stokey2008economics}), 
  and finance 
  (e.g., \cite{benth2001existence, bertsekas1996stochastic, pham2009continuous, touzi2002stochastic, touzi2012optimal, zariphopoulou2001stochastic}), to name but a few.

However, modeling the underlying probability distribution of (the transition of)
the environment 
in stochastic control problems is, in practice, extraordinarily difficult, as the model choice is subject to a high degree of possible model misspecification, typically referred to as \emph{Knightian uncertainty} (\cite{knight1921risk}).
To account for Knightian uncertainty, 
one
typically 
considers a \textit{set} of probability measures, instead of a singleton, which represents the candidates for the true but to the agent unknown probability measure; we refer to \cite{hansen2008robustness} and the references therein for a detailed discussion on model uncertainty.
%
 
Following the seminal work by \cite{gilboa1989maxmin}, the agent then chooses a sequence of actions in order to  \textit{robustly} optimize her expected payoff at maturity with respect to the worst-case probability measure among those in the ambiguity set. 

There are two natural approaches to define the set of probability measures describing the ambiguity. In the first approach, one starts with a \emph{reference probability measure} that can be regarded as the best guess or estimate of the stochastic behavior of the environment. Based on this reference probability measure one then considers  
all probability measures that are in a certain sense close to the reference probability measure, where proximity is typically  measured via a distance on the set of probability measures. The corresponding decision problem in this framework is known as \textit{distributionally robust optimization} (DRO).
This approach can be seen as a purely data-driven approach whenever the corresponding reference measure is estimated solely using data (e.g.\ using the empirical measure).
 We refer to \cite{ben-tal2013robust, erdougan2006ambiguous, mohajerin2018data} for the seminal works in DRO with respect to different distances. 
In the second approach, one fixes a family of parametric distributions where one assumes that the true but unknown law is an element of this family of distributions, but the corresponding parameters are unknown. This is known as \textit{parametric uncertainty}, see, e.g., 
\cite{easley1988controlling, horst2022portfolio, vitus2015stochastic}.

In this paper, we analyze a general multi-period stochastic control problem  under model uncertainty in finite discrete time. Hereby, the corresponding set of probability measures describing model uncertainty are allowed to be \textit{non-dominated}.

Multi-period stochastic control problems under model uncertainty in finite discrete time  with respect to non-dominated probability measures describing the ambiguity have been well-studied in recent years, especially in the financial context of robust utility maximization. Using the general framework of \cite{BouchardNutz} describing model uncertainty in a discrete time multi-period financial market together with its robust arbitrage characterization, \cite{nutz2016utility} first analyzed the robust utility maximization in a  frictionless market with respect to a utility function bounded from above where the set of investment strategies are restricted to those retaining nonnegative wealth. This result in the same setting has then been extended in \cite{blanchard2018multiple}, allowing for unbounded utilities. Later, \cite{carassus2023discrete} generalized \cite{blanchard2018multiple, nutz2016utility} by allowing strategies which could  lead to negative wealth during the trading period. While \cite{ blanchard2018multiple, carassus2023discrete, nutz2016utility} analyzed solely the primal problem and solved the utility maximization problem using dynamic programming (DPP), \cite{bartl2019exponential} combined duality arguments together with DPP methods to solve the robust utility maximization problem with respect to the exponential utility function also in a frictionless market.  \cite{neufeld2018robust} analyzed a general multi-period stochastic control problem  under model uncertainty in finite discrete time which is suited to analyze utility maximization in financial markets with frictions employing dynamic programming. Moreover, \cite{bartl2019robust} obtained existence of optimal strategies for the robust utility maximization problem with random endowment without employing dynamic programming and hence do not require a time-consistency 
assumption on the ambiguity set of probability measures, but instead work under the set-theoretic assumption that medial limits exist, an assumption which cannot be proved under the usual axioms of the Zermelo–Fraenkel set theory with the axiom of choice (ZFC). Furthermore, \cite{rasonyi2021utility} obtained existence of optimal investment strategies for the robust utility maximization under an alternative framework of model uncertainty described by a set of stochastic processes. We highlight that all the above literature \cite{bartl2019exponential, bartl2019robust, blanchard2018multiple, carassus2023discrete, neufeld2018robust, nutz2016utility, rasonyi2021utility} assume that the corresponding utility function is \textit{concave}.

For \textit{non-concave} utilities, \cite{neufeld2019nonconcave} analyzed a general multi-period stochastic control problem  under model uncertainty in finite discrete time, but under the restrictive assumption that the optimal strategies can only take values on a grid (in particular discrete set). Recently, \cite{carassus2024nonconcave} analyzed the non-concave robust utility maximization problem in a multi-period finite discrete time setting under the assumption that the additional set-theoretic assumption called \textit{Projective Determinancy} holds, an axiom which cannot be proved to hold under the ZFC. Moreover, \cite{bayraktar2023nonparametric, bayraktar2022data} analyzed adaptive-robust control problems, where the loss function is allowed to be non-concave. All the works \cite{bayraktar2023nonparametric, bayraktar2022data,carassus2024nonconcave,neufeld2019nonconcave} apply dynamic programming to solve the corresponding stochastic control problem. Furthermore, \cite{bauerle2021q,chen2019distributionally,neufeld2023mdp,uugurlu2018robust,xu2010distributionally} analyzed Markov decision problems under model uncertainty where the corresponding reward function is allowed to be non-concave. This corresponds to a \textit{Markovian} stochastic control problem in an \textit{infinite horizon} setup.

The goal of this paper is to provide a general framework for stochastic optimal control problems under model uncertainty \textit{allowing for non-concave utility functions}. Compared to the literature above (especially in the concave utility setup), we impose more regularity on the loss function and compactness assumptions on the set of kernels describing the ambiguity as well as on the set of strategies. 
We refer to Section~\ref{sec:Comparison} for a more detailed discussion on the conditions imposed on the utility function and the set of priors also in comparison to the existing literature described above.

That said, our  setup then allows  to solve the robust stochastic control problem also in the non-concave setting using dynamic programming, without requiring any additional set-theoretic assumption nor assuming that the set of strategies takes values on a grid. 
More precisely, we establish  in Theorem~\ref{thm_main_result} a dynamic programming principle  which allows us to obtain both optimal control and worst-case measure by solving recursively a sequence of one step optimization problems. 

From a technical point of view, we introduce a new stability condition on the ambiguity sets of probability measures (see Assumption~\ref{asu_P}~(iii)) which then, together with an application of Berge's maximum principle \cite{berge}, allows to derive the main result described above.

We demonstrate  in Theorem~\ref{thm_wasserstein_ambiguity} and Theorem~\ref{thm_param_assumptions} that our general assumptions imposed on the set of priors (see Assumption~\ref{asu_P}) are naturally fulfilled by important classes of ambiguity sets such as Wasserstein-balls around path-dependent reference probability measures with path-dependent radii, as well as parametric classes of probability distributions.

Furthermore, we derive in Theorem~\ref{thm:RobustVSNonRobust} upper bounds for the difference of the values of the robust and non-robust stochastic control problem in the Wasserstein uncertainty and parameter uncertainty case.}

As a concrete application, 
we study {in Section~\ref{sec_applications}}
the robust hedging of {financial derivatives} under an asymmetric (and non-convex) loss function accounting for different preferences of sell- and buy side when it comes to the hedging of financial derivatives (\cite{broll2010prospect}, \cite{carbonneau2023deep}, and  \cite{gobet2020option}).
As our entirely data-driven ambiguity set of probability measures, we consider Wasserstein-balls around the empirical measure derived from real financial data. 
We demonstrate that during adverse scenarios such as a financial crisis, our robust approach 
outperforms typical model-based hedging strategies such as the classical \emph{Delta-hedging} strategy as well as the hedging strategy obtained in the non-robust setting with respect to 
the empirical measure  
and therefore overcomes the problem of model misspecification in such critical periods. To provide evidence, we test 
on data from the peak of the financial disruptions during the  COVID-19 crisis in March $2020$ which constitutes a recent example for a major adverse distribution change in financial markets.

The remainder of this paper is as follows. In Section~\ref{sec_setting} we introduce the setting of our 
stochastic control problem under model uncertainty. Section~\ref{sec_dynamic_programming} contains 
the dynamic programming principle guaranteeing the existence of an  optimizer as well as a worst-case probability measure for our  stochastic control problem under model uncertainty.
In Section~\ref{sec_ambiguity} we present several possible ambiguity sets of probability measures  and several sets of possible controls satisfying our assumptions. 
{In
Section~\ref{sec_RobustVSNonRobust} we establish the upper bounds for the difference of the values of robust and non-robust stochastic control problems. 
In Section~\ref{sec:Comparison} we compare in details the assumptions and results of this paper with existing literature and discuss our contribution to the literature.
{Section~\ref{sec_applications} presents as an application a robust hedging problem under an asymmetric loss function which we solve numerically.} 
The proofs of all results of this paper are provided in Section~\ref{sec_proofs}. 
}

\section{Setting}\label{sec_setting}

To formulate our optimization problem, we consider a fixed and finite time horizon $T \in \N$ as well as a closed set $\Omloc \subseteq \R^d$ for some $d \in \N$. Then, we introduce the set 
\[
\Omega^t:= \underbrace{\Omloc \times \cdots \times \Omloc}_{t-\text{ times }},\qquad t=1\dots,T,
\]
and we define a filtration $(\mathcal{F}_t)_{t=0,\dots,T}$ by setting $\mathcal{F}_t:=\mathcal{B}(\Omega^t)$ for $t=1,\dots,T$ as well as $\mathcal{F}_0:= \{ \emptyset, \Omega\}$, where we abbreviate 
\[
(\Omega, \mathcal{F}):=(\Omega^T, \mathcal{F}_T).
\]
%
%
{ At each time $t=0,\dots,T-1$, the agent takes an action in order to optimize the stochastic control problem.
	%
We consider at each time $t=0,\dots,T-1$ a set of (possibly path-dependent) controls $\mathcal{A}_t$ corresponding to the actions an agent can take at  time $t$.
	\begin{asu}\label{asu_A}
		For every $t\in \{1,\dots,T-1\}$ we assume the following.
		\begin{itemize}
			\item[(i)]
			Let $m_t \in \N$. The correspondence\footnote{Given two spaces $\mathcal{X}, \mathcal{Y}$, we refer to $\Lambda:\mathcal{X}\twoheadrightarrow \mathcal{Y}$ as \textit{correspondence} or \textit{set-valued map}  if
				$\Lambda(x)\subseteq \mathcal{Y}$ for each $x \in \mathcal{X}$.}
				 $\Omega^t \ni 
			\omega^t \twoheadrightarrow \mathcal{A}_t(\omega^t) \subseteq \R^{m_t}$ is compact-valued, upper-hemicontinuous\footnote{We refer to \cite[Chapter~17]{Aliprantis} for the notions \textit{upper-hemicontinuous} and \textit{lower-hemicontinuous} for correspondences.}, and non-empty. 
			\item[(ii)]
			There exists some $L_{\mathcal{A},t}\geq 0$ such that for all $\omega^{t}:= (\omega_1,\dots,\omega_t) \in \Omega^t$, $\widetilde{\omega}^{t}:= (\widetilde{\omega}_1,\dots,\widetilde{\omega}_t) \in \Omega^t$ and for all $a \in \mathcal{A}_t\left(\omega^{t}\right)$ there exists some $\widetilde{a} \in \mathcal{A}_t\left(\widetilde{\omega}^{t}\right)$ satisfying
			\begin{equation}\label{eq_condition_A}
				\|a-\widetilde{a}\|\leq L_{\mathcal{A},t}\left( \sum_{i=1}^t \|\omega_i - \widetilde{\omega}_i \|  \right).
			\end{equation}
			\item[(iii)] Let $m_0 \in \N$. The set $\mathcal{A}_0 \subseteq \R^{m_0}$ is non-empty and compact.
		\end{itemize}
		%
			%
			%
		\end{asu}
		\begin{rem} \label{rem:control}
By Lemma~\ref{lem:Lipschitz-implies-LHC}, Assumption~\ref{eq_condition_A}~(ii) implies that for every $t\in \{1,\dots,T-1\}$ the correspondence $\Omega^t \ni 
		\omega^t \twoheadrightarrow \mathcal{A}_t(\omega^t) \subseteq \R^{m_t}$ is also lower-hemicontinuous. 
		\end{rem}

		Then, we define the set of 
		controls $\mathfrak{A}$ by 
		\begin{align}
			\label{eq:def:controls}
			\mathfrak{A}:= \bigg\{ \mathbf{a}=(a_t)_{t=0,\dots,T-1}~\bigg|~ a_0 \in \mathcal{A}_0,
			\text{ and for all $t=1,\dots,T-1$: }~a_t: \Omega^t \rightarrow \R^{m_t} \text{ $\mathcal{F}_t$-meas.},\  
			a_t(\cdot) \in \mathcal{A}_t(\cdot)
			\bigg\}.
		\end{align}
		Furthermore, for any $s,t\in \{1,\dots, T\}$ with $s-1\leq t$ we denote
		\begin{equation}\label{notation_graph}
			\Omega^t \times \mathcal{A}^s:= \Big\{\big((\omega_1,\dots,\omega_t),(a_0,\dots, a_{s-1})\big) \in \Omega^t \times \bigtimes_{i=0}^{s-1}\R^{m_i} \, \Big|\,  a_0 \in \mathcal{A}_0, a_i \in \mathcal{A}_i(\omega_1,\dots,\omega_i) \ \forall i=1,\dots, s-1\Big\}.
		\end{equation}
	} 
	
	To define our optimization problem, we consider
	a (utility) function $\Psi: \Omega \times \mathcal{A}^T \rightarrow \R$ which fulfills the assumptions below.
	\begin{asu}\label{asu_psi}
	Fix $p \in \N_0$. 
	 Then, we assume that the map $\Psi: \Omega \times \cA^T \to \R$
satisfies the following:
		\begin{itemize}
			{
				\item[(i)]
				There exists some $L_{\Psi}\geq 0$ and  $\alpha \in (0,1]$ such that for all $(\omega^T,a^T) = \big((\omega_1,\dots,\omega_{T}),(a_0,\dots,a_{T-1})\big)$, $(\widetilde{\omega}^T, \widetilde{a}^T)= \big((\widetilde{\omega}_1,\dots,\widetilde{\omega}_{T}),(\widetilde{a}_0,\dots,\widetilde{a}_{T-1})\big) \in \Omega\times \mathcal{A}^T$
				we have
				\begin{equation}\label{eq_Lipschitz_1}
					\left|\Psi(\omega^T,a^T)-\Psi(\widetilde{\omega}^T,\widetilde{a}^T)\right| \leq L_{\Psi} \cdot \left( \sum_{i=1}^T\left\|\omega_i-\widetilde{\omega}_i \right\|^{\alpha}+\|a_{i-1}-\widetilde{a}_{i-1}\|^{\alpha}\right).
				\end{equation}
			}
			\item[(ii)]
			There exists some $C_{\Psi}\geq 1$ such that for all $(\omega^T,a^T) = \big((\omega_1,\dots,\omega_{T}),(a_0,\dots,a_{T-1})\big) \in \Omega\times \mathcal{A}^T$ we have
			\[
			\left|\Psi(\omega^T,a^T)\right| \leq C_{\Psi}\cdot \left(1+\sum_{i=1}^{T} \|\omega_i\|^p\right).
			\]
			
		\end{itemize}
	\end{asu}
%
%
Next, we aim to define our set of probability measures which describe the model uncertainty. To that end, we first denote for every $k \in \N$, $\mathbb{X} \subseteq \R^k$, and $q \in \N_0$ by
\[
C_q(\mathbb{X}, \R):=\left\{g \in C(\mathbb{X},\R) ~\middle|~ \|g\|_{C_q(\mathbb{X})}:=\sup_{x \in X }\frac{|g(x)|}{1+\|x\|^q}< \infty\right\}
\]
the set of continuous functions $g: \mathbb{X}\rightarrow \R$ which additionally possess polynomial growth at most of degree $q$,
where $C(\mathbb{X},\R)$ denotes the set of continuous functions from $\mathbb{X}$ to  $\R$, and where here and in the following $\| \cdot \|$ always denotes the Euclidean norm on the respective Euclidean space. 
Typically, we will consider $\mathbb{X}:=\Omloc$. 
Moreover, 
we denote by 
$(\mathcal{M}_1^q(\Omloc),\tau_q)$ 
 the set of all probability measures on 
  $(\Omloc, \mathcal{B}(\Omloc))$
  with finite $q$-th moment\footnote{We write 
  	$\mathcal{M}_1(\Omloc)$ for $\mathcal{M}_1^0(\Omloc)$, i.e., for the set of all probability measures on $(\Omloc, \mathcal{B}(\Omloc))$ with no moment restrictions.} endowed with the topology $\tau_{q}$ 
characterized by 
\begin{equation}\label{eq_convergence_topology_1}
	\mu_n \xrightarrow{\tau_q} \mu \text{ for } n \rightarrow \infty ~\Leftrightarrow~ \lim_{n \rightarrow \infty} \int_{\Omloc} g \,\D \mu_n = \int_{\Omloc} g \,\D \mu \text{ for all } g \in C_q(\Omloc, \R).
\end{equation}
In particular, for $q=0$, the topology $\tau_0$ coincides with the topology of weak convergence, whereas for $q\geq 1$, $\tau_q$ coincides with the topology induced by the $q$-Wasserstein metric $\operatorname{d}_{W_q}(\cdot, \cdot)$. Recall that for $q \in \N$ we have for two probability measures $\mu_1, \mu_2 \in \mathcal{M}_1^q(\Omloc)$ with finite $q$-moment that their Wasserstein-distance of order $q$ is defined as
\[
\operatorname{d}_{W_q} \left(\mu_1, \mu_2 \right) := \inf_{\pi \in \Pi(\mu_1,\mu_2)} \left(\int_{\Omloc\times \Omloc} \|x-y\|^q \pi(\D x, \D y)\right)^{1/q},
\] 
where $\Pi(\mu_1,\mu_2) \subset \mathcal{M}_1^q(\Omloc \times \Omloc)$ denotes the set of all joint distributions of $\mu_1$ and $\mu_2$, i.e., the set of all probability measures on $\Omloc \times \Omloc$ with first marginal $\mu_1$ and second marginal $\mu_2$, compare for more details, e.g., \cite{villani2009optimal}.

We define an ambiguity set of probability measures by first fixing for $t=0$ a set  $\mathcal{P}_0 \subseteq \mathcal{M}_1(\Omloc)$, whereas for all $t=1,\dots,T-1$ we fix a correspondence $\mathcal{P}_t: \Omega^t \twoheadrightarrow\mathcal{M}_1(\Omloc)$.
%
We impose the following conditions on the correspondences $(\mathcal{P}_t)_{t=0,\dots,T-1}$. 
\begin{asu}\label{asu_P}
Let $p \in \N_0$ be the integer from Assumption~\ref{asu_psi}.
Then, for every $t\in \{1,\dots,T-1\}$ we assume the following.
\begin{itemize}
\item[(i)]
The correspondence $\Omega^t \ni 
\omega^t \twoheadrightarrow \mathcal{P}_t(\omega^t) \subseteq \left( \mathcal{M}_1^p(\Omloc), \tau_{p}\right)$ is compact-valued, continuous\footnote{\emph{Continuous} here refers to both lower-hemicontinuous and upper-hemicontinuous, see, e.g., \cite{Aliprantis}.}, and non-empty. Moreover, if $p=0$, we additionally assume that $\mathcal{P}_t(\omega^t) \subseteq \mathcal{M}_1^1(\Omloc)$ for all $\omega^t \in \Omega^t$.
\item[(ii)] There exists some $C_{\mathcal{P},t} \geq 1$ such that for all $\omega^{t}:= (\omega_1,\dots,\omega_t)\in \Omega^{t}$ and for all $\PP \in \mathcal{P}_t\left(\omega^{t}\right)$ we have
\begin{equation}\label{eq_asu_P_ineq_1}
\int_{\Omloc} \|x\|^p \PP(\D x) \leq C_{\mathcal{P},t}\left(1+\sum_{i=1}^t\|\omega_i\|^p\right).
\end{equation}
\item[(iii)]
There exists some $L_{\mathcal{P},t}\geq 0$ such that for all $\omega^{t}:= (\omega_1,\dots,\omega_t) \in \Omega^t$, $\widetilde{\omega}^{t}:= (\widetilde{\omega}_1,\dots,\widetilde{\omega}_t) \in \Omega^t$ and for all $\PP \in \mathcal{P}_t\left(\omega^{t}\right)$ there exists some $\widetilde{\PP} \in \mathcal{P}_t\left(\widetilde{\omega}^{t}\right)$ satisfying
\begin{equation}\label{eq_condition_P}
\operatorname{d}_{W_1}(\PP, \widetilde{\PP} ) \leq L_{\mathcal{P},t} \left( \sum_{i=1}^t \|\omega_i - \widetilde{\omega}_i \|  \right).
\end{equation}
\item[(iv)] The set $\mathcal{P}_0 \subseteq \left( \mathcal{M}_1^p(\Omloc), \tau_{p}\right)$ is non-empty, compact, and there exists $C_{\mathcal{P},0} \geq 1$ such that
\begin{equation}\label{eq_asu_P_ineq_2}
\int_{\Omloc} \|x\|^p \PP(\D x) \leq C_{\mathcal{P},0}
\end{equation}
for all $\PP \in \mathcal{P}_0$.
 \end{itemize}
\end{asu}
%
%

%
We can now define a probability measure $\PP:= \PP_0 \otimes \dots \otimes \PP_{T-1} \in \mathcal{M}_1(\Omega)$ by 
\begin{equation}\label{eq_P_integrals}
\PP(B) =\PP_0 \otimes \dots \otimes \PP_{T-1}(B):=\int_{\Omloc} \dots \int_{\Omloc} \one_B(\omega_1, \dots,\omega_T) \PP_{T-1}(\omega_1,\dots,\omega_{T-1}; \D \omega_T) \dots \PP_0(\D\omega_1),
\end{equation}
for $B \in \mathcal{F}$.
This allows us to define the ambiguity set $\mathfrak{P} \subseteq \mathcal{M}_1(\Omega)$ of admissible probability measures on $\Omega$ by 
\begin{align*}
\mathfrak{P}:= \bigg\{\PP_0 \otimes \dots \otimes \PP_{T-1}~\bigg|~ 
\text{ for all $t=0,\dots,T-1$: }~&\PP_t: \Omega^t \rightarrow \mathcal{M}_1(\Omloc) \text{ Borel - meas.},\  
\PP_t(\cdot) \in \mathcal{P}_t(\cdot)\bigg\}.
\end{align*}
\begin{rem}
{The Kuratowski--Ryll--Nardzewski Theorem} (e.g., \cite[Theorem 18.13]{Aliprantis}) provides for each $t\in \{1,\dots,T-1\}$ the existence of a measurable kernel $\PP_t:\Omega^t \rightarrow \mathcal{M}_1^p(\Omloc)$ such that $\PP_t(\omega^t) \in \mathcal{P}_t(\omega^t)$ for all $\omega^t \in \Omega^t$.
\end{rem}

Our goal is to analyze the following 
stochastic optimal control problem under model uncertainty which consists in solving the max-min problem 
\begin{equation}\label{eq_optimization_problem}
\mathcal{V}:=\sup_{\mathbf{a} \in \mathfrak{A}} \inf_{\PP \in \mathfrak{P}} \E_{\PP} \left[\Psi(a_0,\dots,a_{T-1})\right], 
\end{equation}
i.e., our goal is to choose a control that maximizes the expected valued of $\Psi$ under the \emph{worst case} probability measure.
{
\begin{rem}
Assumption~\ref{asu_A}~(i), Assumption~\ref{asu_psi}, and Assumption~\ref{asu_P}~(i),~(ii),~(iv) are natural conditions imposed on the set of controls, the utility function, and the probability measures characterizing the model uncertainty 
to derive the existence of a
local (i.e., one-step) worst-case measure in \eqref{eq_defn_J_t} and  local optimal control in \eqref{eq_defn_Psi_t} at each time $t$ in the dynamic programming procedure by applying Berge's maximum theorem \cite{berge}, provided that the function $\Psi_{t+1}$ in \eqref{eq_defn_J_t} possesses the same regularity as $\Psi=\Psi_{T}$. Now, to back-propagate the regularity of $\Psi$, we introduce Assumption~\ref{asu_P}~(iii) (and similarly Assumption~\ref{asu_A}~(ii)) which, to the best of our knowledge, has not appeared in the literature yet. We show in Theorem~\ref{thm_wasserstein_ambiguity}, Theorem~\ref{thm_param_assumptions}, and  Proposition~\ref{prop_controls} that these assumptions are \textit{naturally satisfied} for relevant examples of sets of probability measures and controls.  For a detailed discussion of our results and assumptions compared to the existing literature, we also refer to Section~\ref{sec:Comparison}.
\end{rem}
}
\section{Dynamic Programming Principle}\label{sec_dynamic_programming}
In this section we introduce the dynamic programming procedure allowing us to solve \eqref{eq_optimization_problem} 
 by solving recursively a sequence of one-step optimization problems. 

To that end we first set
\[
\Psi_T :\equiv \Psi,
\]
and then define recursively  for all $t=T-1,\dots,1$  the following quantities:\footnote{For any $t\in \{1,\dots,T-1\}$, $\omega^t \in \Omega^t$, ${\omega}_{\loc} \in \Omloc$ we write $\omega^t \otimes_t {\omega}_{\loc} :=(\omega^t ,{\omega}_{\loc}) \in \Omega^{t+1}$.}
\begin{equation}\label{eq_defn_J_t}
\begin{aligned}
\Omega^t \times \mathcal{A}^{t+1} \ni \left(\omega^t,a^{t+1}\right) \mapsto  J_t(\omega^t,a^{t+1}):&= \inf_{\PP \in \mathcal{P}_t(\omega^t)} \E_{\PP}\left[\Psi_{t+1}\left(\omega^t \otimes_t \cdot, a^{t+1}\right)\right]
\\
&= \inf_{\PP \in \mathcal{P}_t(\omega^t)} \int_{\Omloc}\Psi_{t+1}\left((\omega^t,{\omega}_{\loc}), a^{t+1}\right) \PP(\D {\omega}_{\loc}),
\end{aligned}
\end{equation}
as well as 
\begin{equation}\label{eq_defn_Psi_t}
\begin{aligned}
\Omega^t \times \mathcal{A}^{t} \ni \left(\omega^t,a^{t}\right) \mapsto &\Psi_t(\omega^t,a^t):= \sup_{\widetilde{a} \in \mathcal{A}_t} J_t\left(\omega^t, (a^t,\widetilde{a})\right).
\end{aligned}
\end{equation}
Moreover, for $t=0$ we define
\begin{equation}\label{eq_defn_Psi_0}
\begin{aligned}
\mathcal{A}_0\ni a \mapsto J_0(a)&:=\inf_{\PP \in \mathcal{P}_0} \E_{\PP}\left[\Psi_{1}\left(\cdot, a\right)\right],\\
\Psi_0 &:= \sup_{\widetilde{a} \in \mathcal{A}_0} J_0(\widetilde{a}).
\end{aligned}
\end{equation}
In the following theorem we
establish the existence of an optimal control and a worst-case measure for our 
stochastic optimal control problem under model uncertainty \eqref{eq_optimization_problem} as well as a dynamic programming principle.
\begin{thm}\label{thm_main_result}
Suppose that Assumption~\ref{asu_A}, Assumption~\ref{asu_psi}, and Assumption~\ref{asu_P} are fulfilled. Then the following holds.
\begin{itemize}

\item[(i)]
Let $t\in \{1,\dots,T-1\}$. Then, there exists a measurable selector
\begin{equation}\label{eq_thm_iia_def}
\Omega^t \times \mathcal{A}^t \ni (\omega^t,a^t) \mapsto \widetilde{a}_t^* \left(\omega^t, a^t\right) \in \mathcal{A}_t(\omega^t)
\end{equation}
such that for all $(\omega^t,a^t) \in \Omega^t \times \mathcal{A}^t$
\begin{equation}\label{eq_thm_iia}
\Psi_t(\omega^t,a^t) =  J_t\left(\omega^t,\,\left(a^t,\widetilde{a}_t^* \left(\omega^t, a^t\right)\right)\,\right).
\end{equation}
In addition, there exists $a_0^* \in \mathcal{A}_0$ such that 
\begin{equation}\label{eq_thm_iia_t0}
\Psi_0 =  J_0(a_0^*).
\end{equation}
Moreover, for all $t\in \{1,\dots,T-1\}$ there exists a measurable selector
\begin{equation}\label{eq_thm_iip_def}
\Omega^t \times \mathcal{A}^{t+1} \ni (\omega^t,a^{t+1}) \mapsto  \widetilde{\PP}_t^*(\omega^t,a^{t+1}) \in  \mathcal{P}_t(\omega^t)
\end{equation}
such that for all $(\omega^t,a^{t+1})\in \Omega^t \times \mathcal{A}^{t+1} $ we have
\begin{equation}\label{eq_thm_iiP}
\E_{\widetilde{\PP}_t^*(\omega^t,a^{t+1})} \left[\Psi_{t+1} \left(\omega^t \otimes_t \cdot, a^{t+1}\right)\right]=\inf_{\PP \in \mathcal{P}_t(\omega^t)} \E_{\PP} \left[\Psi_{t+1} \left(\omega^t \otimes_t \cdot, a^{t+1}\right)\right].
\end{equation}
In addition, there exists a measurable selector $\mathcal{A}_0 \ni a \mapsto \widetilde{\PP}_0^*(a) \in \mathcal{P}_0$ such that for all $a \in \mathcal{A}_0$ we have 
\begin{equation}\label{eq_thm_iiP_t0}
\E_{\widetilde{\PP}_0^*(a)} \left[\Psi_{1} \left(\cdot, a\right)\right]=\inf_{\PP \in \mathcal{P}_0} \E_{\PP} \left[\Psi_{1} \left(\cdot, a\right)\right].
\end{equation}
\item[(ii)]
For every $t\in \{1,\dots,T-1\}$ let $\widetilde{a}_t^*$ be the measurable selector from \eqref{eq_thm_iia_def} and let $a_0^*$ be defined in \eqref{eq_thm_iia_t0}, and define
\[
\Omega^t \ni \omega^t =(\omega_1,\dots,\omega_{t}) \mapsto a_t^*(\omega^t):= \widetilde{a}_t^*\left(\omega^t,~\left(a_0^*,\dots,a_{t-1}^*(\omega_1,\dots,\omega_{t-1}\right)\right) \in \mathcal{A}_t(\omega^t).
\]
Moreover, for every $t \in \{1,\dots,T-1\}$ let $\widetilde{\PP}_t^*$ denote the measurable selector from \eqref{eq_thm_iip_def} and define
\[
\Omega^t \ni \omega^t \mapsto {\PP}_t^*(\omega^t) := \widetilde{\PP}_t^*(\omega^t,(a_s^*(\omega^s))_{s=0,\dots,t})) \in  \mathcal{P}_t(\omega^t).
\]
In addition, let $\widetilde{\PP}_0^*$ denote the measurable selector from \eqref{eq_thm_iiP_t0} and define $\PP_0^*:= \widetilde{\PP}_0^*(a_0^*)$.
Then for $\ab^*:=(a_t^*)_{t=0,\dots,T-1} \in \mathfrak{A}$ and $\PP^*:=\PP_0^* \otimes \cdots \otimes \PP_{T-1}^* \in \mathfrak{P}$ we have
\begin{equation}\label{eq_optimality_equation_1}
 \sup_{a\in \mathfrak{A}}\inf_{\PP \in \mathfrak{P}} \E_{\PP}\left[\Psi(a_0,\dots,a_{T-1})\right] 
 = \inf_{\PP \in \mathfrak{P}} \E_{\PP}\left[\Psi(a_0^*,\dots,a_{T-1}^*) \right]
 = \E_{\PP^*}\left[\Psi \left(a_0^*,\dots,a_{T-1}^* \right)\right]
 =\Psi_0.
\end{equation}
\end{itemize}
\end{thm}
\section{Modeling Ambiguity { and Controls}}\label{sec_ambiguity}
{
The goal of this section is to demonstrate that our general conditions imposed on the ambiguity set of probability measures (cf.\ Assumption~\ref{asu_P}) and on the set of controls (cf.\ Assumption~\ref{asu_A}) are \textit{naturally satisfied} for relevant examples. More precisely, 
in Section~\ref{sec_ambiguity_wasserstein} we show that  ambiguity sets  modeled as Wasserstein-balls around  path-dependent reference probability measures and  path dependent radii satisfy Assumption~\ref{asu_P}.
In Section~\ref{sec_ambiguity_parametric} we show in the context of parameter uncertainty that  ambiguity sets of probability measures consisting of all parametrized probability measures whose parameter lie within a ball with path dependent radii  around  a path dependent reference parameter also satisfy Assumption~\ref{asu_P}. 
Moreover, we provide in Section~~\ref{sec_controls} several examples for sets of controls which satisfy Assumption~\ref{asu_A}.}
\subsection{Modeling ambiguity with the Wasserstein distance}\label{sec_ambiguity_wasserstein}

The following result shows that the important canonical example of an ambiguity set given by a Wasserstein-ball around a path-dependent reference probability measure and a path dependent radius fulfills the assumption formulated in Assumption~\ref{asu_P}.
This can be seen as a purely data-driven approach to describe model uncertainty whenever the corresponding reference measure is estimated solely using data (e.g.\ using the empirical measure).
\begin{thm} \label{thm_wasserstein_ambiguity}
For each $t \in \{1,\dots,T-1\}$, $q \in \N$, let $\Omega^t \ni \omega^t \mapsto \widehat{\PP}_t(\omega^t) \in \left(\mathcal{M}_1^q(\Omloc), \tau_q\right)$ be $L_{\widehat{\mathbb{P}},t}$-Lipschitz continuous\footnote{For any $t \in \{1,\dots,T\}$ and any metric space $(\mathcal{X},d_{\mathcal{X}})$ we call a function $f:\Omega^t \to \mathcal{X}$ Lipschitz continuous with respect to a Lipschitz-constant $L \geq 0$ if for all $\omega^t=(\omega_1,\dots,\omega_t), \widetilde{\omega}^t=(\widetilde{\omega}_1,\dots,\widetilde{\omega}_t) \in \Omega^t$ we have $d_{\mathcal{X}}(f(\omega^t),f(\widetilde{\omega}^t))\leq L \cdot \left(\sum_{i=1}^t \|\omega_i- \widetilde{\omega}_i\|\right).$}
with respect to some $L_{\widehat{\mathbb{P}},t} \geq 0$,
{let $\overline{\varepsilon}_t>0$, and let $\varepsilon_t:\Omega^t \mapsto [0,\overline{\varepsilon}_t]$ be $L_{\varepsilon,t}$-Lipschitz continuous with respect to some $L_{\varepsilon,t}\geq 0$.
} Moreover, consider some $\widehat{\PP}_0 \in \mathcal{M}_1^q(\Omloc)$ and some $\varepsilon_0\geq 0$. Furthermore, assume that $\Omloc$ is convex.

Then, the ambiguity set defined by  for $t \in \{1,\dots,T-1\}$ by
\begin{align*}
\Omega^t \ni \omega^t \twoheadrightarrow \mathcal{P}_t\left({\omega}^{t}\right)&:=\mathcal{B}_{\varepsilon_t(\omega^t)}^{(q)} \left(\widehat{\PP}_t(\omega^t)\right):= \left\{\PP\in \mathcal{M}_1^q(\Omloc)~ \middle|~ \operatorname{d}_{W_q}\left(\widehat{\PP}_t(\omega^t),~\PP \right) \leq  {\varepsilon_t(\omega^t)} \right\}, 
\\
 \mathcal{P}_0&:= \mathcal{B}_{\varepsilon_0}^{(q)} \left(\widehat{\PP}_0\right):= \left\{\PP\in \mathcal{M}_1^q(\Omloc)~ \middle|~ \operatorname{d}_{W_q}\left(\widehat{\PP}_0,~\PP \right) \leq  \varepsilon_0 \right\}, 
\end{align*}
 satisfies Assumption~\ref{asu_P} if $p < q$.
\end{thm}
{
\begin{rem}\label{rem:adaptive-robust-control}
The above framework presented in Theorem~\ref{thm_wasserstein_ambiguity} for modeling the ambiguity sets allows to formulate the corresponding stochastic optimal control problem under model uncertainty \eqref{eq_optimization_problem} as a non-parametric adaptive robust control problem similar to \cite{bayraktar2023nonparametric, bayraktar2022data}.
More precisely, let here $(Z_t)_{t=1}^T$ be the canonical process on $(\Omega, \mathcal{F})$, i.e.\ $Z_t(\omega^T)=\omega_t$ for all $t \in \{1,\dots,T\}$, $\omega^T=(\omega_1, \dots, \omega_T) \in \Omega$ representing the process which is observable to the agent, but whose law is not known. Moreover,  let here the (random) state process $(X^\mathbf{a}_t)_{t=1}^T$ be characterized by its controlled dynamics defined for every control $\mathbf{a}=(a_t)_{t=0}^{T-1}\in \mathfrak{A}$ by $X^\mathbf{a}_{0}=x_0\in \R^n$ and 
\begin{equation*}
	X^\mathbf{a}_{t+1}=S(X^\mathbf{a}_t,a_t, Z_{t+1}),
\end{equation*}
where $S:\R^n \times A \times \Omloc \to \R^n$ denotes the (deterministic) transition function and $(a_t)_{t=0}^{T-1}\in \mathfrak{A}$ is an $\mathbb{F}$-adapted control taking values in (for simplicity here) some fixed compact set $A\subseteq \R^m$. Moreover, let $l:\R^n\to \R$ be some loss function. Then we see that the following stochastic control problem
\begin{equation*}
\sup_{\mathbf{a} \in \mathfrak{A}} \inf_{\PP \in \mathfrak{P}} \E_{\PP} \left[l(X^\mathbf{a}_T)\right]
\end{equation*}
fits into our stochastic optimal control problem \eqref{eq_optimization_problem} and satisfies our assumptions imposed for our dynamic programming principle in Theorem~\ref{thm_main_result}, provided that
 $S:\R^n \times A \times \Omloc \to \R^n$ and $l:\R^n\to \R$ satisfy some regularity conditions, for example $S$ being Lipschitz-continuous and $l$ being $\alpha$-H\"older continuous  for some $\alpha \in (0,1]$ with polynomial growth at most $p \in \N_0$, as well as the Wasserstein balls $\Omega^t\ni\omega^t \to \mathcal{P}_t(\omega^t):=\mathcal{B}_{\varepsilon_t(\omega^t)}^{(q)} \big(\widehat{\PP}_t(\omega^t)\big)$, $t=0,\dots,T-1$ satisfying the assumptions of Theorem~\ref{thm_wasserstein_ambiguity}. By choosing appropriately the radii $\varepsilon_t$, $t=0,\dots,T-1$, of the corresponding Wasserstein-ball around the estimated law of each 
 $Z_t$, $t=1,\dots,T$,
 one  can incorporate robust control while learning and reducing uncertainty from data. For example, following \cite{bayraktar2023nonparametric, bayraktar2022data} based on \cite{fournier2015rate}, an appropriate choice for the radius $\varepsilon_t$ at each time $t$ is a multiple of $\frac{1}{t+t_0}$, where here $t_0$ denotes the length in time of historical data of the observation process $(Z_{s})_{s=-t_0+1}^0$ (before the start of the optimization problem at time $0$). We refer to \cite{bayraktar2023nonparametric, bayraktar2022data} for a detailed discussion on nonparametric adaptive robust control problems under model uncertainty using Wasserstein balls for the ambiguity sets, as well as for the appropriate choice of the corresponding radii $\varepsilon_t$, $t=0,\dots,T-1$. We highlight that, compared to \cite{bayraktar2023nonparametric, bayraktar2022data}, our framework allows the radii $\varepsilon_t:\Omega^t \to [0,\overline{\varepsilon}_t]$, $t=1,\dots,T-1$, also to depend on the path of the observations $(Z_{s})_{s=1}^t$ and does not need to be deterministic.
\end{rem}
}
%
%
\subsection{Modeling ambiguity in parametric models}\label{sec_ambiguity_parametric}
In this section we show that the general conditions imposed on the ambiguity sets of probability measures in Assumption~\ref{asu_P} from Section~\ref{sec_setting} also 
allows 
to consider parametric families of distributions as ambiguity sets 
with corresponding ambiguity  described 
 by time- and path-dependent parameter sets. 

To this end,  for each $t\in \{0,1,\dots,T-1\}$  let $\left(\PP_\theta \right)_{\theta \in \Theta_t} \subseteq \mathcal{M}_1(\Omloc)$ be a family of probability measures parameterized by $\theta \in \Theta_t$ for some fixed set $\Theta_t \in R^{D_t}$, and let $p \in \N_0$ be the integer from Assumption~\ref{asu_psi}. Then, we impose the following assumptions. 
\begin{enumerate}[label=\textbf{(A\arabic*)}]
\item\label{eq_a1}
For each $t \in \{0,1,\dots,T-1\}$ let $\Theta_t \in \R^{D_t}$ for some $D_t \in \N$ be non-empty, convex, and closed. 
\item\label{eq_a2} For each $t \in \{0,1,\dots,T-1\}$ let the map
\begin{align*}
\R^{D_t} \supseteq \Theta_t &\rightarrow \left(\mathcal{M}_1^{\max\{1,p\}}(\Omloc), \tau_{\max\{1,p\}} \right)\\
\theta &\mapsto \PP_{\theta}(\D x)
\end{align*}
be Lipschitz continuous with respect to some constant $ L_{P_\theta,t}\geq 0$, i.e.\  $\operatorname{d}_{W_{\max\{1,p\}}}\left(\PP_{\theta_1},\PP_{\theta_2}\right) \leq L_{P_\theta,t}   \cdot \|\theta_1- \theta_2\| $ for all $\theta_1,\theta_2 \in \Theta_t$.
\item\label{eq_a3}  For each $t\in \{1,\dots,T-1\}$ let the map
\begin{align*}
\widehat{\theta}_t: \Omega^t & \rightarrow \Theta_t \\
\omega^t&\mapsto \widehat{\theta}_t (\omega^t) 
\end{align*}
be Lipschitz continuous with respect to some constant $L_{\widehat{\theta},t}\geq 0$.
\item \label{eq_a4}  For each $t\in \{1,\dots,T-1\}$,
{ let $\overline{\varepsilon}_t>0$, let $\varepsilon_t:\Omega^t \mapsto [0,\overline{\varepsilon}_t]$ be $L_{\varepsilon,t}$-Lipschitz continuous with respect to some $L_{\varepsilon,t}\geq 0$, 
}
and define 
\[
\Omega^t \ni \omega^t \twoheadrightarrow\mathcal{P}_{t}(\omega^t):= \left\{\PP_{\theta}(\D x) \in \mathcal{M}_1^{\max\{1,p\}}(\Omloc)~\middle|~\theta \in \Theta_t \text{ with } \| \theta - \widehat{\theta}_t(\omega^t)\| \leq {\varepsilon_t(\omega^t)} \right\}.
\]
\item\label{eq_a5} For $t=0$ let $\widehat{\theta}_0 \in \Theta_0$, let $\varepsilon_0\geq 0$, and define 
\[
\mathcal{P}_{0}:= \left\{\PP_{\theta}(\D x) \in \mathcal{M}_1^{\max\{1,p\}}(\Omloc)~\middle|~\theta \in \Theta_0 \text{ with } \| \theta - \widehat{\theta}_0\| \leq \varepsilon_0 \right\}.
\]
\end{enumerate}
{
This approach \ref{eq_a1}--\ref{eq_a5} for describing model uncertainty can be seen as \textit{semi data-driven}, whenever the reference parameters $\widehat{\theta}_t:\Omega^t\to \Theta_t$, $t=0,\dots T-1$, are estimated using data. Indeed in this approach, some model assumptions are imposed by restricting the ambiguity sets of probability measures to be parameterized,
 but on the other hand  one only considers parameters which are close enough to the  estimated parameter $\widehat{\theta}_t:\Omega^t\to \Theta_t$, $t=0,\dots T-1$.
}
 %
\begin{thm}\label{thm_param_assumptions}
For each $t\in \{0,1,\dots,T-1\}$  let $\left(\PP_\theta \right)_{\theta \in \Theta_t} \subseteq \mathcal{M}_1(\Omloc)$ be a family of probability measures parameterized by $\theta \in \Theta_t$, and let $p \in \N_0$ be the integer from Assumption~\ref{asu_psi}. Moreover, assume that assumptions \ref{eq_a1} - \ref{eq_a5} holds. Then Assumption~\ref{asu_P} is satisfied.
\end{thm}

{
\begin{rem}
Both the Wasserstein-ball approach presented in Section~\ref{sec_ambiguity_wasserstein} and the parametric approach described in \ref{eq_a1}--\ref{eq_a5} in this section can be seen as the natural generalizations of the framework developed in \cite[Section~3.1\&Section~3.2]{neufeld2023mdp}, where here our reference measures $\widehat{\PP}_t:\Omega^t\to\mathcal{M}_1(\Omloc)$ and reference estimators $\widehat{\theta}_t:\Omega^t\to \Theta_t$, as well as  the corresponding radii of the balls $\varepsilon_t:\Omega^t\to[0,\overline{\varepsilon}_t]$ at each time $t$ are allowed to be \text{path-dependent}, compared to the corresponding (single, time-independent) reference measure and estimator in \cite{neufeld2023mdp} which was only allowed to be depending on the current state, and where the radius had to be a constant.

Moreover, from a technical point of view, the most difficult part in Theorem~\ref{thm_wasserstein_ambiguity} and Theorem~\ref{thm_param_assumptions} is to verify that each of these corresponding settings satisfy Assumption~\ref{asu_P}~(iii). This condition was not necessary in \cite{neufeld2023mdp} due to the different nature of the Bellman principle in the infinite-horizon case \cite{neufeld2023mdp}. We refer to Section~\ref{sec:Comparison} for a more detailed discussion on the comparison between this work and~\cite{neufeld2023mdp}.
\end{rem}
}

Next, we present 
{ the following parametric families} of distributions fulfilling the above mentioned assumptions \ref{eq_a1} - \ref{eq_a5} and therefore also, according to Theorem~\ref{thm_param_assumptions}, Assumption~\ref{asu_P}.
{
\begin{prop}[Normal distributions]\label{prop_normal_family}
Let $\Omloc := \R^d$ and for each $t \in \{0,1,\dots,T-1\}$ let $\Theta_t:= \R^d \times[0,\infty)^d$. For each $\theta:=(\boldsymbol{\mu}, \boldsymbol{\sigma})
:=\big((\mu_1,\dots,\mu_d),(\sigma_1,\dots,\sigma_d)\big)
\in \R^d \times[0,\infty)^d$, we define\footnote{For any $\boldsymbol{\sigma}:=(\sigma_1,\dots,\sigma_d) \in [0,\infty)^d$ we denote by $\mbox{diag}(\boldsymbol{\sigma})$ the diagonal matrix $M$ in $\R^{d\times d}$ with diagonal entries~$\boldsymbol{\sigma}$, i.e. $M_{i,i}=\sigma_i$ for each $i \in \{1,\dots,d\}$.}
 $\PP_{\theta}=\PP_{(\boldsymbol{\mu}, \boldsymbol{\sigma})}:=\mathcal{N}(\boldsymbol{\mu},\mbox{diag}(\boldsymbol{\sigma}^2))$,
i.e., as an multivariate normal\footnote{
	We say that a $d$-dimensional random variable $X$ is $d$-dimensional multivariate normally distributed with mean $\boldsymbol{\mu} \in \R^d$ and covariance matrix $\Sigma \in \R^{D\times D}$ which is symmetric and positive semidefinite if the characteristic function of $X$ is of the form $\R^d \ni u \mapsto \varphi_X(u):= \exp\left(i u^T \boldsymbol{\mu}- \tfrac{1}{2} u^T \Sigma u\right)$, compare, e.g., \cite[p. 124]{gut2009multivariate}. We write $X \sim \mathcal{N}_d(\boldsymbol{\mu}, \Sigma)$.} distribution with mean $\boldsymbol{\mu}$ and covariance matrix $\mbox{diag}(\boldsymbol{\sigma}^2)$ with  $\boldsymbol{\sigma}^2:=(\sigma^2_1,\dots,\sigma^2_d)$.
	Moreover, let $\widehat{\boldsymbol{\mu}}_0 \in \R^d$ and define for each 
	$t\in \{1,\dots,T-1\}$
	\begin{equation}\label{eq_defn_mu_hat}
		\begin{aligned}
			\widehat{\boldsymbol{\mu}}_t:\Omega^t &\rightarrow \R^d\\
		(\omega_1,\dots,\omega_t) &\mapsto \overline{\omega^t}:=\frac{1}{t}\sum_{s=1}^t \omega_s.
		\end{aligned}
	\end{equation}
Furthermore, let $\widehat{\boldsymbol{\sigma}}_0, \widehat{\boldsymbol{\sigma}}_1 \in [0,\infty)^d$ and define  for each $t\in \{2,\dots,T-1\}$
	$\widehat{\boldsymbol{\sigma}}_t=(\widehat{\sigma}_{t,1},\dots,\widehat{\sigma}_{t,d}):\Omega^t \rightarrow [0,\infty)^d$ by setting for each  $j=1,\dots,d$
	\begin{equation}\label{eq_defn_sigma_hat}
		\begin{aligned}
		\widehat{\sigma}_{t,j}:\Omega^t &\rightarrow [0,\infty)\\
			(\omega_1,\dots,\omega_t) &\mapsto \frac{\sqrt{\pi}}{\sqrt{2}}\frac{\sqrt{t}}{\sqrt{t-1}} \frac{1}{t}\sum_{s=1}^t \big|\omega_{s,j}- \overline{\omega_j}\big|,
		\end{aligned}
	\end{equation}
where for $s=1,\dots,t$ and $j=1,\dots, d$ we write 
$\omega_s=(\omega_{s,1}, \dots, \omega_{s,d})\in \R^d$ and 
$\overline{\omega^t_j}:=\frac{1}{t}\sum\limits_{s=1}^t \omega_{s,j}\in \R$.
Moreover, for each $t\in \{1,\dots,T-1\}$ let $\overline{\varepsilon}_t>0$ and let $\varepsilon_t:\Omega^t \mapsto [0,\overline{\varepsilon}_t]$ be $L_{\varepsilon,t}$-Lipschitz continuous with respect to some $L_{\varepsilon,t}\geq 0$ and let $\varepsilon_0\geq 0$.
Then, for any 
 $p\in \{0,1,2\}$, the set-valued maps
\begin{align*}
	\Omega^t \ni \omega^t \twoheadrightarrow \mathcal{P}_t(\omega^t)&:=\left\{\PP_{(\boldsymbol{\mu}, \boldsymbol{\sigma})}(\D x)
	~\middle|~ (\boldsymbol{\mu}, \boldsymbol{\sigma}) \in \R^d\times [0,\infty)^d \text{ with } \| (\boldsymbol{\mu}, \boldsymbol{\sigma}) -(\widehat{\boldsymbol{\mu}}_t(\omega^t),\widehat{\boldsymbol{\sigma}}_t(\omega^t))\| \leq \varepsilon_t(\omega^t)
	\right\}\\
	\mathcal{P}_0&:= \left\{\PP_{(\boldsymbol{\mu}, \boldsymbol{\sigma})}(\D x)
	~\middle|~ (\boldsymbol{\mu}, \boldsymbol{\sigma}) \in \R^d\times [0,\infty)^d \text{ with } \| (\boldsymbol{\mu}, \boldsymbol{\sigma}) -(\widehat{\boldsymbol{\mu}}_0,\widehat{\boldsymbol{\sigma}}_0)\| \leq \varepsilon_0 
	\right\}\\
\end{align*}
satisfy  \ref{eq_a1} - \ref{eq_a5}.
\end{prop}
\begin{rem}\label{rem:estimator-normal}
	To motivate $\eqref{eq_defn_mu_hat}$ and $\eqref{eq_defn_sigma_hat}$, note that for any i.i.d.\ samples  of normally distributed random variables $Y_1,\dots, Y_t \sim \mathcal{N}_1(\mu,\sigma^2)$,  
	\begin{equation*}
		\widehat{\mu}:=\overline{Y}:=\frac{1}{t}\sum_{s=1}^t Y_s 
		\qquad \mbox {and }\qquad 
		\widehat{\sigma}:= \frac{\sqrt{\pi}}{\sqrt{2}}\frac{\sqrt{t}}{\sqrt{t-1}} \frac{1}{t}\sum_{s=1}^t \big|Y_s- \overline{Y}\big|
	\end{equation*}
	are unbiased estimators of $\mu$ and $\sigma$, respectively. Indeed,  this clearly holds for $\widehat{\mu}$. To see this also for 	$\widehat{\sigma}$, note that  for each $s=1,\dots, t$ we have $Y_s- \overline{Y} \sim \mathcal{N}_1(0,\frac{t-1}{t}\sigma^2)$, which implies that $\E[|Y_s- \overline{Y}|]=\frac{\sqrt{2}}{\sqrt{\pi}}\frac{\sqrt{t-1}}{\sqrt{t}}\sigma$. Therefore, we conclude that indeed $\E[\widehat{\sigma}]=\sigma$.
	

\end{rem}
}

\begin{prop}[Exponential distributions]\label{prop_exponential_family}
Let $\Omloc := [0,\infty)$ and for each $t \in \{0,1,\dots,T-1\}$ let $\Theta_t:=  [0,\infty)$. For each $\theta \in \Theta_t$ define $\PP_{\theta}$ by 
\[
\PP_{\theta}: = \begin{cases}
\operatorname{Exp}\left(\tfrac{1}{\theta}\right) &\text{ if } \theta >0, \\
\delta_{\{0\}} 
&\text{ if } \theta =0,
\end{cases}
\]
i.e., as an exponential distribution with rate parameter $\tfrac{1}{\theta}$ if $\theta>0$ and as the Dirac measure at $0$ if $\theta = 0$.
Moreover, we let $\widehat{\theta}_0 \in \Theta_0$ and define for each $t\in \{1,\dots,T-1\}$
\begin{equation}\label{eq_defn_theta_hat}
\begin{aligned}
\widehat{\theta}_t:\Omega^t &\rightarrow \Theta_t\\
(\omega_1,\dots,\omega_t) &\mapsto \frac{1}{t}\sum_{i=1}^t \omega_i.
\end{aligned}
\end{equation}
{
	Moreover, for each $t\in \{1,\dots,T-1\}$ let $\overline{\varepsilon}_t>0$ and let $\varepsilon_t:\Omega^t \mapsto [0,\overline{\varepsilon}_t]$ be $L_{\varepsilon,t}$-Lipschitz continuous with respect to  some $L_{\varepsilon,t}\geq 0$ and let $\varepsilon_0\geq 0$.
}
Then, 
for any $p\in \N_0$ the set-valued maps
\begin{align*}
\Omega^t \ni \omega^t \twoheadrightarrow \mathcal{P}_t(\omega^t)&:=\left\{\PP_\theta(\D x)\in \mathcal{M}_1^{\max\{1,p\}}(\Omloc)~\middle|~ \theta \in \Theta_t \text{ with } \| \theta-\widehat{\theta}_t(\omega^t)\| \leq {\varepsilon_t(\omega^t)}\right\}\\
\mathcal{P}_0&:= \left\{\PP_\theta(\D x)\in \mathcal{M}_1^{\max\{1,p\}}(\Omloc)~\middle|~ \theta \in \Theta_0 \text{ with } \| \theta-\widehat{\theta}_0\| \leq \varepsilon_0\right\}
\end{align*}
satisfy  \ref{eq_a1} - \ref{eq_a5}.
\end{prop}
\begin{rem}~
\begin{itemize}
\item[(i)]Note that for each $(\theta_n)_{n \in \N} \subset (0,\infty)$ with $\lim_{n \rightarrow \infty} \theta_n = 0$ we have that $\PP_{\theta_n} \sim \operatorname{Exp}\left(\frac{1}{\theta_n}\right)$ converges for any $p \in \N_0$ in $\tau_p$ to $\delta_{\{0\}}=:\PP_0$. Indeed, the corresponding characteristic function satisfies 
\[
\varphi_{\PP_{\theta_n}}(t)=\frac{1}{1-i t {\theta_n}} \rightarrow 1 = \varphi_{\delta_{\{0\}}}(t) \text{ for all } t \in \R \text{ as } n \rightarrow \infty,
\]
which by L\'evy's continuity theorem for characteristic functions implies the result for $p=0$. If $p>0$ note that additionally
\[
\int_0^\infty |x|^p \PP_{\theta_n}(\D x) = \frac{p!}{\left(\tfrac{1}{\theta_n}\right)^p} = p! \theta_n^p \rightarrow 0 = \int_0^\infty |x|^p \delta_{\{0\}}(\D x) \text{ as } n \rightarrow \infty.
\]
Therefore, the result now follows by \cite[Theorem 2.2.1]{panaretos2020invitation}.
\item[(ii)]

Further, note that $\widehat{\theta}_t$ defined in \eqref{eq_defn_theta_hat} corresponds to the maximum likelihood estimator for the parametric family of distributions $\PP_{\theta} \sim \operatorname{Exp}\left(\tfrac{1}{\theta}\right), ~\theta \in (0,\infty)$. Indeed, for any $\theta \in (0,\infty)$, let $\Omega^t \ni (\omega_1,\dots,\omega_t)\mapsto L(\omega_1,\dots,\omega_t|\theta):= \prod_{i=1}^t \frac{1}{\theta} e^{-\frac{1}{\theta} \omega_i}$ denote the likelihood function and 
\[
\Omega^t \ni (\omega_1,\dots,\omega_t)\mapsto \ell(\omega_1,\dots,\omega_t|\theta):= \log\left( L(\omega_1,\dots,\omega_t|\theta)\right) =t\cdot  \log\left(\tfrac{1}{\theta}\right)-\frac{1}{\theta}\sum_{i=1}^t \omega_i
\]
the log-likelihood function.
Then, the partial derivative of the log-likelihood function w.r.t.\,$\theta$ is given by 
\begin{align*}
\frac{\partial}{\partial \theta} \ell(\omega_1,\dots,\omega_t|\theta) =-\frac{t}{\theta}+\frac{1}{\theta^2} \sum_{i=1}^t \omega_i
\end{align*}
which vanishes, since $\theta>0$ by assumption, if and only if  $\theta = \frac{1}{t}\sum_{i=1}^t \omega_i$.
\end{itemize}
\end{rem}
%
%
{
\subsection{Modeling of control sets}\label{sec_controls}
The following result provides natural examples of controls which satisfy Assumption~\ref{asu_A}.
\begin{prop}\label{prop_controls}
For any $t  \in \{0,1,\dots,T-1\}$, let $m_t \in \N$, let $A_0\subseteq \R^{m_0}$ be non-empty and compact, and let $\Omega^t\ni \omega^t \twoheadrightarrow \mathcal{A}_t(\omega^t)\subseteq \R^{m_t}$ satisfy one of the following three conditions.
\begin{itemize}
	\item[(i)] $\omega^t\twoheadrightarrow\mathcal{A}_t(\omega^t)=A_t\subseteq \R^d$ is a constant mapping with $A_t\subseteq \R^{m_t}$ being non-empty and compact.
	\item[(ii)] Let $A_t\subseteq \R^{m_t}$ be non-empty, closed, and convex. 
	Let $\widehat{a}_t:\Omega^t\mapsto A_t$ be Lipschitz continuous,
	let $\overline{\varepsilon}_t>0$, and
	let $\varepsilon_t:\Omega^t \mapsto [0,\overline{\varepsilon}_t]$ be Lipschitz continuous.
	Moreover, for any $\omega^t\in \Omega^t$
	  let
	\begin{equation*}
		\cA_t(\omega^t):=\Big\{a \in A_t \, \Big|\,  \Vert a-\widehat{a}_t(\omega^t)\Vert\leq \varepsilon_t(\omega^t)\Big\}.
	\end{equation*}
	\item[(iii)] For any $j=1,\dots, m_t$ let $\underline{a}_{t,j}:\Omega^t\mapsto \R$ and $\overline{a}_{t,j}:\Omega^t\mapsto \R$
	be both Lipschitz continuous and satisfy (pointwise) $\underline{a}_{t,j}\leq \overline{a}_{t,j}$.
	Moreover, for any $\omega^t\in \Omega^t$
	let
	\begin{equation*}
		\cA_t(\omega^t):= \bigtimes_{j=1}^{m_t} \Big[\underline{a}_{t,j}(\omega^t), \overline{a}_{t,j}(\omega^t)\Big]\subseteq \R^{m_t}.
	\end{equation*} 
 \end{itemize}
Then $\mathcal{A}_t$, $t=0,\dots,T-1$, satisfy Assumption~\ref{asu_A}.
\end{prop}
}
%
%
{
\section{Bounding the difference of value functions with and without\\ model uncertainty}\label{sec_RobustVSNonRobust}
In this subsection, we derive (non-asymptotic) upper bounds for the difference between the value function with and without uncertainty, in the case where the model uncertainty is described either by Wasserstein-uncertainty as in Section~\ref{sec_ambiguity_wasserstein} or by parameter uncertainty as in Section~\ref{sec_ambiguity_parametric}. 

More precisely, assume that there is a true probability measure
$\PP^{\rm{TR}}:= \PP_0^{\rm{TR}} \otimes \dots \otimes \PP_{T-1}^{\rm{TR}} \in \mathcal{M}_1(\Omega)$ given by 
\begin{equation}\label{eq_P_integrals-true}
	\PP^{\rm{TR}}(B) =\PP^{\rm{TR}}_0 \otimes \dots \otimes \PP^{\rm{TR}}_{T-1}(B):=\int_{\Omloc} \dots \int_{\Omloc} \one_B(\omega_1, \dots,\omega_T) \PP^{\rm{TR}}_{T-1}(\omega_1,\dots,\omega_{T-1}; \D \omega_T) \dots \PP^{\rm{TR}}_0(\D\omega_1),
\end{equation}
 $B \in \mathcal{F}$, which is unknown to the agent. Moreover, define the corresponding stochastic control problem \textit{without model uncertainty} by its value function
\begin{equation}\label{eq_optimization_problem-true}
	\mathcal{V}^{\rm{TR}}:=\sup_{\mathbf{a} \in \mathfrak{A}}  \E_{\PP^{\rm{TR}}} \left[\Psi(a_0,\dots,a_{T-1})\right], 
\end{equation}
where the utility function $\Psi:\Omega \times \mathcal{A}^T \to \R$ satisfies Assumption~\ref{asu_psi} and where the controls $\mathfrak{A}$ satisfy Assumption~\ref{asu_A}. The agent does not know $\PP^{\rm{TR}}  \in \mathcal{M}_1(\Omega)$, but has some reference measure 
$\widehat\PP:= \widehat\PP_0 \otimes \dots \otimes \widehat \PP_{T-1} \in \mathcal{M}_1(\Omega)$ given for any $B \in \mathcal{F}$ by 
\begin{equation}\label{eq_P_integrals-estimate}
	\widehat\PP(B) =\widehat\PP_0 \otimes \dots \otimes \widehat\PP_{T-1}(B):=\int_{\Omloc} \dots \int_{\Omloc} \one_B(\omega_1, \dots,\omega_T) \widehat\PP_{T-1}(\omega_1,\dots,\omega_{T-1}; \D \omega_T) \dots \widehat\PP_0(\D\omega_1). 
\end{equation}
We think of $\widehat\PP$ as the agent's estimate for the true but unknown probability measure $\PP^{\rm{TR}}$. If the agent is confident that her estimate $\widehat\PP$  reflects the real behavior of the environment accurately, the agent aims to solve the following stochastic control problem \textit{without model uncertainty} with respect to reference measure defined by its value function
\begin{equation}\label{eq_optimization_problem-reference}
	\widehat{\mathcal{V}}:=\sup_{\mathbf{a} \in \mathfrak{A}}  \E_{\widehat\PP} \left[\Psi(a_0,\dots,a_{T-1})\right]. 
\end{equation}
If the agent is not fully confident in her estimate, she may want to solve the stochastic control problem \textit{with model uncertainty}  described in \eqref{eq_optimization_problem-reference}, where in this section the model uncertainty is described either by a Wasserstein-ball around each estimated kernel $\widehat\PP_{t}$ as described in Section~\ref{sec_ambiguity_wasserstein} or by parameter uncertainty as in Section~\ref{sec_ambiguity_parametric}, where in the parametric case we assume that $\widehat{\PP}_t=\PP_{\widehat{\theta}_t}$, $t=0,\dots,T-1$, for reference parameters $\theta_t$, $t=0,\dots,T-1$.

In the following result, we aim to bound the difference between these value functions under the assumption that for each time $t$, the reference kernel $\widehat{\PP}_t$ might not coincide with the true one $\PP^{\rm{TR}}_{t}$, but is a good enough estimate in the sense that $\PP^{\rm{TR}}_{t} \in \mathcal{P}_t$.

To that end, we introduce for each $s=0,\dots,T-1$ and each $t=0,\dots,T-1$   with $t\leq s$  the following quantities via backwards recursion:

\begin{equation}
	\label{eq:def:err-alpha}
\begin{split}
\mbox{For each $s=1,\dots T-1$ set: } \quad \mu^{\rm{err},\alpha}_{s,s}(\omega^s)&:= \big(d_{W_1}\big(\PP^{\rm{TR}}_{s}(\omega^{s}),\widehat{\PP}_{s}(\omega^{s})\big)\big)^\alpha, \qquad \omega^s \in \Omega^s,\\
\mbox{for each $t=s-1,\dots 1$ set: } \quad \mu^{\rm{err},\alpha}_{s,t}(\omega^t)&:= \E_{\PP^{\rm{TR}}_{t}(\omega^{t})}\big[\mu^{\rm{err},\alpha}_{s,t+1}(\omega^{t}\otimes_{t} \cdot)\big], \qquad \omega^t \in \Omega^t,\\
\mu^{\rm{err},\alpha}_{s,0}&:= \E_{\PP^{\rm{TR}}_{0}}\big[\mu^{\rm{err},\alpha}_{s,1}(\cdot)\big],\\
\mu^{\rm{err},\alpha}_{0,0}&:= \big(d_{W_1}\big(\PP^{\rm{TR}}_{0},\widehat{\PP}_{0}\big)\big)^\alpha,
\end{split}
\end{equation}
where $\alpha \in (0,1]$ is the H\"older exponent introduced in Assumption~\ref{asu_psi}~(i).
Moreover, under the setting of Theorem~\ref{thm_wasserstein_ambiguity} or of Theorem~\ref{thm_param_assumptions}, we define the following quantities with respect to the corresponding radii functions $\varepsilon_t:\Omega^t \to [0,\overline{\varepsilon}_t]$, $t=1,\dots T-1,$ and $\varepsilon_0\geq 0$.
\begin{equation}
	\label{eq:def:epsilon-alpha}
	\begin{split}
\mbox{for each $s=1,\dots T-1$ set: } \quad		\mu^{\varepsilon,\alpha}_{s,s}(\omega^s)&:= \big(\varepsilon_s(\omega^{s})\big)^\alpha, \qquad \omega^s \in \Omega^s,\\
		\mbox{for each $t=s-1,\dots 1$ set: } \quad \mu^{\varepsilon,\alpha}_{s,t}(\omega^t)&:= \E_{\PP^{\rm{TR}}_{t}(\omega^{t})}\big[\mu^{\varepsilon,\alpha}_{s,t+1}(\omega^{t}\otimes_{t} \cdot)\big], \qquad \omega^t \in \Omega^t,\\
		\mu^{\varepsilon,\alpha}_{s,0}&:= \E_{\PP^{\rm{TR}}_{0}}\big[\mu^{\varepsilon,\alpha}_{s,1}(\cdot)\big],\\
		\mu^{\varepsilon,\alpha}_{0,0}&:= \varepsilon_0^\alpha.
	\end{split}
\end{equation}
\begin{thm}\label{thm:RobustVSNonRobust}
	Let Assumption~\ref{asu_A}, Assumption~\ref{asu_psi}, and Assumption~\ref{asu_P} hold and let $p \in \N_0$ be the integer from Assumption~\ref{asu_psi}.
	 Then the following holds:
	 \vspace{0.1cm}
	 \begin{enumerate}
	 	\item[(1)]\textbf{Stability:}
	 	
	 	\vspace{0.1cm}
	 	\noindent
Assume additionally to the above that 
the true measure $\PP^{\rm{TR}}:= \PP_0^{\rm{TR}} \otimes \dots \otimes \PP_{T-1}^{\rm{TR}} \in \mathcal{M}_1(\Omega)$ and the reference measure $\widehat\PP:=\widehat\PP_0 \otimes \dots \otimes \widehat\PP_{T-1} \in \mathcal{M}_1(\Omega)$ satisfy for each $t \in \{1,\dots,T-1\}$ 
that 
	$\Omega^t \ni \omega^t \mapsto \widehat{\PP}_t(\omega^t) \in \mathcal{M}_1^{\max\{1,p\}}(\Omloc)$
	and
	 $\Omega^t \ni \omega^t \mapsto \PP^{\rm{TR}}_t(\omega^t) \in \mathcal{M}_1^{\max\{1,p\}}(\Omloc)$ are both Lipschitz continuous with respect to the $d_{W_{1}}(\cdot,\cdot)$-metric 
and some Lipschitz-constants $L_{\widehat{\mathbb{P}},t} \geq 0$ and $L_{\mathbb{P}^{\rm{TR}},t} \geq 0$, respectively, that both $\PP^{\rm{TR}}_0, \widehat{\PP}_0 \in \mathcal{M}_1^{\max\{1,p\}}(\Omloc)$, as well as 
for  each $s=0,\dots,T-1$ and  $t=0,\dots, T-1$ with $t\leq s$  that $\mu_{s,t}^{\rm{err},\alpha}<\infty$. Then
\begin{equation}\label{eq:stability}
	\big|\mathcal{V}^{\rm{TR}}-\widehat{\mathcal{V}}\big|
	\leq
	L_{\Psi} \sum_{s={0}}^{T-1} \bigg[
	\Big(2^{T-(s+1)}  \prod_{u=s+1}^{T-1} \max\left\{L_{\mathcal{A},u}^{\alpha}+L_{\widehat{\PP},u}^{\alpha},1\right\}\Big)
	\cdot
	\mu^{\rm{err},\alpha}_{s,0} \bigg]<\infty.
\end{equation}
	\item[(2)] \textbf{Robust vs.\ non-robust in Wasserstein-uncertainty-case}:
	
	\vspace{0.1cm}
	\noindent
	Assume here additionally to the above that the setting and assumptions of Theorem~\ref{thm_wasserstein_ambiguity} hold with respect to the reference measure  $\widehat\PP:=\widehat\PP_0 \otimes \dots \otimes \widehat\PP_{T-1} \in \mathcal{M}_1(\Omega)$ and with $p<q\in \N$. Moreover, assume  for each $t=1\dots,T-1$  that 	$\Omega^t \ni \omega^t \mapsto \PP^{\rm{TR}}_t(\omega^t) \in \mathcal{M}_1^{\max\{1,p\}}(\Omloc)$ is Lipschitz continuous with respect to the $d_{W_{1}}(\cdot,\cdot)$-metric 
	and some Lipschitz-constant $L_{\mathbb{P}^{\rm{TR}},t} \geq 0$
	and that 
	$\PP^{\rm{TR}}_t(\omega^t)\in \mathcal{B}_{\varepsilon_t(\omega^t)}^{(q)}\big(\widehat{\PP}_t(\omega^t)\big)$ holds for each $\omega^t \in \Omega^t$, as well as that
	$\PP^{\rm{TR}}_0\in \mathcal{B}_{\varepsilon_0}^{(q)}\big(\widehat{\PP}_0\big)$. 
	Then the following holds:
	\begin{equation}\label{eq:upper-bound-Wasserstein}
		0 \leq \mathcal{V}^{\rm{TR}}-\mathcal{V}
		\leq 
		2^{\alpha} L_{\Psi} \sum_{s={0}}^{T-1} \bigg[
		\Big(2^{T-(s+1)}  \prod_{u=s+1}^{T-1} \max\left\{L_{\mathcal{A},u}^{\alpha}+(L_{\widehat{\PP},u}+L_{\varepsilon,u})^{\alpha},1\right\}\Big)
		\cdot
		\mu^{\varepsilon,\alpha}_{s,0} \bigg]<\infty.
	\end{equation} 
	\item[(3)] \textbf{Robust vs.\ non-robust in Parameter-uncertainty-case:}
	
	\vspace{0.1cm}
	\noindent
Assume here additionally to the above that now the setting and assumptions of Theorem~\ref{thm_param_assumptions} hold, that for each $t \in \{1,\dots,T-1\}$ and $\omega^t\in \Omega^t$ we have that $\widehat{\PP}_t(\omega^t)=\PP_{\widehat{\theta}_t(\omega^t)}$ holds with
	$\widehat{\theta}_t: \Omega^t \rightarrow \Theta_t$ being defined in \ref{eq_a3}
	and that $\widehat{\PP}_0=\PP_{\widehat{\theta}_0}$ holds with
	$\widehat{\theta}_0 \in \Theta_0$ being defined in \ref{eq_a5}, as well as
	that for each $t \in \{1,\dots,T-1\}$ and $\omega^t \in \Omega^t$ we have that $\PP^{\rm{TR}}_t(\omega^t)=\PP_{\theta^{\rm{TR}}_t(\omega^t)}$ holds for some
	$\theta^{\rm{TR}}_t: \Omega^t \rightarrow \Theta_t$ being Lipschitz continuous with respect to some Lipschitz-constant $L_{\theta^{\rm{TR}},t}\geq 0$ 
	 satisfying for each $\omega^t\in \Omega^t$ that 
	$\| \theta^{\rm{TR}}_t(\omega^t) - \widehat{\theta}_t(\omega^t)\| \leq \varepsilon_t(\omega^t)$, and that $\PP^{\rm{TR}}_0=\PP_{\theta^{\rm{TR}}_0}$ holds for some $\theta^{\rm{TR}}_0 \in \Theta_0$ satisfying 
		$\| \theta^{\rm{TR}}_0 - \widehat{\theta}_0\| \leq \varepsilon_0$. Then
		\begin{equation}\label{eq:upper-bound-parametric}
		0 \leq \mathcal{V}^{\rm{TR}}-\mathcal{V}
		\leq 
		2^{\alpha} L_{\Psi} \sum_{s={0}}^{T-1} \bigg[
		\Big(2^{T-(s+1)}  \prod_{u=s+1}^{T-1} \max\left\{L_{\mathcal{A},u}^{\alpha}+\big(L_{\PP_{{\theta}},u} \cdot (L_{\widehat{\theta},u} + L_{\varepsilon,u})\big)^{\alpha},1\right\}\Big)
		\cdot
		L_{\PP_{\theta},s}^\alpha
		\cdot
		\mu^{\varepsilon,\alpha}_{s,0} \bigg]<\infty.
	\end{equation}
	 \end{enumerate}	
\end{thm}
}
%
%
{
\section{Comparison to existing literature and our contributions}
\label{sec:Comparison}
In this paper we provide a general framework for multi-period finite discrete time  stochastic optimal control problems under model uncertainty where the utility function is allowed to be non-concave. Our assumptions both on the utility function (see Assumption~\ref{asu_psi}) and on the set of probability measures describing the ambiguity (see Assumption~\ref{asu_P}) are kept fairly general allowing for a broad class of stochastic control problems to fit into our setup.
In our main result, Theorem~\ref{thm_main_result}, we  establish a dynamic programming principle  which allows us to obtain both optimal control and worst-case measure by solving recursively a sequence of one step optimization problems. Moreover, in Theorem~\ref{thm_wasserstein_ambiguity} and Theorem~\ref{thm_param_assumptions} we demonstrate that our general assumptions imposed on the set of priors  are naturally fulfilled by important classes of ambiguity sets such as Wasserstein-balls around reference probability measures and by parametric classes of probability distributions.

Let us compare our result and our assumptions imposed to obtain these results with the literature closest to our work. Multi-period finite discrete time stochastic optimal control problems under model uncertainty have been  solved both for concave utility functions \cite{blanchard2018multiple, carassus2023discrete,  neufeld2018robust, nutz2016utility} and for non-concave utility functions \cite{bayraktar2023nonparametric, bayraktar2022data, carassus2024nonconcave,  neufeld2019nonconcave}, by deriving a dynamic programming principle. The main approach in all of these works is to find suitable conditions on the utility function, the set of priors, and the set of controls which can be propagated through backwards induction  so that the dynamic programming principle (DPP) holds. 

In the papers \cite{blanchard2018multiple, carassus2023discrete,  neufeld2018robust, nutz2016utility} where the utility function is concave, the corresponding authors managed to keep the requirements on the utility function to be minimal. Roughly speaking (using our notation) they only required the utility function $a^T \mapsto \Psi(\omega^T,a^T)$ to be upper-semicontinuous for each $\omega^T \in \Omega$ and that $(\omega^T,a^T) \mapsto \Psi(\omega^T,a^T)$ is either Borel measurable or (even more general) lower semianalytic. To define the set of probability measures describing the ambiguity  they followed the framework of \cite{BouchardNutz} where at each time $t$ a set-valued map of local priors (instead of a single kernel) characterizes the model uncertainty for the next time step.
 This structure is also in line with the condition known as \textit{rectangularity} of the ambiguity sets of kernels (where there the set-valued map typically only depends on the current state and action, whereas \cite{BouchardNutz} allows for path dependency). \cite{blanchard2018multiple, carassus2023discrete,  neufeld2018robust, nutz2016utility} only required the graph of the set-valued map of kernels to be analytic, which is a very minimal measurability condition imposed in order to guarantee the existence of (universally) measurable selectors (see, e.g., \cite[Proposition~7.49, p.182]{bertsekas1996stochastic}).

In the papers \cite{carassus2024nonconcave,  neufeld2019nonconcave} analyzing the non-concave case, they followed the approach of \cite{blanchard2018multiple, carassus2023discrete,  neufeld2018robust, nutz2016utility} to formulate as minimal regularity conditions on the utility function and the priors. Now a priori, not concavity is the key assumption to get DPP to hold, but rather continuity conditions which can be propagated through backwards induction in order to obtain, e.g., measurable selectors. However, the release of the concavity  assumption on the utility function naturally reduces the regularity,
which increases the difficulty to obtain suitable measurable selectors in order to derive the dynamic programming principle. As a consequence,  \cite{neufeld2019nonconcave} had to assume that set of controls take values on a grid (in particular a discrete set), whereas \cite{carassus2024nonconcave} assumed the set-theoretical axiom of
\textit{Projective Determinacy (PD)} to hold, an axiom  which cannot be proved in the usual Zermelo-Fraenkel set theory with the Axiom of Choice (ZFC), in order to get suitable selectors.

To overcome the difficulty for the non-concave case, our approach is to assume more regularity on the utility function (namely to be H\"older continuous), as well as on the set-valued map of (local) priors within the framework of \cite{BouchardNutz}, namely by assuming that they are continuous (as set-valued maps) and compact valued. From a technical point of view, we introduce a new stability/continuity condition on the set-valued map of local priors (see Assumption~\ref{asu_P}~(iii)) which allows to propagate the H\"older continuity of the value functions $\Psi_t$ at each time $t=T,\dots, 1,0$ through backward induction. While Assumption~\ref{asu_P}~(iii) naturally appears in the proof for the backward propagation of the H\"older continuity, it is one of our main contributions of this paper to demonstrate that Assumption~\ref{asu_P}~(iii) actually naturally holds true for relevant examples of set of priors such as the Wasserstein-ball around a (possibly) path-dependent reference measure with (possibly) path-dependent radius (see Theorem~\ref{thm_wasserstein_ambiguity}) as well as parametric classes of distributions (see Theorem~\ref{thm_param_assumptions}).

Now let us comment on the set of controls. All of the related literature \cite{blanchard2018multiple, carassus2023discrete, carassus2024nonconcave, neufeld2018robust, neufeld2019nonconcave, nutz2016utility} do not assume any constraints on the values of the strategies, whereas we  impose that the set-valued maps describing the possible values of the controls 
are compact-valued. That said, compared to the literature above, we do not impose any no-arbitrage condition. More precisely, while \cite{blanchard2018multiple, carassus2023discrete, carassus2024nonconcave, nutz2016utility} imposed some robust no-arbitrage conditions which can be seen as the natural extensions of classical no-arbitrage conditions in financial markets without model uncertainty, \cite{neufeld2018robust, neufeld2019nonconcave} introduced
general linearity type conditions which can be seen as the robust analog of the linearity type conditions introduced in \cite{pennanen2011convex, pennanen2012stochastic} in settings without model uncertainty. While these linearity type conditions are closely related to no-arbitrage conditions in financial markets, topologically speaking they ensure that one can restrict one-selves to strategies taking values on a compact set, see, e.g., \cite[Lemma~2.7]{nutz2016utility} or \cite[Proposition]{neufeld2018robust}. This basically follows from the following 
observations in optimization \cite[Theorem 9.2]{Rockafellar.97} and \cite[Theorem~3.31 \& Theorem~1.17]{RockafellarWets2009variational}, see also the discussion in \cite[Remark~2.6]{neufeld2018robust}.

Moreover, previous to our work, \cite{bayraktar2023nonparametric, bayraktar2022data} introduced a framework for data-driven adaptive robust control problems. In the case where the set-valued maps of kernels consist of Wasserstein-balls around path-dependent reference measures (see Theorem~\ref{thm_wasserstein_ambiguity}), we can basically recover the setup of \cite{bayraktar2023nonparametric, bayraktar2022data}, but where our setting allows the radii of the Wasserstein-balls also to be random and path-dependent, instead of being deterministic as in \cite{bayraktar2023nonparametric, bayraktar2022data}.
We  refer to Remark~\ref{rem:adaptive-robust-control} for some further comparison to  \cite{bayraktar2023nonparametric, bayraktar2022data}.

Furthermore, \cite{bauerle2021q,chen2019distributionally,neufeld2023mdp,uugurlu2018robust,xu2010distributionally} have analyzed Markov decision problems under model uncertainty. This can be seen as the \textit{infinite-time} analog of this work in the special case where the utility function~$\Psi$ is of the form  
\begin{equation}\label{eq:MDP-utility}
	\Psi(\omega_1,\omega_2,\dots,a_0,a_1,\dots,)=\sum_{t=0}^\infty \gamma^t r(\omega_t,a_t,\omega_{t+1}),
\end{equation} 
where $\gamma \in (0,1)$ and $r(\cdot,\cdot,\cdot)$ is some reward function. In \cite{chen2019distributionally,xu2010distributionally}, the authors assumed that both state and action space is finite.   \cite{bauerle2021q} assumed that the ambiguity set of probability measures is \textit{dominated}, whereas \cite{uugurlu2018robust} considered so called \textit{conditional risk mappings} on an atomless probability space. The closest work on Markov decision problems under model uncertainty to this paper is \cite{neufeld2023mdp}. 
Similarly to this work, \cite{neufeld2023mdp} assumed Lipschitz-continuity type of regularity on the reward function $r(\cdot,\cdot,\cdot)$. Moreover, the set-valued map of (local) kernels were also assumed in \cite{neufeld2023mdp} to be  compact-valued and continuous,  but in \cite{neufeld2023mdp} the set valued-map only takes the current state and action as input, not the whole path. Moreover, the values of the controls in \cite{neufeld2023mdp} are assumed to take values in a \textit{fixed} compact set, whereas in this work the controls are allowed to take values in a compact valued \textit{correspondence} which is path-dependent.
Moreover, we highlight that due to the different nature of the  infinite time stochastic control problem with utility function \eqref{eq:MDP-utility}, the corresponding Bellman equation consists only of a single one-step (local) optimization involving a fixed-point.
Hence one does not have to solve $T$-many different one-step optimization problems through backward induction. In particular, one does not need to find suitable conditions on the utility function, the set of priors, and the set of controls which can be propagated through backward induction so that DPP holds. As a consequence, a condition on the priors such as Assumption~\ref{asu_P}~(iii) was not necessary in \cite{neufeld2023mdp}. It is one of the technical contributions of this paper to realize that Assumption~\ref{asu_P}~(iii) is a good condition allowing to backward propagate the $\alpha$-H\"older continuity of the utility function $\Psi$, while at the same time realizing that Assumption~\ref{asu_P}~(iii) naturally holds true for relevant examples for set of priors (see Theorem~\ref{thm_wasserstein_ambiguity} and Theorem~\ref{thm_param_assumptions}). 

Next, in Section~\ref{sec_RobustVSNonRobust} we obtain an upper bound for the difference between the values of the (non-robust) stochastic control problems defined with respect to the true (but to the agent unknown) probability measure $\PP^{\rm{TR}}:= \PP_0^{\rm{TR}} \otimes \dots \otimes \PP_{T-1}^{\rm{TR}} \in \mathcal{M}_1(\Omega)$ and the reference measure $\widehat\PP:= \widehat\PP_0 \otimes \dots \otimes \widehat \PP_{T-1} \in \mathcal{M}_1(\Omega)$, respectively. 
The upper bound is expressed in terms of the estimation error of the true kernel $\PP^{\rm{TR}}_t$ at each time~$t$  defined as the (integrated) $d_{W_1}$-distance of $\PP_{t}^{\rm{TR}}$ and its estimate $\widehat \PP_{t}$, see Theorem~\ref{thm:RobustVSNonRobust}(1).
Moreover, in the setting of Wasserstein-uncertainty (see Theorem~\ref{thm_wasserstein_ambiguity}) or the setting of parameter-uncertainty (see Theorem~\ref{thm_param_assumptions}), we obtain an upper bound for the difference of the value of the robust stochastic control  problem and the non-robust stochastic control problem defined with respect to the reference measure $\widehat\PP$, see Theorem~\ref{thm:RobustVSNonRobust}~(2)~\&~(3). Theorem~\ref{thm:RobustVSNonRobust}~(1) hence can be seen as the path-dependent generalization of such stability results obtained  for Markov decision problems (MDP) with corresponding utility function of the form \eqref{eq:MDP-utility} 
in a finite horizon setting  (\cite[Theorem~2.2.8]{kern2020sensitivity}, \cite[Theorem~6.2]{zahle2022concept}, and \cite[Theorem~3]{kiszka2022stability}))
and in an infinite horizon setting (\cite[Theorem~4.2]{muller1997does}). Moreover, the upper bound  in Theorem~\ref{thm:RobustVSNonRobust}~(2) for the difference between the values of the robust stochastic control problem and the non-robust stochastic control problem 
 when model uncertainty is described by Wasserstein-balls around the reference kernel at each time~$t$ can be seen as the path-dependent generalization in a finite horizon setting of the upper bound obtained in \cite[Theorem~3.1]{neufeldsester2024upperbound}.
In addition, \cite{bartl2022sensitivity} has analyzed a stochastic control problem in a finite horizon setting where model uncertainty is described via a ball around the reference measure defined with respect to the \textit{adapted} Wasserstein distance. In their main result \cite[Theorem~2.4]{bartl2022sensitivity}, they showed that \textit{asymptotically} when the radius of the ball tends to zero, the value of the robust stochastic control problem asymptotically equals to
the value of the non-robust problem
plus an explicit correction term. Roughly speaking, this result hence can be seen in spirit as a similar result to our Theorem~\ref{thm:RobustVSNonRobust}(2) in the particular situation where our reference measure coincides with the true probability measure.
Furthermore, to the best of our knowledge, we are not aware of any result in the literature that has analyzed the difference between the values of the robust stochastic control problem and the non-robust stochastic control problem in the parameter-uncertainty-case similar to our result obtained in Theorem~\ref{thm:RobustVSNonRobust}(3).

Finally, we highlight that our approach heavily relies on \textit{time-consistency} of the stochastic control problem due to our approach of solving it using dynamic programming principle. Therefore, relevant \textit{time-inconsistent} stochastic optimization problems such as, e.g., multi-period mean-variance portfolio optimization (see, e.g., \cite{hakansson1971multi}) or multi-period portfolio selection under cumulative prospect theory (see, e.g., \cite{shi2015discrete}) cannot (directly) be covered in our setting. We leave this for future research to see how one can solve such multi-period time-inconsistent stochastic  optimal control problems under model uncertainty.
}
%
%
\section{{Application: Data-driven hedging with asymmetric loss function}}\label{sec_applications}

\subsection{Setting to data-driven hedging with asymmetric loss function}\label{subsec_hedging}
We consider {$d \in \N$} underlying financial {assets} attaining values $S_{0} ,S_{1},\dots,S_{T} { \in \R^d}$ over the future time horizon $0, 1,\dots, T$ as well as some financial derivative with payoff $\Phi(S_0,\dots,S_T)$ paid at maturity $T$. We face the problem of having sold the derivative at initial time $0$ and therefore being exposed to possibly unexpected high   payoff obligations $\Phi(S_0,\dots,S_T)$. To reduce this risk, we are interested in \emph{hedging} the payoff obligation $\Phi(S_0,\dots,S_T)$ by investing in the underlying asset such that our potential exposure is reduced, compare also, e.g. \cite{hull1993options} for further financial background on hedging of financial derivatives.

We start by modeling an ambiguity set of probability measures and follow to this end the approach outlined in Section~\ref{sec_ambiguity_wasserstein}. The underlying asset returns in the time period between $t-1$ and $t$ are given by
\[
\mathcal{R}_{t}:=\frac{S_{t}-S_{t-1}}{S_{t-1}}\in \Omloc\subseteq { \R^d}, \qquad t \in \{1,\dots,T\},
\]
where $\Omloc ={[-C,C]^d \subset  \R^d}$ for some constant value $C>0$.
To construct ambiguity sets of probability measures {to model the evolution of the returns}, we consider a time series of historically realized returns 
\begin{equation}\label{eq_time_series_returns}
\left(\mathscr{R}_1,\dots,\mathscr{R}_{N} \right) \in \Omloc^{\hspace{-0.4cm}N},\qquad \text{ for some } N \in \N.
\end{equation}
Relying on the time series from \eqref{eq_time_series_returns}, we design ambiguity sets $\mathcal{P}_t$, $t=0,\dots,T-1$. To this end, we define $\widehat{\PP}_0$ through a sum of Dirac-measures given by
\begin{equation}\label{eq_defn_phat_portfolio_1}
\widehat{\PP}_0(\D x):= \frac{1}{N}\sum_{s=1}^{N}\delta_{\mathscr{R}_{s}}(\D x)\in \mathcal{M}_1(\Omloc ),
\end{equation}
and $\widehat{\PP}_t$ for $t=1,\dots,T-1$ by
\begin{equation}\label{eq_defn_phat_portfolio_2}
\Omega^t \ni \omega^t=(\omega_1,\dots,\omega_t) \mapsto \widehat{\PP}_t\left(\omega^t\right) (\D x):= \sum_{s=t}^{N-1}\pi_s^t(\omega^t) \cdot \delta_{\mathscr{R}_{s+1}}(\D x)\in \mathcal{M}_1(\Omloc),
\end{equation}
where $\pi_s^t(\omega^t)\in [0,1]$, $s =t,\cdots,N-1$ with $\sum_{s=t}^{N-1}\pi_s^t(\omega^t) = 1$. We want to weight the distance between the past $t$ returns  before $\mathcal{R}_{t+1}$ and  the $t$ returns before $\mathscr{R}_{s+1}$, while assigning higher probabilities to more similar sequences of $t$ returns. This means, the measure {$\widehat{\PP}_t$ assigns the highest probability to that historical return that followed the sequence of historical returns which resembles the current market situation the most. Mathematically this construction is akin to kernel regression methods from time series analysis (see, e.g., \cite{badiane2018kernel}) where nearby points (in our case similar sequences of returns) are assumed to imply similar future evolution of returns}. To this end, we set
\begin{equation*}\label{eq_defn_ps_empirical}
\Omega^t \ni \omega^t \mapsto \pi_s^t(\omega^t):= \left(\frac{\exp(-\beta \operatorname{dist}_s(\omega^t)^2)}{\sum_{\ell=t}^{N-1}\exp(-\beta\operatorname{dist}_{\ell}(\omega^t)^2)}\right),
\end{equation*}
for\footnote{ In our experiments, we chose an inverse temperature parameter of $\beta=500$.} { inverse temperature parameter} $\beta >0$ and with
\begin{align*}
\operatorname{dist}_s(\omega^t):=\left\|\left(\mathscr{R}_{s-t+1},\cdots,\mathscr{R}_{s}\right)-\omega^t\right\|=\left\|\left(\mathscr{R}_{s-t+1},\cdots,\mathscr{R}_{s}\right)-\left(\omega_1,\cdots,\omega_t\right)\right\|,~~ 
\end{align*}
for all $s=t,\cdots,N-1$. Then, for any\footnote{ We choose for our numerical results here $\varepsilon$ to be constant and analyze various choices of $\varepsilon$ empirically.} {$\varepsilon \in [0,\infty)$ } we define ambiguity sets of probability measures via
\begin{equation}\label{eq_P_hedging}
\begin{aligned}
\Omega^t \ni \omega^t \mapsto \mathcal{P}_t\left({\omega}^{t}\right):=\mathcal{B}_{\varepsilon}^{(1)} \left(\widehat{\PP}_t(\omega^t)\right)&\text{ for } t \in \{1,\dots,T-1\},\\
 \mathcal{P}_0:= \mathcal{B}_{\varepsilon}^{(1)} \left(\widehat{\PP}_0\right)&\text{ for }t =0,
\end{aligned}
\end{equation}
where $\mathcal{B}_{\varepsilon}^{(1)} (\widehat{\PP})$ denotes a Wasserstein-ball of order $1$ with radius $\varepsilon$ around some probability measure $\widehat{\PP}$, as defined in Theorem~\ref{thm_wasserstein_ambiguity}.

As our assumptions on the objective function $\Psi$ formulated in Assumption~\ref{asu_psi} do not require the objective function to be concave, we are able to consider non-concave objective functions as they appear in behavioral economics and, in particular, in prospect theory (\cite{kahneman1979prospect}). To this end, we consider a {loss function} of the form
\begin{equation}\label{eq_u_prospect}
\R \ni x \mapsto U(x):= x^{\mathfrak{{a}}}\one_{\{x \geq 0\}}+ \mathfrak{b} \cdot (-x)^{\mathfrak{a}}\one_{\{x < 0\}} \qquad \text{ where } \mathfrak{a} \in (0,1), \mathfrak{b}>1,
\end{equation}
which assigns the value $x^{\mathfrak{a}}$ to positive inputs $x$ and $\mathfrak{b}x^{\mathfrak{a}}$ to negative inputs $x$, see also \cite{tversky1989rational} or \cite{weber2009decisions}, and compare the illustration in Figure~\ref{fig_s_shape}. Considering an asymmetric loss function as in \eqref{eq_u_prospect} accounts for the empirical fact that agents exhibit distinct behaviors in response to gains and losses, namely they are much more sensitive with respect to losses.
 Experimental studies reported in \cite{tversky1992advances} provide estimates of $\mathfrak{a}=0.88$, $\mathfrak{b} = 2.25$.
\begin{figure}[h!]
\centering
\includegraphics[scale=0.3]{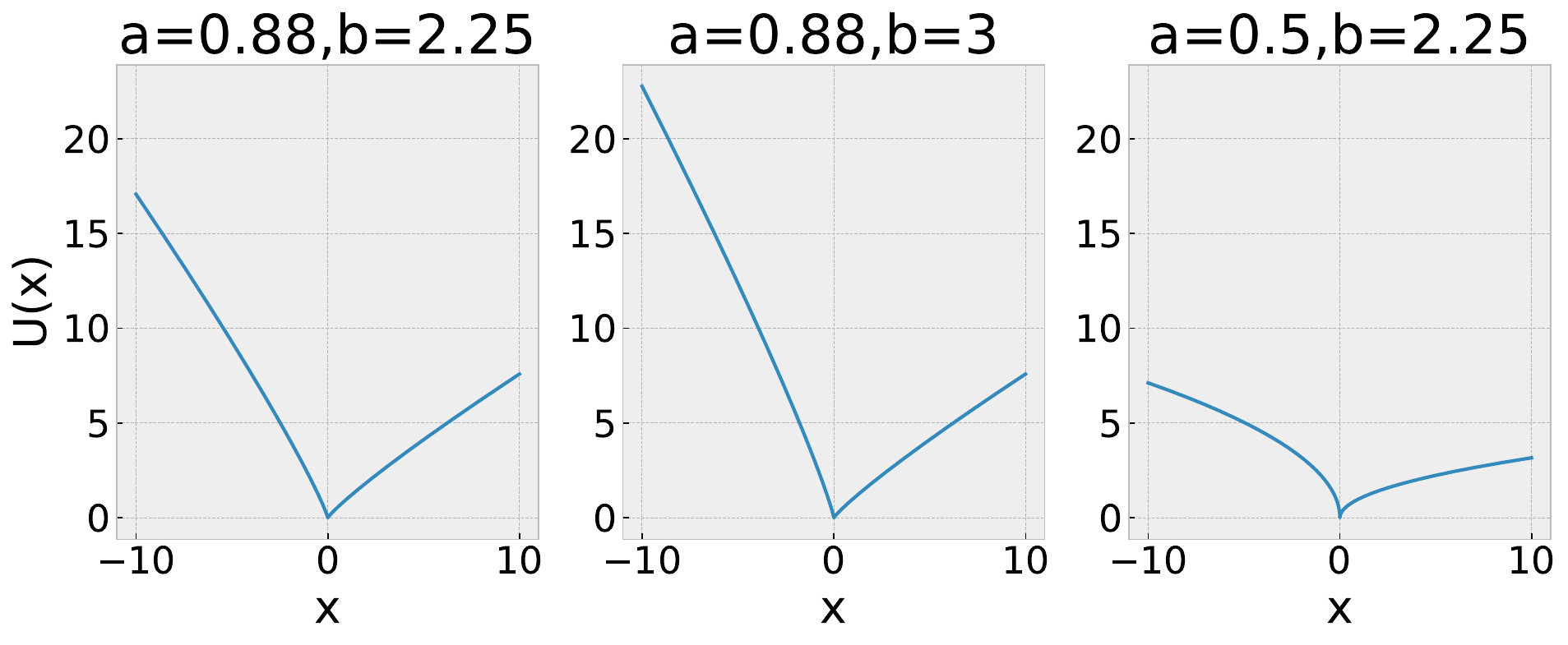}
\caption{ The three plots show the loss function described in \eqref{eq_u_prospect} for different choices of the parameters $\mathfrak{a}$ and $\mathfrak{b}$. }\label{fig_s_shape}
\end{figure}

To define our optimization problem, we consider some financial derivative with payoff function $\Phi(S_0,S_1,\dots,S_T)$. The agent's goal is to find a hedging strategy, i.e., a combination of an initial cash position $d_0$ and time-dependent positions { $(\Delta_{j})_{j=0,\dots,d}:=(\Delta_j^i)_{i=1,\dots,d, \atop j=0,\dots,T-1}$} invested in the underlying {assets}
minimizing the \emph{hedging error}
between the time-$T$ value of the self-financing hedging strategy given by 
{
\[
d_0+\sum_{i=1}^d \sum_{j=0}^{T-1}  \Delta_j^i (S_{j+1}^i-S_j^i) = d_0+\sum_{i=1}^d\sum_{j=0}^{T-1}   \Delta_j^i \left\{S_0^i \left(\prod_{k=1}^{j}(\omega_k^i+1)\right)\cdot \omega_{j+1}^i\right\}
\]
}
and the payoff of $\Phi$, { where we have used for all $t=0,\dots,T$ the notation $S_t = (S_t^i)_{i=1,\dots,d}$ as well as $\omega_t = (\omega_t^i)_{i=1,\dots,d} \in \Omloc$, and that $\omega_t = \mathcal{R}_t = \frac{S_t-S_{t-1}}{S_t}$ for each $\omega_t \in \Omloc$} . To this end, we identify the financial positions of the agent with her actions and define for some $\overline{B}>0$ and $\overline{A}>0$ the { (here path-independent)} sets of actions $({\mathcal{A}}_T)_{t=0,\dots,T-1}$ by 
{
\begin{align*}
{\mathcal{A}}_0&:= \left\{(d_0,(\Delta_0^i)_{i=1,\dots,d}) \in [-\overline{B},\overline{B}]\times [-\overline{A},\overline{A}]^d \right\},\\
{\mathcal{A}}_t&:= \left\{(\Delta_t^i)_{i=1,\dots,d} \in [-\overline{A},\overline{A}]^d \right\},~\text{ for }~t=1,\dots,T-1.
\end{align*}
} {Then we} introduce the objective function
\begin{equation}\label{eq_Psi_hedging}
\begin{aligned}
\Omega \times {\mathcal{A}}^T \ni ({\omega^T, a^T}) \mapsto \Psi({\omega^T, a^T}):= -{{U}\Bigg(d_0+\sum_{i=1}^d \sum_{j=0}^{T-1}  \Delta_j^i (S_{j+1}^i-S_j^i)}-\Phi(S_0,S_1,\dots,S_T)\Bigg),
\end{aligned}
\end{equation}
where we abbreviate {$a^T = \left(d_0,(\Delta_0^i)_{i=1,\dots,d} , (\Delta_1^i)_{i=1,\dots,d},\dots,(\Delta_{T-1}^i)_{i=1,\dots,d} \right) \in {\mathcal{A}_0 \times \cdots \times  \mathcal{A}_{T-1}}$, and where $S_t = S_0\prod_{n=1}^t(\omega_n+1)$ for each $t=1,\dots,T$ with $S_0 \in \R^d$.}

Solving $\sup_{\mathbf{a} \in \mathfrak{A}} \inf_{\PP \in \mathfrak{P}} \E_{\PP} \left[\Psi(a_0,\dots,a_{T-1})\right]$ with $\Psi$ as defined in \eqref{eq_Psi_hedging} corresponds to minimizing the distributionally robust expected hedging error induced by {$U$} between the outcome of a self-financing hedging strategy and some financial derivative {with payoff function $\Phi$}. 
Considering the situation of a financial institution which has sold the derivative $\Phi$, the asymmetry of {$U$}  accounts for the fact that a hedging-strategy that leads to higher payoffs than the derivative is considered a better outcome than a hedging-strategy that leads to smaller payoffs than the derivative payoff and therefore to unexpected losses.  We refer to  \cite{broll2010prospect}, \cite{carbonneau2023deep}, and  \cite{gobet2020option} for related literature on asymmetric hedging, but without model uncertainty. 

{ Moreover, we compare our proposed approach (i.e., the ambiguity sets of probability measures, defined in \eqref{eq_P_hedging}) with a robust adaptive hedging strategy, recently introduced in \cite{bayraktar2023nonparametric}. To construct ambiguity sets, \cite{bayraktar2023nonparametric} proposed to use the empirical distribution as a reference measure, i.e.,
 \begin{equation}\label{eq_defn_phat_portfolio_adaptive}
 \begin{aligned}
\Omega^t \ni \omega^t=(\omega_1,\dots,\omega_t) \mapsto \widehat{\PP}^{\rm{ada.}}_t\left(\omega^t\right) (\D x):&=  \frac{1}{N+t}\left(\sum_{s=1}^{N} \delta_{\mathscr{R}_{s}}(\D x)+\sum_{i=1}^{t} \delta_{\omega_i}(\D x)\right)\in \mathcal{M}_1(\Omloc), \\
\widehat{\PP}^{\rm{ada.}}_0(\D x):&= \frac{1}{N}\sum_{s=1}^{N}\delta_{\mathscr{R}_{s}}(\D x)\in \mathcal{M}_1(\Omloc ),
 \end{aligned}
\end{equation}
and then to use a Wasserstein ball around the reference measure \eqref{eq_defn_phat_portfolio_adaptive} with a radius that decreases over time:\begin{equation}\label{eq_P_hedging_adaptive}
\begin{aligned}
\Omega^t \ni \omega^t \mapsto \mathcal{P}^{\rm{ada.}}_t\left({\omega}^{t}\right):=\mathcal{B}_{\varepsilon_t}^{(1)} \left(\widehat{\PP}^{\rm{ada.}}_t(\omega^t)\right)&\text{ for } t \in \{1,\dots,T-1\},\\
 \mathcal{P}^{\rm{ada.}}_0:= \mathcal{B}_{\varepsilon_0}^{(1)} \left(\widehat{\PP}^{\rm{ada.}}_0\right)&\text{ for }t =0.
\end{aligned}
\end{equation}
We specify the explicit choices of $\varepsilon_t$ in Section~\ref{sec:heging1dim} and \ref{sec:heging5dim}, respectively.}

{
\begin{prop}\label{prop_assumptions_exa1}
Assume the framework from Section~\ref{subsec_hedging} and let $p=0$. Then, the ambiguity sets of probability measures defined in \eqref{eq_P_hedging} and \eqref{eq_P_hedging_adaptive} both satisfy Assumption~\ref{asu_P}. Moreover, if the payoff function of the derivative $\Phi: \R^{d(T+1)} \rightarrow \R$ is $\beta$-Hölder-continuous for some $\beta \in (0,1]$, then $\Psi$, defined in \eqref{eq_Psi_hedging}, satisfies Assumption~\ref{asu_psi}.
\end{prop}
}
\subsection{Numerical method} \label{sec_numerics}
To solve our optimization problem \eqref{eq_optimization_problem} numerically, we apply the dynamic programming principle from Theorem~\ref{thm_main_result}.
The numerical routine which is summarized in Algorithm~\ref{algo_1} approximates optimal actions $(a_t^*)_{t=0,\dots,T-1}$ as defined in Theorem~\ref{thm_main_result} by using deep neural networks.

\begin{algorithm}[h!]
\SetAlgoLined
\SetKwInOut{Input}{Input}
\SetKwInOut{Output}{Output}

\Input{Hyperparameters for the neural networks; Number of  iterations $\operatorname{Iter}_\Psi$  for the improvement of each function $\Psi_t$; Number of iterations  $\operatorname{Iter}_a$ for the improvement of the action function; Number of measures  $N_{\mathcal{P}}$; Number  of Monte-Carlo simulations $N_{\operatorname{MC}}$; State space $\Omloc$; {Path-dependent action spaces $(\mathcal{A}_t)_{t=0,\dots,T-1}$ with $\mathcal{A}_t \subseteq \R^{m_t}$ for all $t$}; Payoff function $\Psi$;}
Set ${\mathcal{NN}_{\Psi,T}} \equiv \Psi$;\\
\For{$t=T-1,\dots,0$}{
Initialize a neural network ${\mathcal{NN}_{\Psi,t}} : (\R^d)^t \times (\R^{m_0} \times \cdots\R^{m_{t-1}})  \rightarrow \R $;\\
Initialize a neural network  ${\mathcal{NN}_{a,t}} : (\R^d)^t \times  (\R^{m_0} \times \cdots\R^{m_{t-1}})  \rightarrow {\R^{m_t}} $;\\

\For{$\operatorname{iteration} =1,\dots, \operatorname{Iter}_a $}{
Sample $(\omega_1,\dots,\omega_t)=\omega^t \in \Omega^t$; \\

Sample $(a_0,\dots,a_{t-1})=a^t \in {\mathcal{A}^t(\omega^t)}$;\\
\For{$k=1,\dots,N_{\mathcal{P}}$}{
Sample next states
$w^{t+1, (k),(i)} \sim \mathbb{P}_k \in \mathcal{P}_t(\omega^t)$ for $i= 1,\dots,N_{MC}$
}
Maximize
\begin{equation}\label{eq:max_step_algo1}
\begin{aligned}
&\min_k \frac{1}{N_{MC}}\sum_{i=1}^{N_{MC}} {\mathcal{NN}_{\Psi,t+1}}\left((\omega^t, w^{t+1, (k),(i)}),~\left(a^t,{\mathcal{NN}_{a,t}} (\omega^t,a^t)\right)\right) \\
\end{aligned}
\end{equation}
 w.r.t.\,the parameters of the neural network $\mathcal{NN}_{a,t}$;
}

\For{$\operatorname{iteration} =1,\dots,\operatorname{Iter}_\Psi $}{
Sample $(\omega_1,\dots,\omega_t)=\omega^t \in \Omega^t$; \\

Sample $(a_0,\dots,a_{t-1})=a^t \in {\mathcal{A}^t(\omega^t)}$;\\
\For{$k=1,\dots,N_{\mathcal{P}}$}{
Sample next states
$w^{t+1, (k),(i)} \sim \mathbb{P}_k \in \mathcal{P}_t(\omega^t)$ for $i= 1,\dots,N_{MC}$
}
Minimize 
\begin{equation}\label{eq:min_step_algo1}
\left({\mathcal{NN}_{\Psi,t}}(\omega^t, a^t)- \min_k \frac{1}{N_{MC}}\sum_{i=1}^{N_{MC}} {\mathcal{NN}_{\Psi,t+1}}\left((\omega^t, w^{t+1, (k),(i)}),~\left(a^t, {\mathcal{NN}_{a,t}} (\omega^t,a^t)\right)\right)  \right)^2
\end{equation}
 w.r.t.\,the parameters of the neural network  ${\mathcal{NN}_{\Psi,t}}$;
}
}
Define $a_0: = {\mathcal{NN}_{a}^0} $;
\For {$t=1,\dots,T-1$}{
Define 
\begin{equation}\label{eq:composition_optimal_action}
\Omega^t \ni \omega^t =(\omega_1,\dots,\omega_{t}) \mapsto a_t(\omega^t):= {\mathcal{NN}_{a,t}}\left(\omega^t,~\left(a_0,\dots,a_{t-1}(\omega_1,\dots,\omega_{t-1}\right)\right) \in {\mathcal{A}^t(\omega^t)};
\end{equation}
}
\Output{Actions $(a_t)_{t=0,\dots,T-1}$;}
 \caption{Training of optimal actions}\label{algo_1}
\end{algorithm}

{
The idea of Algorithm~\ref{algo_1} is to approximate recursively the quantities $J_t$ defined in \eqref{eq_defn_J_t} and $\Psi_t$ defined in \eqref{eq_defn_Psi_t} through deep neural networks. To this end, at all times $t=T-1,\dots,0$ one represents $\Psi_{t+1}$ by a neural network, and approximates the expectation $\inf_{\PP \in \mathcal{P}_t(\omega^t)} \E_{\PP}\left[\Psi_{t+1}\left(\omega^t \otimes_t \cdot, a^{t+1}\right)\right]$ via Monte-Carlo simulation, compare \eqref{eq:max_step_algo1}, where the infimum over probability measures is approximated by picking $N_{\mathcal{P}}$ probability measures from the ambiguity set $\mathcal{P}_t$ and then performing  Monte-Carlo simulation for each of the considered measures. The minimum of the $N_{\mathcal{P}}$ Monte-Carlo outcomes is then the resultant approximation of $J_t$ given path-action pairs $(\omega^t, a^{t+1})$. To obtain an approximation of $\Psi_t(\omega^t,a^t):= \sup_{\widetilde{a} \in \mathcal{A}_t} J_t\left(\omega^t, (a^t,\widetilde{a})\right)$, one first parameterizes controls from $\mathcal{A}_t$ by a neural network $ {\mathcal{NN}_{a,t}} (\omega^t,a^t)$, and  optimizes $\widehat{J}_t\left(\omega^t, ~\left(a^{t}, {\mathcal{NN}_{a,t}} (\omega^t,a^t)\right)\right)$ w.r.t.\,the parameters of  $\mathcal{NN}_{a,t}$ on several samples $(\omega^t,a^t)$, e.g., via stochastic gradient descent, where $\widehat{J}_t$ denotes the approximation of $J_t$ computed in \eqref{eq:max_step_algo1}.

In a next step one assigns the computed approximation $\widehat{J}_t\left(\omega^t, ~\left(a^{t}, {\mathcal{NN}_{a,t}} (\omega^t,a^t)\right)\right)$ to a neural network $\mathcal{NN}_{\Psi,t}$. This is done in \eqref{eq:min_step_algo1} via least square minimization on a batch of samples of path-action pairs $(\omega^t, a^t)$.  We leave the specific method for sampling path-action pairs $(\omega^t, a^t)$ to the applicant, as the most efficient approach will depend on the chosen application. A natural, though likely suboptimal, starting point on compact spaces is uniform sampling.

Upon reaching the terminal time step $t=0$, in line with Theorem~\ref{thm_main_result}~(ii),  the optimal action can be computed by combining the local maximizers computed in \eqref{eq:max_step_algo1}, compare \eqref{eq:composition_optimal_action}. \\

If one knows how to effectively select distributions from an ambiguity set $\mathcal{P}_t$, as, e.g., in the parametric case described in Section~\ref{sec_ambiguity_parametric}, then the step where we sample $w^{t+1, (k),(i)} \sim \mathbb{P}_k \in \mathcal{P}_t(\omega^t)$ for $i= 1,\dots,N_{MC}$ 
is straightforward. However, in the often considered situation that the ambiguity set $\mathcal{P}_t$ is defined as a Wasserstein-ball\footnote{While in our numerical experiments we choose $\varepsilon_t(\cdot)$ to be deterministic, we present the algorithm in its most general form.} of size $\varepsilon_t$ centered at a reference measure $\widehat{\PP}_t$ (compare Section~\ref{sec_ambiguity_wasserstein}), this step is not clear, and hence one can refine Algorithm~\ref{algo_1} by applying a well-known duality argument that allows to represent at each time $t$
\begin{equation}\label{eq:wasserstein_dual_numerics}
\begin{aligned}
&\inf_{\PP \in \mathcal{B}_{\varepsilon_t(\omega^t)}^{(q)} \left(\widehat{\PP}_t(\omega^t)\right)} \E_{\PP}\left[\Psi_{t+1}\left(\omega^t \otimes_t \cdot, a^{t+1}\right)\right] \\
&= \sup_{\lambda>0} \E_{\widehat{\PP}_t(\omega^t)}\left[\inf_{z\in \Omloc}\left\{\Psi_{t+1}\left(\omega^t \otimes_t z, a^{t+1}\right)+\lambda \| \cdot - z\|\right\}\right]- \lambda \varepsilon_t(\omega^t)^q,
\end{aligned}
\end{equation}

compare for more details, e.g., \cite{mohajerin2018data}, \cite{gao2023distributionally}, or \cite{bartl2020computational}. This consideration leads to Algorithm~\ref{algo_2} which follows exactly the same idea as Algorithm~\ref{algo_1} but replaces the step involving the sampling of $N_{\mathcal{P}}$ measures by using the above outlined duality leading to a tractable optimization problem involving the optimization over the scalar $\lambda >0$, as well as the minimization $\inf_{z \in \Omloc}$ which can be implemented for compact spaces $\Omloc$ via sampling from a (uniform) grid and then minimizing the corresponding quantities.

\begin{algorithm}[h!]
{
\SetAlgoLined
\SetKwInOut{Input}{Input}
\SetKwInOut{Output}{Output}

\Input{Hyperparameters for the neural networks; Number of  iterations $\operatorname{Iter}_\Psi$  for the improvement of each function $\Psi_t$; Number of iterations  $\operatorname{Iter}_a$ for the improvement of the action function; Number  of Monte-Carlo simulations $N_{\operatorname{MC}}$; State space $\Omloc$; Path-dependent action spaces $(\mathcal{A}_t)_{t=0,\dots,T-1}$ with $\mathcal{A}_t \subseteq \R^{m_t}$ for all $t$; Payoff function $\Psi$; Reference measures $(\widehat{\PP}_t)_{t=0,\dots_{T-1}}$; Pathwise size $\varepsilon_t(\cdot)$ of the Wasserstein-ball}
Set ${\mathcal{NN}_{\Psi,T}} \equiv \Psi$;\\
\For{$t=T-1,\dots,0$}{
Initialize a neural network ${\mathcal{NN}_{\Psi,t}} : (\R^d)^t \times (\R^{m_0} \times \cdots\R^{m_{t-1}})  \rightarrow \R $;\\
Initialize a neural network  ${\mathcal{NN}_{a,t}} : (\R^d)^t \times  (\R^{m_0} \times \cdots\R^{m_{t-1}})  \rightarrow {\R^{m_t}} $;\\

\For{$\operatorname{iteration} =1,\dots, \operatorname{Iter}_a $}{
Sample $(\omega_1,\dots,\omega_t)=\omega^t \in \Omega^t$; \\

Sample $(a_0,\dots,a_{t-1})=a^t \in {\mathcal{A}^t(\omega^t)}$;\\
Sample next states
$\omega^{t+1, (i)} \sim  \widehat{\PP}_t(\omega^t)$ for $i= 1,\dots,N_{MC}$;\\
Sample $z_1,\dots,z_N \in \Omloc$;\\
Maximize
\begin{equation}\label{eq:max_step_algo2_wasserstein}
\begin{aligned}
&\frac{1}{N_{MC}} \sum_{i=1}^{N_{MC}} \left[\min_{j=1,\dots,N} \left\{\Psi_{t+1}\left(\omega^t \otimes_t z_j, (a^{t},{\mathcal{NN}_{a,t}} (\omega^t,a^t))\right)+\lambda \|\omega^{t+1, (i)} -z_j\|\right\}\right] - \lambda \varepsilon_t^q(\omega^t)
\end{aligned}
\end{equation}
 w.r.t.\,the parameters of the neural network $\mathcal{NN}_{a,t}$ and w.r.t.\,the parameter $\lambda>0$;
}
\For{$\operatorname{iteration} =1,\dots,\operatorname{Iter}_\Psi $}{
Sample $(\omega_1,\dots,\omega_t)=\omega^t \in \Omega^t$; \\

Sample $(a_0,\dots,a_{t-1})=a^t \in {\mathcal{A}^t(\omega^t)}$;\\
Sample next states
$\omega^{t+1, (i)} \sim  \widehat{\PP}_t(\omega^t)$ for $i= 1,\dots,N_{MC}$\\
Sample $z_1,\dots,z_N \in \Omloc$;\\
Minimize 
{\small
\begin{equation*}
\left({\mathcal{NN}_{\Psi,t}}(\omega^t, a^t)- \frac{1}{N_{MC}} \sum_{i=1}^{N_{MC}} \left[\min_{j=1,\dots,N} \left\{\Psi_{t+1}\left(\omega^t \otimes_t z_j, (a^{t},{\mathcal{NN}_{a,t}} (\omega^t,a^t))\right)+\lambda \|\omega^{t+1, (i)} -z_j\|\right\}\right] + \lambda \varepsilon_t^q(\omega^t) \right)^2
\end{equation*}
}
 w.r.t.\,the parameters of the neural network  ${\mathcal{NN}_{\Psi,t}}$
}
Define $a_0: = {\mathcal{NN}_{a}^0} $;
\For {$t=1,\dots,T-1$}{
Define 
\begin{equation}\label{eq:composition_optimal_action_wasserstein}
\Omega^t \ni \omega^t =(\omega_1,\dots,\omega_{t}) \mapsto a_t(\omega^t):= {\mathcal{NN}_{a,t}}\left(\omega^t,~\left(a_0,\dots,a_{t-1}(\omega_1,\dots,\omega_{t-1}\right)\right) \in {\mathcal{A}^t(\omega^t)};
\end{equation}
}
}
\Output{Actions $(a_t)_{t=0,\dots,T-1}$;}
 \caption{Training of optimal actions for Wasserstein ambiguity sets.}\label{algo_2}
 }
\end{algorithm}
}
{
Note that Algorithm~\ref{algo_1} can be seen as a finite horizon adaption of the algorithm proposed in \cite[Section 4.4.2]{neufeld2023mdp}. Algorithm~\ref{algo_2} is an adaption of Algorithm~\ref{algo_1} to the specific case where model ambiguity is modeled by Wasserstein-balls, and is similar to the algorithm proposed in \cite[Section 3.1]{bayraktar2022data} where the dynamic programming equation is also solved by a duality argument as in \eqref{eq:wasserstein_dual_numerics}. However, the regression task in \cite{bayraktar2022data} is then solved by Gaussian processes instead of neural networks.
}

{ 
\subsection{Empirical Experiments}\label{sec:experiments}
\subsubsection{Data}\label{sec:data}

To test our approach empirically, we consider a time-series of daily returns of the { stocks of \emph{Google}, \emph{Ebay}, \emph{Amazon}, \emph{Microsoft}, and \emph{Apple}} between beginning of January $2010$ and beginning of February $2020$ to construct data-driven ambiguity sets of probability measures, as described in \eqref{eq_P_hedging} and \eqref{eq_P_hedging_adaptive}, which we use to train our agent. To evaluate the trained strategy, we consider two test periods ranging from beginning of February $2020$ until mid of April $2020$, and between mid of April $2020$ and end of June $2020$, respectively, compare also Figure~\ref{fig_apple_train_test}. { Both periods consist of exactly 50 trading days.}

\begin{figure}[h!]
\begin{center}
\includegraphics[scale=0.5]{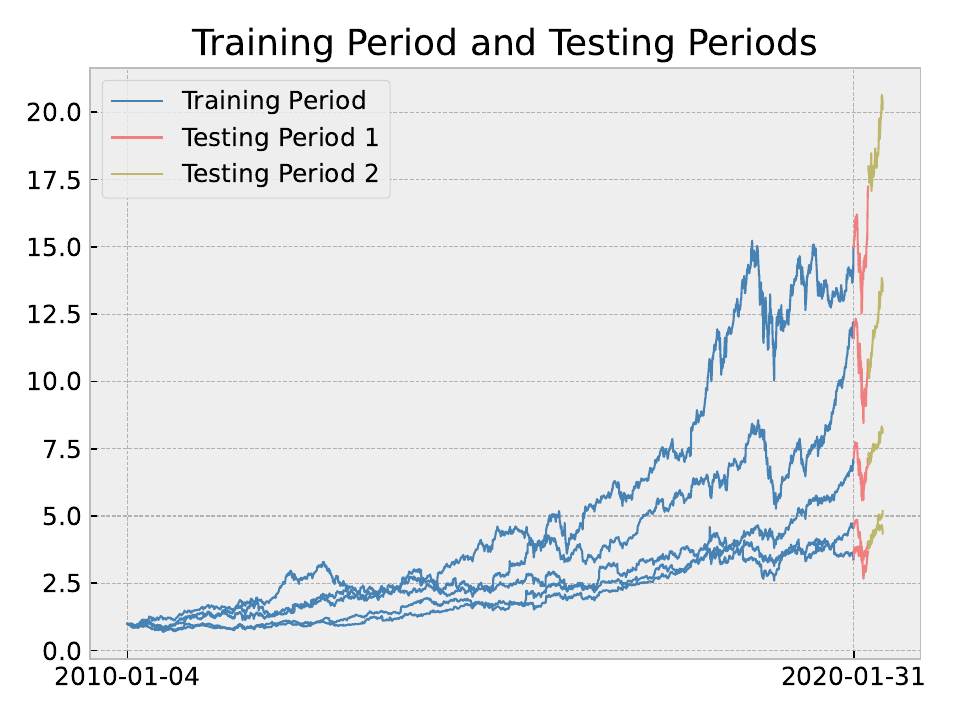}
\end{center}
\caption{The graph depicts the { normalized} evolution of the price of the { stocks of \emph{Google}, \emph{Ebay}, \emph{Amazon}, \emph{Microsoft}, and \emph{Apple}} and the separation of the data into a training period (beginning of January $2010$ until beginning of February $2020$) and the two test periods: beginning of February $2020$ until mid of April $2020$, and  mid of April $2020$ until  end of June $2020$.} \label{fig_apple_train_test}
\end{figure}
\subsubsection{Hedging at-the-money call options}\label{sec:heging1dim}

We set { $d=1$ and $T=10$, which corresponds roughly to an expiration date of two weeks in the future,} and we train { several hedging strategies:
\begin{enumerate}
\item  a non-robust ($\varepsilon=0$) hedging strategy,
\item  a robust hedging strategy with $\varepsilon = 0.0001$ for ambiguity sets defined via \eqref{eq_P_hedging},
\item  a robust hedging strategy with $\varepsilon = 0.0005$ for ambiguity sets defined via \eqref{eq_P_hedging},
\item  a robust adaptive hedging strategy using the approach presented in \cite{bayraktar2023nonparametric}.
\item  a Black-Scholes delta hedging strategy (compare, e.g., \cite{hull1993options}, \cite{hull2017optimal}) where the Black-Scholes volatility is estimated according to historical data.
\end{enumerate}
}
{ The strategies are trained }according to\footnote{To apply Algorithm~\ref{algo_2}, we use the following hyperparameters: 
 Monte-Carlo sample                                    size $N_{\operatorname{MC}}=2^{7}$; number of iterations for $a$: $\operatorname{Iter}_a=500$; number of iterations for $\Psi$: $\operatorname{Iter}_\Psi=2000$, as well as a Batch Size of $128$. The neural networks that approximate $a$ and $\Psi$ constitute of $5$ layers with $32$ neurons each possessing \emph{ReLU} activation functions in each layer, except for the output layers. The learning rate used to optimize the networks $a$ and $\Psi$ when applying the \emph{Adam} optimizer (\cite{kingma2014adam}) is $0.001$. Further details of the implementation can be found under \href{https://github.com/juliansester/Robust-Hedging-Finite-Horizon}{https://github.com/juliansester/Robust-Hedging-Finite-Horizon}.}
 { Algorithm~\ref{algo_2}}  for an at-the-money call option with payoff function 
 $$
{ \Phi(S_0,S_1,\cdots,S_{10})=(S_{10}-S_0)^+,}
 $$
 where we normalize $S_0$ to $1$.

{
To implement the robust adaptive approach from \cite{bayraktar2023nonparametric}, outlined in \eqref{eq_P_hedging_adaptive}, we set $\varepsilon_t = \frac{H^{\alpha}}{\sqrt{N+t}}$ for some constant $H^\alpha\in \R$ which can be computed as the $\alpha$-quantile\footnote{We set $\alpha =0.9$. However any other quantile level could be chosen instead of $90\%$. We decided for comparability to stick to the parameter choice proposed in \cite{bayraktar2023nonparametric}.} of an integral over a Brownian bridge, see \cite[Section 2.1]{bayraktar2023nonparametric} for more details, where it is also outlined that by results from \cite{del1999central}, this choice ensures that 
\begin{equation}
\PP^*(d_{W_1}(\widehat{\PP}^{\rm{ada.}}_t,\PP^*) \leq \varepsilon_t) \geq \alpha
\label{eq:epsilon_adaptive}
\end{equation}
for $\PP^*$ denoting the true but unknown distribution of the realized returns, and where we used for our experiments $\alpha=0.9$ in line with the experiments from \cite{bayraktar2023nonparametric}.
 }

To evaluate the performance of { the different strategies (1)--(5)} we { initiate every day in the respective testing periods a hedging strategy for an at-the-money call option with maturity $T=10$ days. Subsequently, we} evaluate { the realized hedging errors} on testing period 1 and testing period 2 and show the results in Figures~\ref{fig_test_1}, \ref{fig_test_2} and in Table~\ref{tbl:hedging_1dim}. { For a testing period with $50$ trading days, we therefore can observe for each stock the outcomes of $40$ different hedges of at-the-money call options.}

The results show that  during test period 1 (the advent of the Covid-19 pandemic), the { two robust hedging strategies $(\varepsilon=0.0001$ and $\varepsilon=0.0005$)  outperform the non-robust strategy, the classical Black--Scholes delta hedging strategy and also the adaptive strategy. The adaptive strategy on the other hand outperforms the Black-Scholes delta hedge (compare the rightmost column labeled "All Stocks" of Table \eqref{tbl:hedging_1dim}). The robust and adaptive hedging strategies are designed to perform well under the worst-case measure from differently constructed ambiguity sets. We suspect that this explains the relative outperformance of these approaches in comparison to non-robust approaches in adverse financial scenarios for which test period 1 is a prime example. Note that even if the financial scenario does not follow exactly the \emph{worst-case distribution} from the respective ambiguity set, the stability results from Section~\ref{sec_RobustVSNonRobust} ensure that the robust approach performs reasonably well whenever the realized distribution is close enough to the worst-case measure from the ambiguity set of probability measures.}

{ In contrast, in test period 2 the Black--Scholes delta hedging strategy}\footnote{The Black--Scholes delta hedging strategy for a call option with strike $K$ and maturity $t_n$, invests at time $t_i$, the amount $N(d_1)$ in the underlying asset with value $S_{t_i}$, where $N$ denotes the cdf of a standard normal distribution and where $d_1=\frac{1}{\sigma\sqrt{t_n-t_i}} \left[\ln \left(S_{t_i}/K\right)+\frac{1}{2}\sigma^2(t_n-t_i)\right]$. Note that we set the interest rate to be $0$. To apply the strategy, we estimate the annual volatility $\sigma$ from historical data. 
{ In our case we use for Google: $\sigma_{\rm{Google}} \approx 0.24$, for Ebay: $\sigma_{\rm{Ebay}} \approx 0.31$, for Amazon: $\sigma_{\rm{Amazon}} \approx 0.31$, for Microsoft: $\sigma_{\rm{Microsoft}} \approx 0.23$, for Apple: $\sigma_{\rm{Apple}} \approx 0.25$.}} is the best performing strategy. This accounts for the fact that in turbulent times, where the underlying empirical distributions of the asset returns might change dramatically compared to the training period, it turns out to be favorable to take model uncertainty into account as both the non-robust model and the Black--Scholes model rely on misspecified probability distributions. However, in \emph{good weather} periods (as in test period 2), { where the normal distribution underlying the Black-Scholes model is almost not misspecified,} it is hard to beat a strategy such as the Black--Scholes delta hedging strategy. { The worst-case approach that underpins our presented robust hedging approach using ambiguity sets of probability measures is, by the stability results from Section~\ref{sec_RobustVSNonRobust}, designed to perform well if the worst case measure is close to the realized distribution. If this is however not the case there is no performance guarantee for the robust approach leading to suboptimal hedging behavior. This becomes apparent in test period 2. We leave the improvement of robust approaches in such not-misspecified periods for future research.} 

\begin{figure}[h!]
\begin{center}
\includegraphics[scale=0.35]{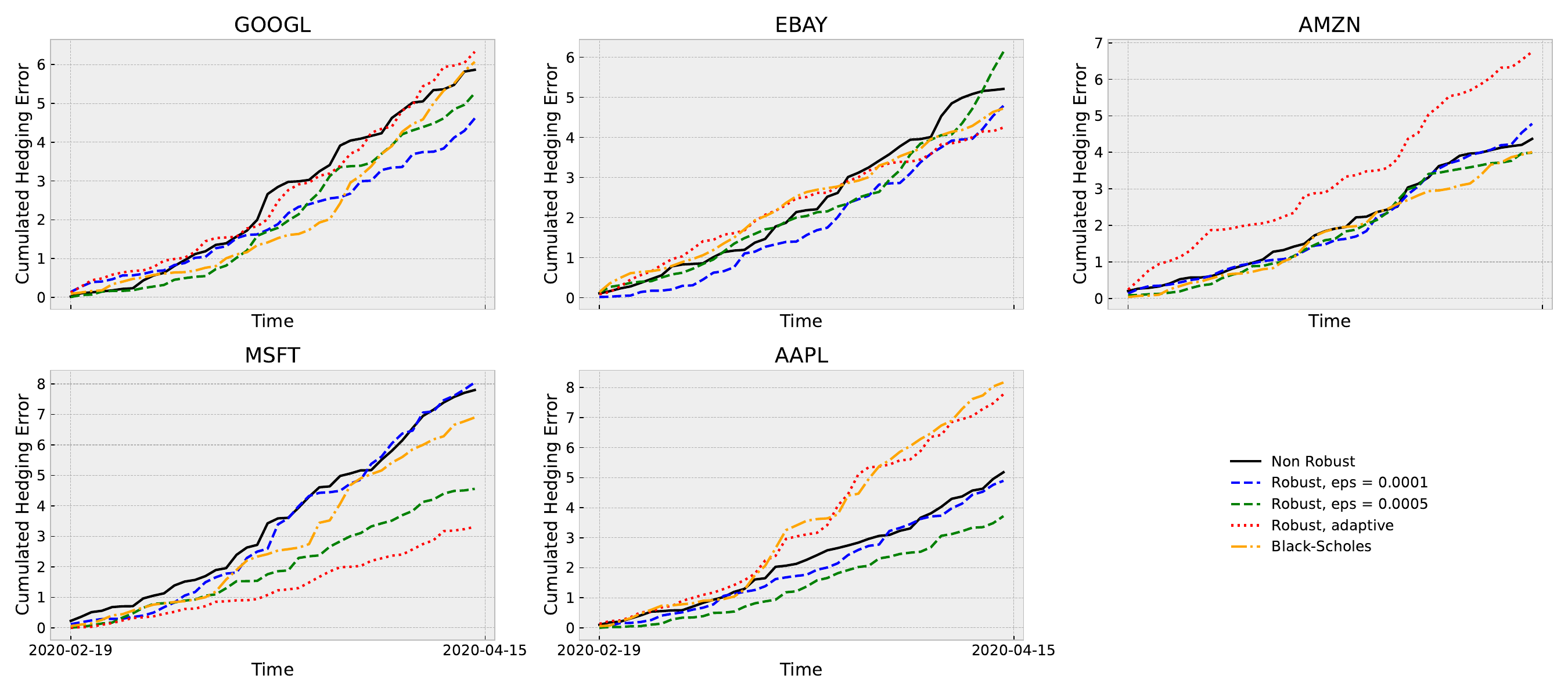}
\includegraphics[scale=0.35]{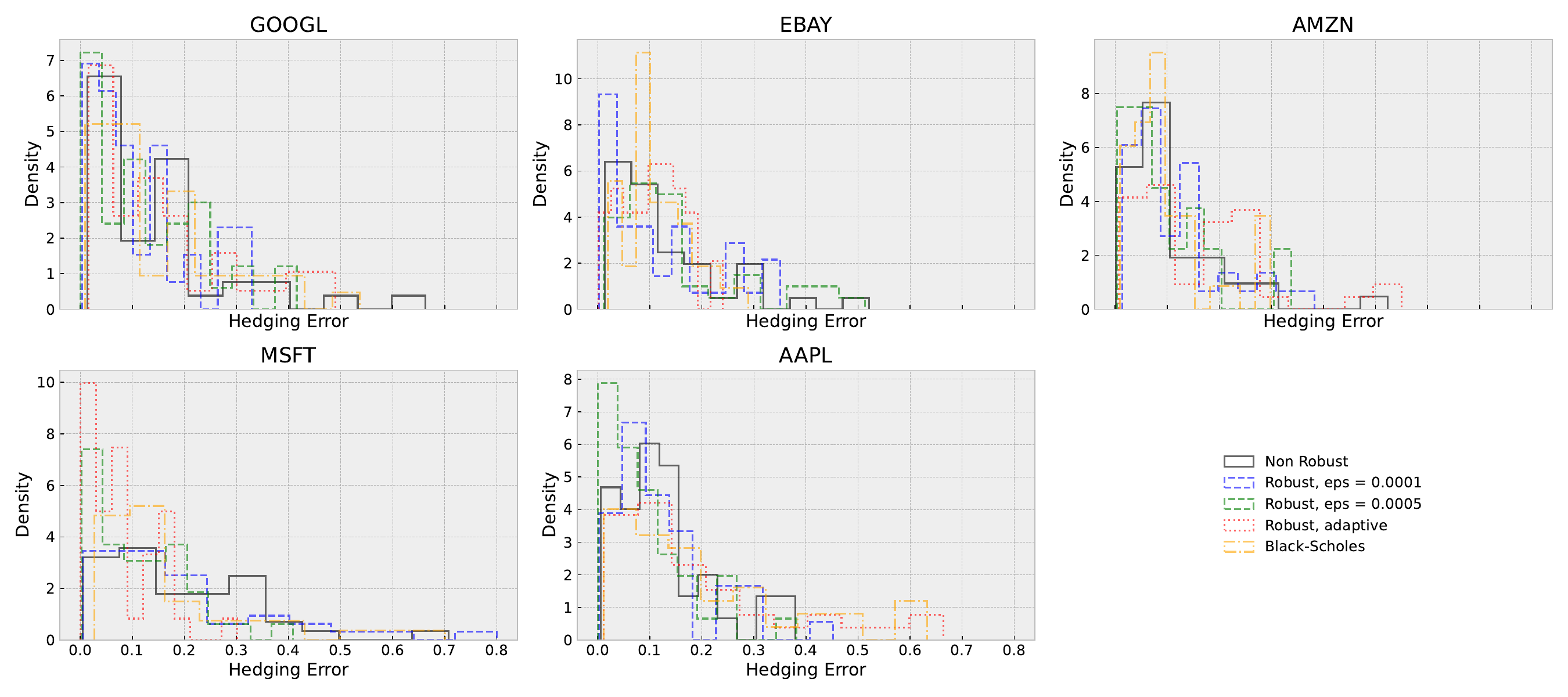}
\end{center}
\caption{{Cumulated hedging errors and error distributions for different strategies during Test Period 1. The top panels show the cumulative hedging errors over time, while the bottom panels display the corresponding histograms of hedging error {distributions obtained over all $40$ hedges.}} } \label{fig_test_1}
\end{figure}

\begin{figure}[h!]
\begin{center}
\includegraphics[scale=0.35]{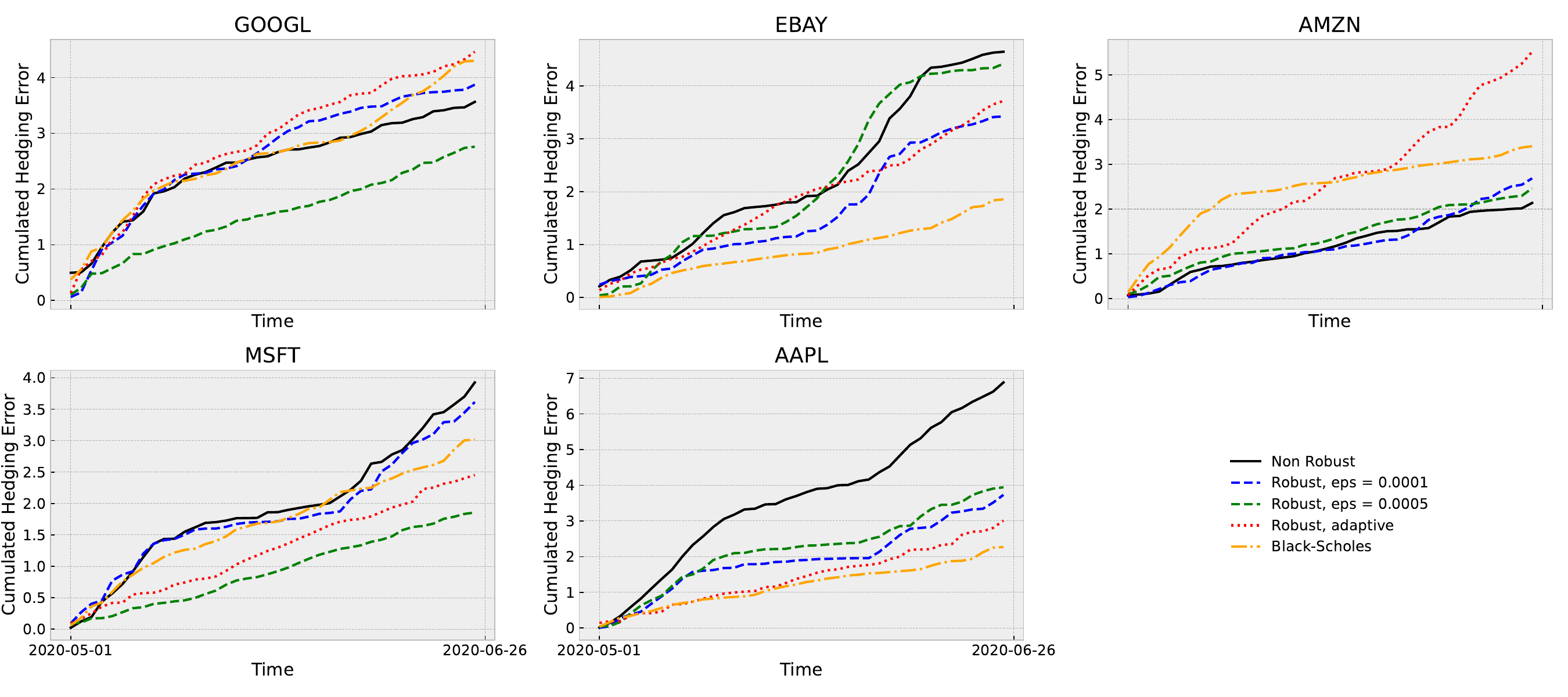}
\includegraphics[scale=0.35]{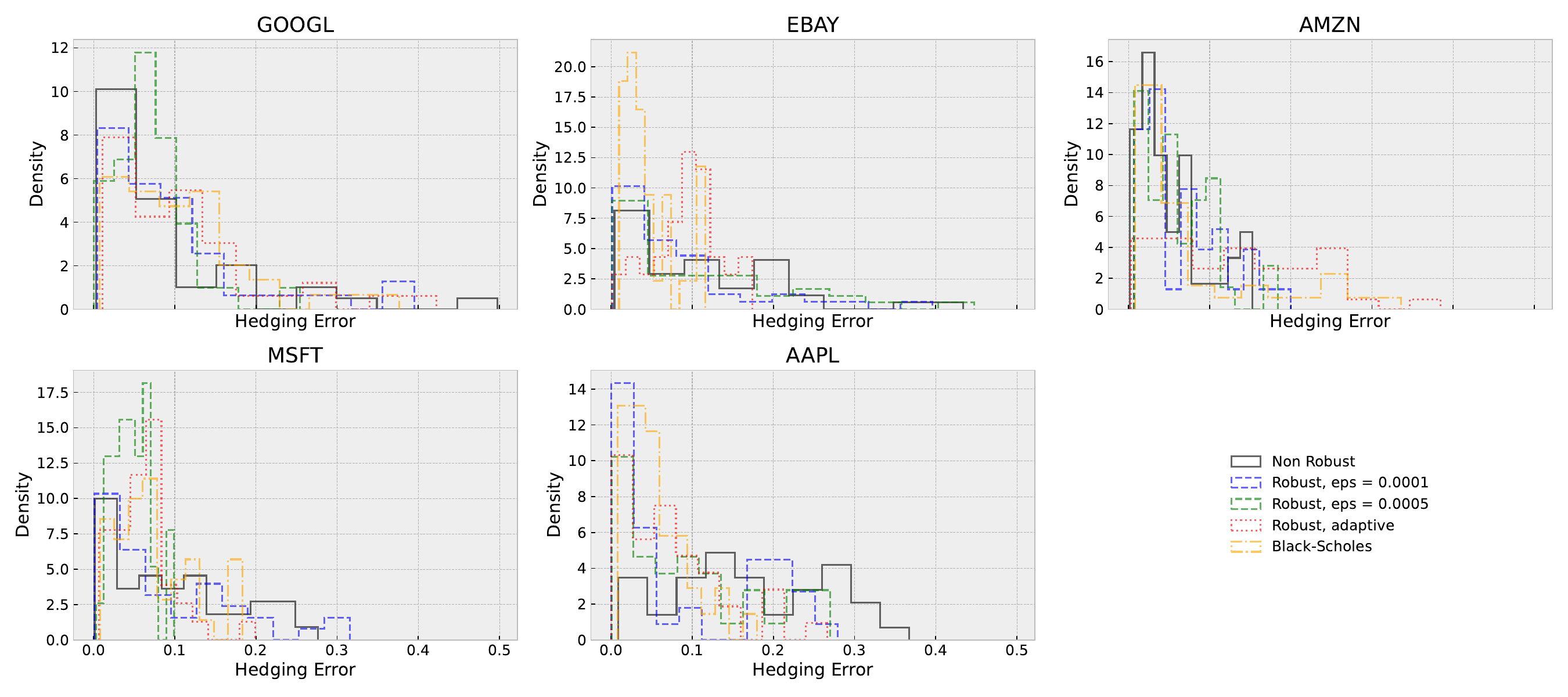}
\end{center}
\caption{{ Cumulated hedging errors and error distributions for different strategies during Test Period 2. The top panels show the cumulative hedging errors over time, while the bottom panels display the corresponding histograms of hedging error {distributions obtained over all $40$ hedges.}}} \label{fig_test_2}
\end{figure}

\begin{table}[htbp]
\centering
\caption{{Descriptive statistics of hedging errors across different strategies and stocks in Test Periods 1 and 2. \emph{count} denotes the number of samples, \emph{mean} the empirical mean, \emph{std dev} the empirical standard deviation,  \emph{min} the minimum, \emph{max} the maximum, {and \emph{25 \%, 50 \%, 75\%} stand for the corresponding quantile levels.}}}\label{tbl:hedging_1dim}
\begin{adjustbox}{width=\textwidth}
\begin{tabular}{llcccccc|cccccc}
\toprule
\multicolumn{2}{c}{} & \multicolumn{6}{c|}{\textbf{Test Period 1}} & \multicolumn{6}{c}{\textbf{Test Period 2}} \\
\cmidrule(lr){3-8} \cmidrule(lr){9-14}
\textbf{Strategy} & \textbf{Statistic} & AAPL & AMZN & EBAY & GOOGL & MSFT & \textbf{All Stocks} & AAPL & AMZN & EBAY & GOOGL & MSFT & \textbf{All Stocks} \\
\midrule
\multirow{7}{*}{Non-Robust}
&count & 40 & 40 & 40 & 40 & 40 & 200 & 40 & 40 & 40 & 40 & 40 &200 \\
& mean & 0.129 & 0.109 & 0.130 & 0.147 & 0.195 & 0.142 & 0.172 & 0.053 & 0.116 & 0.089 & 0.098 &0.105\\
& std & 0.093 & 0.098 & 0.114 & 0.139 & 0.149 & 0.122 & 0.097 & 0.042 & 0.099 & 0.102 & 0.078 &0.094\\
& min & 0.006 & 0.001 & 0.014 & 0.014 & 0.005 & 0.001 & 0.009 & 0.002 & 0.004 & 0.003 & 0.001 &0.001\\
& 25\% & 0.076 & 0.046 & 0.050 & 0.040 & 0.086 & 0.054 & 0.108 & 0.022 & 0.037 & 0.031 & 0.028 &0.028\\
& 50\% & 0.105 & 0.082 & 0.095 & 0.130 & 0.157 & 0.102 & 0.162 & 0.039 & 0.104 & 0.053 & 0.082 &0.076\\
& 75\% & 0.156 & 0.121 & 0.167 & 0.188 & 0.301 & 0.194 & 0.259 & 0.069 & 0.179 & 0.100 & 0.143 &0.153\\
& max & 0.380 & 0.523 & 0.521 & 0.663 & 0.707 & 0.707 & 0.368 & 0.152 & 0.434 & 0.497 & 0.276 &0.497\\
\midrule
\multirow{7}{*}{Robust, eps = 0.0001}
&count & 40 & 40 & 40 & 40 & 40 & 200 & 40 & 40 & 40 & 40 & 40 &200 \\
& mean & 0.122 & 0.120 & 0.120 & 0.115 & 0.201 & 0.135 & 0.093 & 0.067 & 0.086 & 0.097 & 0.090 &0.086\\
& std & 0.094 & 0.090 & 0.107 & 0.096 & 0.177 & 0.121 & 0.093 & 0.049 & 0.089 & 0.094 & 0.084 &0.084\\
& min & 0.003 & 0.014 & 0.002 & 0.004 & 0.005 & 0.002 & 0.000 & 0.007 & 0.001 & 0.004 & 0.001 & 0.001\\
& 25\% & 0.059 & 0.060 & 0.030 & 0.047 & 0.062 & 0.048 & 0.014 & 0.032 & 0.024 & 0.029 & 0.024 & 0.024\\
& 50\% & 0.096 & 0.087 & 0.089 & 0.083 & 0.152 & 0.098 & 0.045 & 0.051 & 0.053 & 0.073 & 0.057 &0.057\\
& 75\% & 0.158 & 0.145 & 0.180 & 0.151 & 0.261 & 0.181 & 0.187 & 0.094 & 0.108 & 0.121 & 0.147 &0.121\\
& max & 0.452 & 0.383 & 0.351 & 0.329 & 0.800 & 0.799 & 0.279 & 0.200 & 0.396 & 0.395 & 0.315 &0.396\\
\midrule
\multirow{7}{*}{Robust, eps = 0.0005}
&count & 40 & 40 & 40 & 40 & 40 & 200 & 40 & 40 & 40 & 40 & 40 &200 \\
& mean & 0.093 & 0.100 & 0.153 & 0.132 & 0.114 & 0.118 & 0.099 & 0.062 & 0.110 & 0.069 & 0.046 &0.077\\
& std & 0.083 & 0.086 & 0.130 & 0.110 & 0.093 & 0.103 & 0.084 & 0.043 & 0.109 & 0.047 & 0.023 &0.071\\
& min & 0.001 & 0.004 & 0.011 & 0.001 & 0.003 & 0.001 & 0.001 & 0.007 & 0.001 & 0.000 & 0.003 &0.000\\
& 25\% & 0.032 & 0.042 & 0.068 & 0.040 & 0.032 & 0.041 & 0.024 & 0.025 & 0.025 & 0.033 & 0.029 &0.027\\
& 50\% & 0.062 & 0.071 & 0.113 & 0.097 & 0.095 & 0.091 & 0.082 & 0.053 & 0.068 & 0.063 & 0.047 &0.056\\
& 75\% & 0.135 & 0.152 & 0.197 & 0.206 & 0.175 & 0.164 & 0.153 & 0.090 & 0.175 & 0.088 & 0.062 &0.099\\
& max & 0.381 & 0.338 & 0.513 & 0.416 & 0.408 & 0.513 & 0.270 & 0.184 & 0.448 & 0.254 & 0.099 &0.448\\
\midrule
\multirow{7}{*}{Robust adaptive}
&count & 40 & 40 & 40 & 40 & 40 & 200 & 40 & 40 & 40 & 40 & 40 &200 \\
& mean & 0.194 & 0.169 & 0.106 & 0.158 & 0.083 & 0.142 & 0.075 & 0.138 & 0.093 & 0.112 & 0.061 &0.095\\
& std & 0.170 & 0.131 & 0.061 & 0.131 & 0.067 & 0.125 & 0.063 & 0.093 & 0.041 & 0.093 & 0.036 &0.074\\
& min & 0.012 & 0.007 & 0.002 & 0.016 & 0.001 & 0.001 & 0.000 & 0.003 & 0.001 & 0.011 & 0.007 &0.000\\
& 25\% & 0.084 & 0.065 & 0.059 & 0.058 & 0.030 & 0.055 & 0.024 & 0.072 & 0.071 & 0.044 & 0.039 &0.042\\
& 50\% & 0.136 & 0.164 & 0.102 & 0.127 & 0.072 & 0.104 & 0.069 & 0.135 & 0.096 & 0.093 & 0.059 &0.077\\
& 75\% & 0.243 & 0.245 & 0.149 & 0.200 & 0.133 & 0.178 & 0.107 & 0.216 & 0.117 & 0.134 & 0.075 &0.125\\
& max & 0.664 & 0.550 & 0.240 & 0.490 & 0.302 & 0.663 & 0.266 & 0.385 & 0.174 & 0.422 & 0.199 &0.422\\
\midrule
\multirow{7}{*}{Black-Scholes}
&count & 40 & 40 & 40 & 40 & 40 & 200 & 40 & 40 & 40 & 40 & 40 &200 \\
& mean & 0.204 & 0.100 & 0.118 & 0.152 & 0.173 & 0.149 & 0.057 & 0.085 & 0.046 & 0.108 & 0.075 &0.074\\
& std & 0.172 & 0.080 & 0.063 & 0.130 & 0.159 & 0.132 & 0.038 & 0.090 & 0.032 & 0.085 & 0.048 &0.066\\
& min & 0.012 & 0.009 & 0.020 & 0.009 & 0.027 & 0.009 & 0.008 & 0.008 & 0.010 & 0.007 & 0.008 &0.007\\
& 25\% & 0.075 & 0.046 & 0.086 & 0.053 & 0.070 & 0.057 & 0.026 & 0.028 & 0.025 & 0.045 & 0.041 &0.027\\
& 50\% & 0.155 & 0.083 & 0.101 & 0.101 & 0.118 & 0.107 & 0.050 & 0.044 & 0.034 & 0.086 & 0.064 &0.051\\
& 75\% & 0.284 & 0.124 & 0.150 & 0.202 & 0.196 & 0.189 & 0.076 & 0.107 & 0.064 & 0.138 & 0.107 &0.102\\
& max & 0.633 & 0.298 & 0.289 & 0.537 & 0.700 & 0.700 & 0.180 & 0.336 & 0.116 & 0.377 & 0.183 &0.377\\
\midrule
\bottomrule
\end{tabular}

\end{adjustbox}
\end{table}

{
\subsubsection{Hedging at-the-money basket options}\label{sec:heging5dim}
Next, we aim at hedging an at-the-money basket option depending on the outcome of the $d=5$ assets (\emph{Google}, \emph{Ebay}, \emph{Amazon}, \emph{Microsoft}, and \emph{Apple}), presented in Section~\ref{sec:data}. This means the payoff function is given by 
$$
 \Phi(S_0,S_1,\cdots,S_{10})=\frac{1}{d}\left(\sum_{i=1}^d S_{10}^i-S_0^i\right)^+.
$$
As in the one-dimensional case, we compare our results with a robust adaptive approach, which was in the multi-dimensional case introduced in \cite{bayraktar2022data}. In the multi-dimensional case Equation \eqref{eq:epsilon_adaptive} is no more applicable with $\varepsilon_t = \frac{H^{\alpha}}{\sqrt{N+t}}$. Instead we choose at all times $t$
\begin{align*}
\widetilde{\varepsilon}_t:= \frac{64}{2.7} \bigg[\gamma^*_t+(N+t)^{-1/2} \bigg( &(C\tfrac{\sqrt{d}}{2}-\gamma^*_t)+\log\left(C \tfrac{\sqrt{d}}{2\gamma^*_t}\right) 2C \sqrt{d} \lceil d/2 \rceil \\
&+  \sum_{k=2}^{\lceil d/2 \rceil} {{\lceil d/2 \rceil} \choose k } (2C \sqrt{d})^k \bigg(\tfrac{\left(C\tfrac{\sqrt{d}}{2}\right)^{1-k}-{\gamma^*_t}^{1-k}}{1-k}\bigg)\bigg)\bigg]
\end{align*}
with $\gamma^*_t = \frac{2C \sqrt{d}}{(N+t)^{1/(2 \lceil d/2 \rceil)}-1}$ and $C = \max_{s=1,\dots,N \atop i = 1,\dots,5} |\mathscr{R}_s^i|$, 
which by Lemma~\ref{lem:epsilon_multidim} ensures, analogue to \eqref{eq:epsilon_adaptive}, that
\begin{equation*}
\PP^*(d_{W_1}(\widehat{\PP}^{\rm{ada.}}_t,\PP^*) \leq \widetilde{\varepsilon}_t) \geq 0.9.
\end{equation*}
for $\PP^*$ denoting the true but unknown distribution of the realized asset returns.

We depict the results of the trained hedging strategies in Figures~\ref{fig_basket_1}, \ref{fig_basket_2} and in Table~\ref{tbl_basket}.
\begin{figure}[h!]
\begin{center}
\includegraphics[scale=0.35]{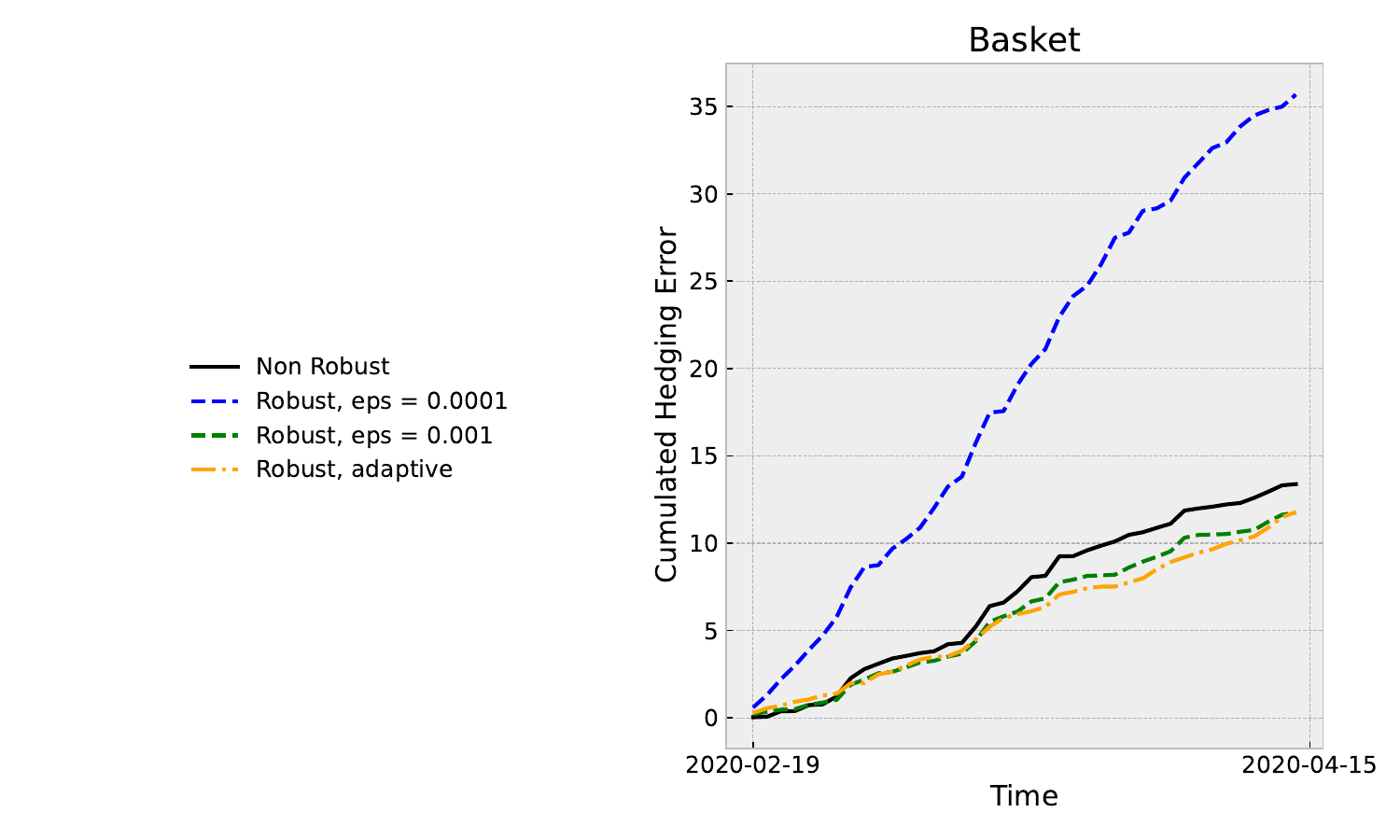}
\includegraphics[scale=0.35]{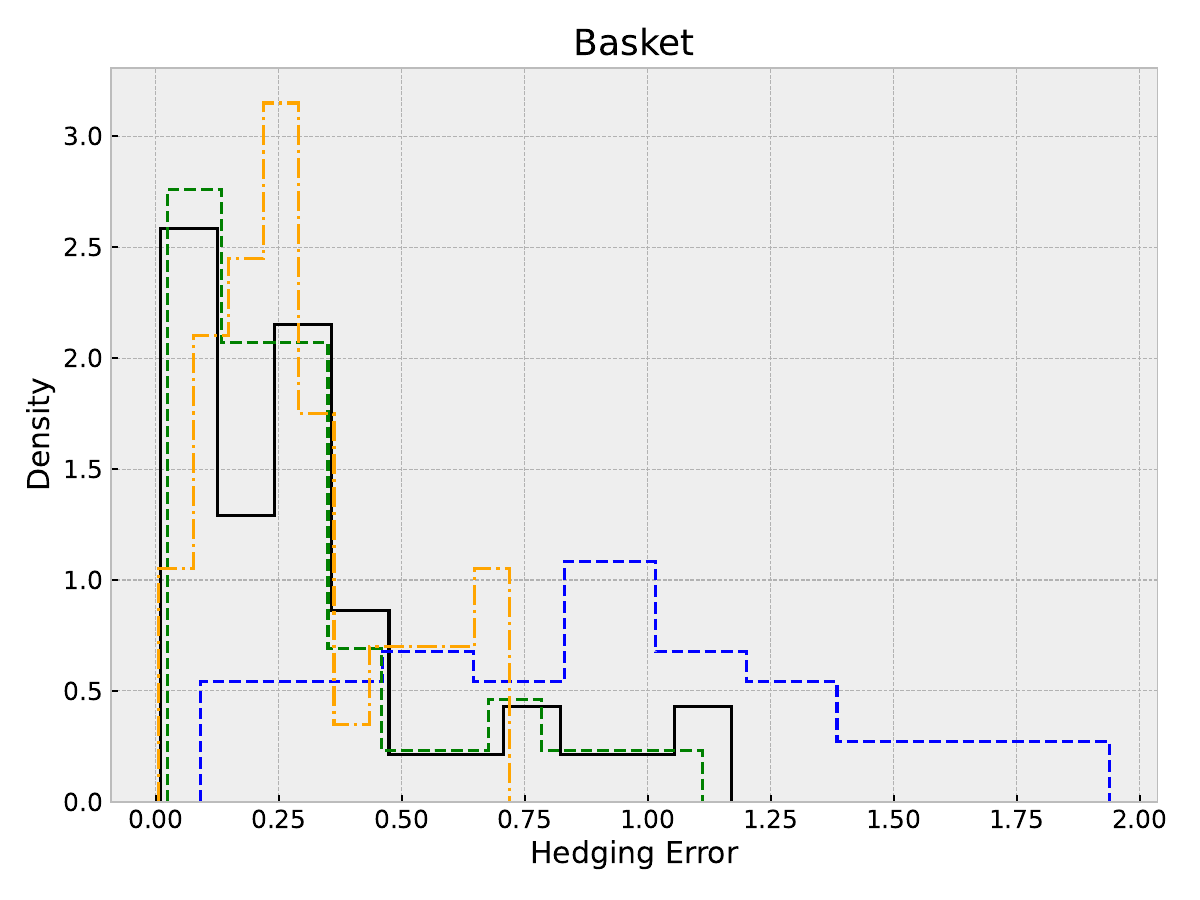}
\end{center}
\caption{{  Cumulated hedging errors and error distributions for different strategies when hedging a basket option during Test Period 1. The left panel shows the cumulative hedging errors over time, while the right panel displays the corresponding histograms {of hedging error distributions obtained over all $40$ hedges.}}} \label{fig_basket_1}
\end{figure}

\begin{figure}[h!]
\begin{center}
\includegraphics[scale=0.35]{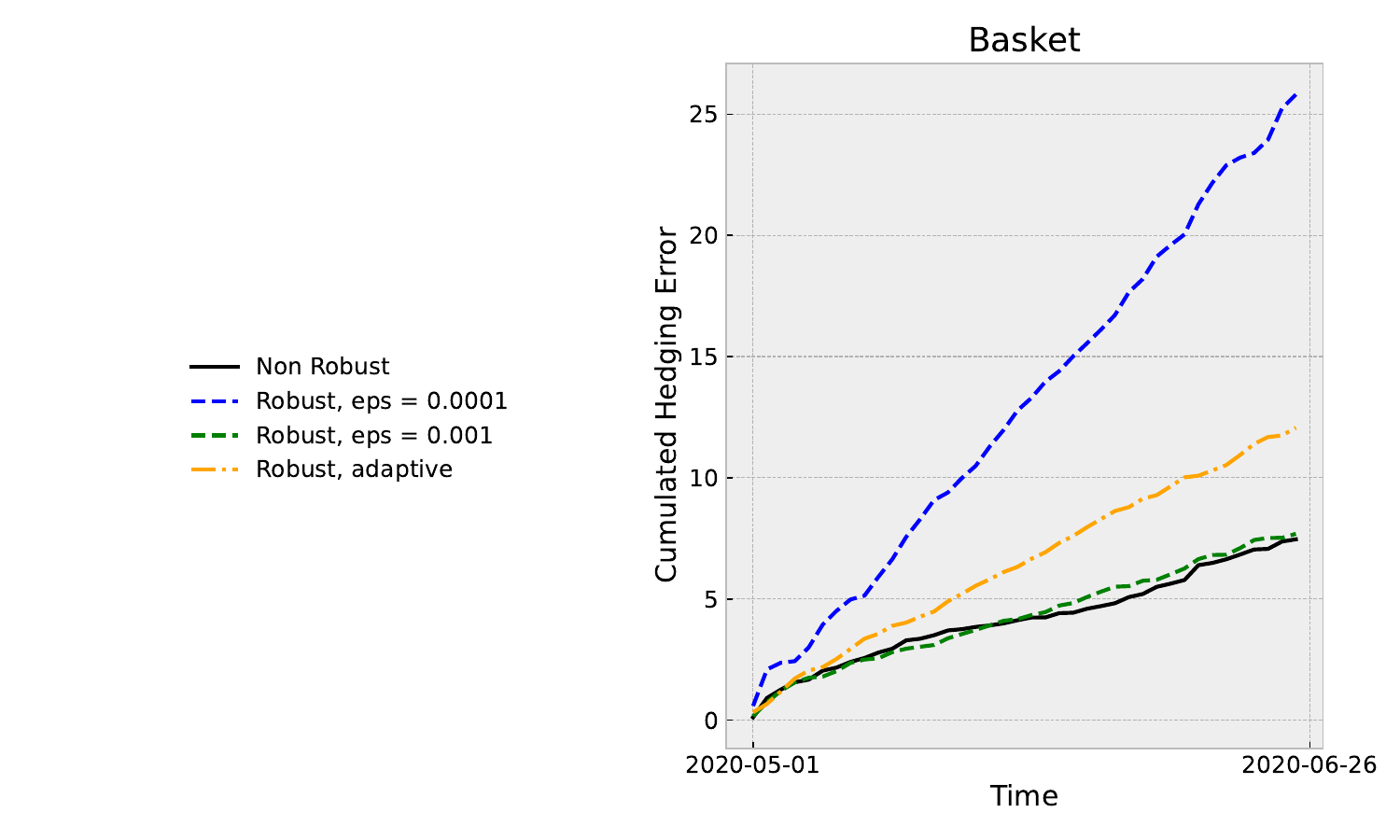}
\includegraphics[scale=0.35]{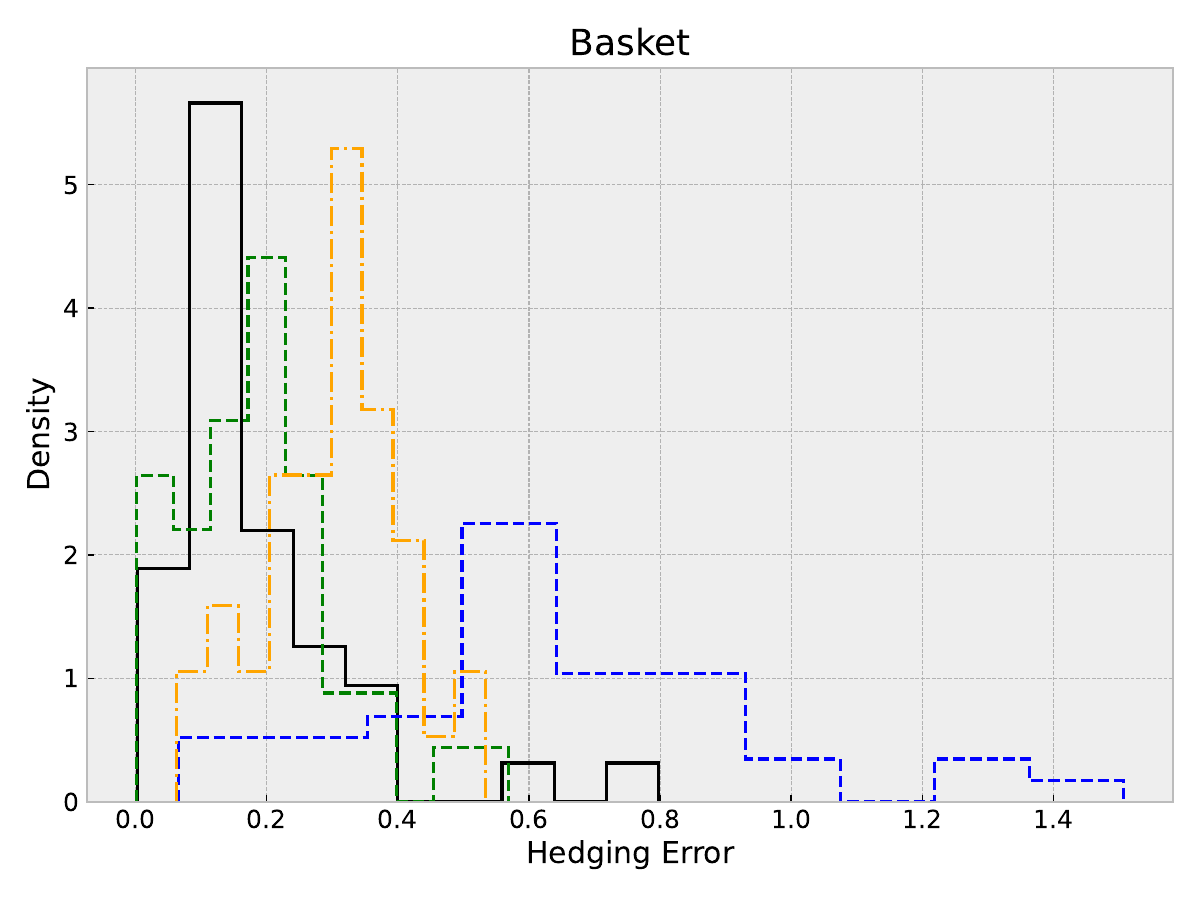}
\end{center}
\caption{{  Cumulated hedging errors and error distributions for different strategies when hedging a basket option during Test Period 2. The left panel shows the cumulative hedging errors over time, while the right panel displays the corresponding histograms {of hedging error distributions obtained over all $40$ hedges.}}} \label{fig_basket_2}
\end{figure}

\begin{table}[ht]
\centering
\caption{ Summary Statistics of Hedging Errors of a Basket option for Test Periods 1 and 2. \emph{count} denotes the number of samples, \emph{mean} the empirical mean, \emph{std dev} the empirical standard deviation,  \emph{min} the minimum, \emph{max} the maximum, and \emph{25 \%, 50 \%, { 75\%}} {stand for the corresponding quantile levels.}}
\label{tbl_basket}
\begin{adjustbox}{width=0.4\textwidth}
\begin{tabular}{llcc}
\toprule
Strategy & Statistic & Test Period 1 & Test Period 2 \\
\midrule
\multirow{8}{*}{Non-Robust} 
& Count     & 40 & 40 \\
& Mean      & 0.335 & 0.186 \\
& Std Dev   & 0.313 & 0.152 \\
& Min       & 0.009 & 0.002 \\
& 25\%      & 0.102 & 0.105 \\
& Median    & 0.258 & 0.148 \\
& 75\%      & 0.382 & 0.227 \\
& Max       & 1.170 & 0.797 \\
\midrule
\multirow{8}{*}{Robust ($\varepsilon = 0.0001$)} 
& Count     & 40 & 40 \\
& Mean      & 0.892 & 0.645 \\
& Std Dev   & 0.485 & 0.294 \\
& Min       & 0.090 & 0.066 \\
& 25\%      & 0.582 & 0.502 \\
& Median    & 0.849 & 0.589 \\
& 75\%      & 1.202 & 0.791 \\
& Max       & 1.939 & 1.507 \\
\midrule
\multirow{8}{*}{Robust ($\varepsilon = 0.001$)} 
& Count     & 40 & 40 \\
& Mean      & {0.294} & {0.192} \\
& Std Dev   & {0.261} & {0.124} \\
& Min       & 0.024 & 0.001 \\
& 25\%      & 0.128 & 0.100 \\
& Median    & 0.226 & 0.180 \\
& 75\%      & 0.338 & 0.250 \\
& Max       & 1.110 & 0.568 \\
\midrule
\multirow{8}{*}{Robust (adaptive)} 
& Count     & 40 & 40 \\
& Mean      & 0.295 & 0.302 \\
& Std Dev   & 0.187 & 0.109 \\
& Min       & 0.005 & 0.062 \\
& 25\%      & 0.164 & 0.221 \\
& Median    & 0.247 & 0.326 \\
& 75\%      & 0.373 & 0.362 \\
& Max       & 0.719 & 0.534 \\
\bottomrule
\end{tabular}
\end{adjustbox}
\end{table}
\noindent Overall, they confirm the results of the one-dimensional setting and show clearly that including uncertainty can improve the hedging performance significantly, in particular, during test period 1, where both the adaptive approach and the robust approach with $\varepsilon=0.001$ perform similar. However, in test period 2 the non-robust approach is the best-performing approach slightly outperforming the robust approach with $\varepsilon=0.001$ and clearly outperforming the robust adaptive approach, due to less model misspecification, analogue to the discussion in the one-dimensional case.
}

\section{Proofs} \label{sec_proofs}

\subsection{Proofs of Section~\ref{sec_dynamic_programming}}
Before presenting the proof of Theorem~\ref{thm_main_result},  we first establish the following crucial lemma which proves the existence of local (i.e.\ one-step) worst-case measures and controls as well as that the regularity of $\Psi$ is back-propagated in the backward recursion of the dynamic programming procedure \eqref{eq_defn_J_t}--\eqref{eq_defn_Psi_0}.
{
\begin{lem}\label{lem_appendix_backwards_iteration}
	Let Assumption~\ref{asu_A} and Assumption~\ref{asu_P} hold, and let $p \in \N_0$ be the integer from Assumption~\ref{asu_psi}.
	 Moreover,
let $t\in \{1,\dots,T-1\}$ and let  $\Psi_{t+1}:\Omega^{t+1} \times \cA^{t+1} \to \R$.
Assume that there exists some $C_{\Psi,t+1}\geq 1$, $\alpha_{t+1}\in (0,1]$, and $L_{\Psi,t+1}\geq 0$ such that for all\footnote{For any $s,t \in \{1,\dots T\}$ with $s-1\leq t$, we use the notation $(\omega^{t},a^s):= \big((\omega_1,\dots,\omega_{t}),(a_0,\dots,a_{s-1})\big)$ for any element $(\omega^{t},a^s)\in \Omega^t\times \cA^s$.}
$(\omega^{t+1},a^{t+1}),
(\widetilde{\omega}^{t+1},\widetilde{a}^{t+1})
\in \Omega^{t+1}\times \mathcal{A}^{t+1}$ 
 we have
\begin{equation}\label{eq_linear_growth_thm_1}
\left|\Psi_{t+1}(\omega^{t+1},a^{t+1})\right| \leq C_{\Psi,t+1} \cdot \left(1+\sum_{i=1}^{t+1} \| \omega_i\|^p\right)
\end{equation}
as well as
\begin{equation}\label{eq_Lipschitz_1_thm_1}
\left|\Psi_{t+1}\left(\omega^{t+1},a^{t+1}\right)-\Psi_{t+1}\left(\widetilde{\omega}^{t+1},\widetilde{a}^{t+1}\right)\right|
 \leq
  L_{\Psi,t+1} \cdot \left(  \sum_{i=1}^{t+1}\left\|\omega_i-\widetilde{\omega}_i\right\|^{\alpha_{t+1}}+\|a_{i-1}-\widetilde{a}_{i-1}\|^{\alpha_{t+1}}\right).
\end{equation}
Then, the following holds.
\begin{itemize}
\item[(i)] There exists a measurable map $\Omega^t \times \cA^{t+1} \ni (\omega^t, a^{t+1}) \mapsto \widetilde{\PP}_t^*(\omega^t,a^{t+1})\in \mathcal{M}_1(\Omloc)$  satisfying for all $(\omega^t,a^{t+1}) \in \Omega^t \times \cA^{t+1}$ that  $\widetilde{\PP}_t^*(\omega^t,a^{t+1}) \in \mathcal{P}_t(\omega^t)$ and 
\[
\E_{\widetilde{\PP}_t^*(\omega^t,a^{t+1})} \left[\Psi_{t+1} \left(\omega^t \otimes_t \cdot, a^{t+1}\right)\right]=\inf_{\PP \in \mathcal{P}_t(\omega^t)} \E_{\PP} \left[\Psi_{t+1} \left(\omega^t \otimes_t \cdot, a^{t+1}\right)\right].
\]
\item[(ii)] There exists a measurable map
$
\Omega^t \times \cA^t \ni (\omega^t,a^t) \mapsto \widetilde{a}_t^* \left(\omega^t, a^t\right) \in \cA_t(\omega^t)
$
satisfying for all $(\omega^t,a^{t}) \in \Omega^t \times \cA^{t}$ that
\begin{equation*}
	\begin{split}
\Psi_t(\omega^t, a^t):=&\sup_{\widetilde{a} \in \cA_t(\omega^t)} \inf_{\PP \in \mathcal{P}_t(\omega^t)} \E_{\PP} \left[\Psi_{t+1} \left(\omega^t \otimes_t \cdot, (a^{t},\widetilde{a})\right)\right]\\
=&\inf_{\PP \in \mathcal{P}_t(\omega^t)} \E_{\PP} \bigg[\Psi_{t+1} \bigg(\omega^t \otimes_t \cdot, ~\left(a^{t},\widetilde{a}_t^* \left(\omega^t, a^t\right)\right)\bigg)\bigg].
\end{split}
\end{equation*}
\item[(iii)]
There exists some 
$C_{\Psi,t}\geq 1$, $\alpha_{t}\in (0,1]$, and $L_{\Psi,t}\geq 0$ such that for all
$(\omega^{t},a^t),(\widetilde{\omega}^{t},\widetilde{a}^{t}) \in \Omega^{t}\times \cA^t$
the following inequalities hold
\begin{align}
\label{eq_linear_growth_thm_2}
&\left|\Psi_{t}(\omega^{t},a^{t})\right| \leq C_{\Psi,t} \cdot \left(1+\sum_{i=1}^{t} \| \omega_i\|^p\right) , \\
\label{eq_Lipschitz_2_thm_1}
&\left|\Psi_{t}\left(\omega^{t},a^{t}\right)-\Psi_t\left(\widetilde{\omega}^{t},\widetilde{a}^{t}\right)\right| 
\leq 
L_{\Psi,t} \cdot \left(  \sum_{i=1}^t\left\|\omega_i-\widetilde{\omega}_i\right\|^{\alpha_{t}}+\|a_{i-1}-\widetilde{a}_{i-1}\|^{\alpha_{t}}\right).
\end{align}
\end{itemize}

\end{lem}
\begin{proof}[Proof of Lemma~\ref{lem_appendix_backwards_iteration}]
  Let $t \in \{1,\dots,T-1\}$.
%
We consider the map
\begin{align*}
F: \left\{(\omega^t,a^t,a,\PP)~\middle|~ (\omega^t,a^t)\in  \Omega^t\times \cA^t,  a \in \cA_t(\omega^t), \PP \in \mathcal{P}_t(\omega^t)\right\} &\rightarrow \R\\
(\omega^t,a^t,a,\PP)&\mapsto \E_{\PP} \left[\Psi_{t+1}\left(\omega^t \otimes_t \cdot, (a^t,a) \right) \right].
\end{align*}
We aim at applying Berge's maximum theorem (see, e.g., \cite{berge} or \cite[Theorem 18.19]{Aliprantis}) to $F$, and therefore first want to show that $F$ is continuous. To that end,
we consider a sequence $(\omega^t_n,a^t_n,a_n,\PP_n)_{n \in \N} \subseteq \left\{(\omega^t,a^t,a,\PP)~\middle|~ (\omega^t,a^t)\in  \Omega^t\times \cA^t,  a \in \cA_t(\omega^t), \PP \in \mathcal{P}_t(\omega^t)\right\}$ with $(\omega^t_n,a^t_n,a_n,\PP_n) \rightarrow (\omega^t,a^t,a,\PP)$ as $n \rightarrow \infty$ for some $(\omega^t,a^t)\in  \Omega^t\times \cA^t$,  $a \in \cA_t(\omega^t)$, $\PP \in \mathcal{P}_t(\omega^t)$. Then, we have
\begin{align*}
&\left|F\left(\omega^t_n,a^t_n,a_n,\PP_n\right)-F\left(\omega^t,a^t,a,\PP\right)\right|\\
&\leq \left|F\left(\omega^t_n,a^t_n,a_n,\PP_n\right)-F\left(\omega^t,a^t,a,\PP_n\right)\right|+ \left|F\left(\omega^t,a^t,a,\PP_n\right)-F\left(\omega^t,a^t,a,\PP\right)\right|.
\end{align*}
Note that with $\PP_n \rightarrow \PP$ in $\tau_p$ we obtain that $\lim_{n \rightarrow \infty} \left|F\left(\omega^t,a^t,a,\PP_n\right)-F\left(\omega^t,a^t,a,\PP\right)\right| = 0$ as $\Psi_{t+1}$ is by assumptions \eqref{eq_linear_growth_thm_1} and \eqref{eq_Lipschitz_1_thm_1} continuous and of polynomial growth of order $p$. Further, we use \eqref{eq_Lipschitz_1_thm_1} and compute\footnote{We use here the notation $\omega^t_n=\big((\omega_1)_n,\dots,(\omega_t)_n\big)$ and $a^t_n=\big((a_0)_n,\dots,(a_{t-1})_n\big)$.}
\begin{align*}
&\lim_{n \rightarrow \infty} \left|F\left(\omega^t_n,a^t_n,a_n,\PP_n\right)-F\left(\omega^t,a^t,a,\PP_n\right)\right|\\
&\leq \lim_{n \rightarrow \infty} \int_{\Omloc} \left| \Psi_{t+1}\left((\omega^t_n,\omega),(a^t_n,a_n)\right)- \Psi_{t+1}\left((\omega^t,\omega),(a^t,a)\right)\right| \PP_n(\D \omega)\\
&\leq \lim_{n \rightarrow \infty} \int_{\Omloc}  L_{\Psi,t+1}  \cdot \left( \sum_{i=1}^t \bigg[\|(\omega_i)_n-\omega_i\|^{\alpha_{t+1}}+\|(a_{i-1})_{n}-a_{i-1}\|^{\alpha_{t+1}}\bigg]+\|a_n-a\|^{\alpha_{t+1}}\right)\PP_n(\D \omega)\\
&=L_{\Psi,t+1}  \int_{\Omloc}  1 \cdot \PP(\D \omega)\cdot \lim_{n \rightarrow \infty} \left( \sum_{i=1}^t \bigg[\|(\omega_i)_n-\omega_i\|^{\alpha_{t+1}}+\|(a_{i-1})_{n}-a_{i-1}\|^{\alpha_{t+1}}\bigg]+\|a_n-a\|^{\alpha_{t+1}}\right)= 0.
\end{align*}
We have thus shown that $F$ is continuous. Therefore, by using Assumption~\ref{asu_P}~(i), we may now apply Berge's maximum theorem which yields that
\begin{equation}\label{eq_Berge_argument_PP}
\left\{(\omega^t,a^t,a)~\middle|~ (\omega^t,a^t)\in  \Omega^t\times \cA^t,  a \in \cA_t(\omega^t)\right\}
\ni (\omega^t,a^t,a) \mapsto \inf_{\PP \in \mathcal{P}_t(\omega^t)} \E_{\PP} \left[\Psi_{t+1} \left(\omega^t \otimes_t \cdot, (a^t,a)\right)\right]
\end{equation}
is continuous. Note that the above application of Berge's maximum theorem also implies the existence of a minimizer of \eqref{eq_Berge_argument_PP}. With the measurable maximum theorem (see, e.g., \cite[Theorem 18.19]{Aliprantis}) and the notation \eqref{notation_graph}, we therefore obtain a measurable map  $\Omega^t \times \cA^{t+1} \ni (\omega^t, a^{t+1}) \mapsto \widetilde{\PP}_t^*(\omega^t,a^{t+1})\in \mathcal{P}_t(\omega^t)$ satisfying for all $(\omega^t,a^{t+1}) \in \Omega^t \times \cA^{t+1}$ that 
\[
\E_{\widetilde{\PP}_t^*(\omega^t,a^{t+1})} \left[\Psi_{t+1} \left(\omega^t \otimes_t \cdot, a^{t+1}\right)\right]=\inf_{\PP \in \mathcal{P}_t(\omega^t)} \E_{\PP} \left[\Psi_{t+1} \left(\omega^t \otimes_t \cdot, a^{t+1}\right)\right],
\]
showing (i).\\

Next, using the continuity of the map \eqref{eq_Berge_argument_PP} together with Assumption~\ref{asu_A}~(i) allows us again to apply 
Berge's maximum theorem to obtain that
\begin{equation} \label{eq_Berge_argument_1}
\Omega^t \times \cA^t \ni (\omega^t,a^t) \mapsto \sup_{a\in \cA_t(\omega^t)}\inf_{\PP \in \mathcal{P}_t(\omega^t)} \E_{\PP} \left[\Psi_{t+1} \left(\omega^t \otimes_t \cdot, (a^t,a)\right)\right]=\Psi_t(\omega^t,a^t)
\end{equation}
is continuous and that a maximizer exists. By the measurable maximum theorem  we deduce therefore the existence of a measurable map
$
\Omega^t \times \cA^t \ni (\omega^t,a^t) \mapsto \widetilde{a}_t^* \left(\omega^t, a^t\right) \in \cA_t(\omega^t)
$
satisfying for all $(\omega^t,a^{t}) \in \Omega^t \times \cA^{t}$ that
\[
\sup_{\widetilde{a} \in \cA_t(\omega^t)} \inf_{\PP \in \mathcal{P}_t(\omega^t)} \E_{\PP} \left[\Psi_{t+1} \left(\omega^t \otimes_t \cdot, (a^{t},\widetilde{a})\right)\right]=\inf_{\PP \in \mathcal{P}_t(\omega^t)} \E_{\PP} \bigg[\Psi_{t+1} \bigg(\omega^t \otimes_t \cdot, ~\left(a^{t},\widetilde{a}_t^* \left(\omega^t, a^t\right)\right)\bigg)\bigg],
\]
which shows (ii).

It remains to prove (iii). We assume that \eqref{eq_linear_growth_thm_1} holds and show the polynomial growth condition stated in \eqref{eq_linear_growth_thm_2}.
 To this end, we consider arbitrary elements $(\omega^t,a^t), (\widetilde{\omega}^{t},\widetilde{a}^t)  \in \Omega^t\times \cA^t$,
  we set $C_{\Psi,t}:= 2 C_{\Psi,t+1}\cdot C_{\mathcal{P},t}\geq 1$ and obtain by \eqref{eq_asu_P_ineq_1} and \eqref{eq_linear_growth_thm_1} that
\begin{align*}
\left|\Psi_t(\omega^t,a^t)\right|&\leq  \sup_{\widetilde{a} \in \cA_t(\omega^t) } \inf_{\PP \in \mathcal{P}_t(\omega^t)} \int_{\Omloc} \left|\Psi_{t+1}\left((\omega^t,\omega),(a^t,\widetilde{a})\right)\right|\PP(\D \omega)\\
&\leq  
\inf_{\PP \in \mathcal{P}_t(\omega^t)} \int_{\Omloc} C_{\Psi,t+1} \left(1+ \sum_{i=1}^t \|\omega_i\|^p + \|\omega\|^p \right) \PP( \D \omega) \\
&=  C_{\Psi,t+1} \left(1+ \sum_{i=1}^t \|\omega_i\|^p +\inf_{\PP \in \mathcal{P}_t(\omega^t)} \int_{\Omloc} \|\omega\|^p \,\PP( \D \omega) \right)\\
&\leq  C_{\Psi,t+1} \left(1+\sum_{i=1}^t \|\omega_i\|^p +C_{\mathcal{P},t}\left(1+ \sum_{i=1}^t\|\omega_i\|^p\right)\right)\\
&\leq  2 C_{\Psi,t+1} \cdot C_{\mathcal{P},t}\cdot \left(1+ \sum_{i=1}^t \|\omega_i\|^p\right) =C_{\Psi,t}\cdot \left(1+ \sum_{i=1}^t \|\omega_i\|^p\right),
\end{align*}
which indeed proves \eqref{eq_linear_growth_thm_2}. 
Next, we assume that \eqref{eq_Lipschitz_1_thm_1} holds and aim at showing \eqref{eq_Lipschitz_2_thm_1}. We first compute
\begin{equation}\label{eq_proof_main_thm_ineq_Psi_t_1}
\begin{aligned}
&\Psi_t\left(\omega^t,a^t\right)-\Psi_t\left(\widetilde{\omega}^{t},\widetilde{a}^t\right) \\
&= \sup_{a\in \cA_t(\omega^t)} \inf_{\PP\in \mathcal{P}_t(\omega^{t})} \E_{\PP}\left[\Psi_{t+1}\left(\omega^{t}\otimes_t \cdot ,(a^t,a)\right)\right]
-
\sup_{\widetilde{a}\in \cA_t(\widetilde{\omega}^{t})} \inf_{\widetilde{\PP}\in \mathcal{P}_t(\widetilde{\omega}^{t})} \E_{\widetilde{\PP}}\left[\Psi_{t+1}\left(\widetilde{\omega}^{t}\otimes_t \cdot ,(\widetilde{a}^t,\widetilde{a})\right)\right] \\
&=\inf_{\PP\in \mathcal{P}_t(\omega^{t})} \E_{\PP}\left[\Psi_{t+1}\left(\omega^{t}\otimes_t \cdot ,(a^t,\aloc^*)\right)\right]
-
\sup_{\widetilde{a}\in \cA_t(\widetilde{\omega}^{t})} \inf_{\widetilde{\PP}\in \mathcal{P}_t(\widetilde{\omega}^{t})} \E_{\widetilde{\PP}}\left[\Psi_{t+1}\left(\widetilde{\omega}^{t}\otimes_t \cdot ,(\widetilde{a}^t,\widetilde{a})\right)\right] \\
&\leq \inf_{\PP\in \mathcal{P}_t(\omega^{t})} \E_{\PP}\left[\Psi_{t+1}\left(\omega^{t}\otimes_t \cdot ,(a^t,\aloc^*)\right)\right]
-
\inf_{\widetilde{\PP}\in \mathcal{P}_t(\widetilde{\omega}^{t})} \E_{\widetilde{\PP}}\left[\Psi_{t+1}\left(\widetilde{\omega}^{t}\otimes_t \cdot ,(\widetilde{a}^t,\widetilde{\aloc}^*))\right)\right] \\ 
 &= \inf_{\PP\in \mathcal{P}_t(\omega^{t})} \left[\Psi_{t+1}\left(\omega^{t}\otimes_t \cdot ,(a^t,\aloc^*)\right)\right]
 - 
 \E_{\PP_1}\left[\Psi_{t+1}\left(\widetilde{\omega}^{t}\otimes_t \cdot ,(\widetilde{a}^t,\widetilde{\aloc}^*)\right)\right]\\
&\leq  \E_{\PP_2}\left[\Psi_{t+1}\left(\omega^{t}\otimes_t \cdot ,(a^t,\aloc^*)\right)\right]
- 
\E_{\PP_1}\left[\Psi_{t+1}\left(\widetilde{\omega}^{t}\otimes_t \cdot ,(\widetilde{a}^t,\widetilde{\aloc}^*)\right)\right],
\end{aligned}
\end{equation}
where $\aloc^*:= \widetilde{a}_t^*(\omega^t,a^t) \in \cA_t(\omega^t)$ with $\widetilde{a}_t^*$ denoting the minimizer from (ii) 
and  $\widetilde{\aloc}^*\in \cA_t(\widetilde{\omega}^{t})$ is chosen such that inequality~\eqref{eq_condition_A} in Assumption~\ref{asu_A}~(ii) is fulfilled w.r.t.\,$\aloc^*$, 
and where 
$\PP_1:=\widetilde{\PP}_t^*(\widetilde{\omega}^t,\widetilde{\aloc}^*)\in \mathcal{P}_t(\widetilde{\omega}^t)$ with $\widetilde{\PP}_t^*$ being the maximizer from (i) and  $\PP_2 \in \mathcal{P}_t({\omega}^t)$ is chosen such that inequality~\eqref{eq_condition_P} in Assumption~\ref{asu_P}~(iii) is fulfilled w.r.t.\,$\PP_1$.
Then, we denote by $\Pi(\PP_1,\PP_2) \subset \mathcal{P}_1(\Omloc \times \Omloc)$ the set of probability measures on $\Omloc \times \Omloc$  with respective marginal distributions $\PP_1$ and $\PP_2$. We use the representation from \eqref{eq_proof_main_thm_ineq_Psi_t_1}, apply the assumption from \eqref{eq_Lipschitz_1_thm_1}, and Jensen's inequality to obtain
\begin{align*}
&\Psi_t\left(\omega^t,a^t\right)
-
\Psi_t\left(\widetilde{\omega}^{t},\widetilde{a}^t\right) \\
&\leq  
\inf_{\pi \in \Pi(\PP_1,\PP_2)} \int_{\Omloc \times \Omloc} \bigg(\Psi_{t+1}\left(\omega^{t}\otimes_t \omega_{1, \operatorname{loc}} ,(a^t,\aloc^*)\right)\\
&\hspace{3cm}-\Psi_{t+1}\left(\widetilde{\omega}^{t}\otimes_t \omega_{2, \operatorname{loc}},(\widetilde{a}^t,\widetilde{\aloc}^*)\right)\bigg)\,\pi(\D \omega_{1, \operatorname{loc}}, \D \omega_{2, \operatorname{loc}})\\
&\leq 
 L_{\Psi,t+1} \cdot \left(  \sum_{i=1}^t\big(\left\|\omega_i-\widetilde{\omega}_i\right\|^{\alpha_{t+1}}+\|a_{i-1}-\widetilde{a}_{i-1}\|^{\alpha_{t+1}}\big)
+ \|\aloc^*-\widetilde{\aloc}^* \|^{\alpha_{t+1}} \right) \\
&\hspace{3cm}+L_{\Psi,t+1} \inf_{\pi \in \Pi(\PP_1,\PP_2)}  \int_{\Omloc \times \Omloc} \|\omega_{1, \operatorname{loc}}-\omega_{2, \operatorname{loc}}\|^{\alpha_{t+1}}\,\pi(\D \omega_{1, \operatorname{loc}}, \D \omega_{2, \operatorname{loc}})\\
&\leq 
 L_{\Psi,t+1} \cdot \left(  \sum_{i=1}^t\big(\left\|\omega_i-\widetilde{\omega}_i\right\|^{\alpha_{t+1}}+\|a_{i-1}-\widetilde{a}_{i-1}\|^{\alpha_{t+1}}\big)
+ \|\aloc^*-\widetilde{\aloc}^* \|^{\alpha_{t+1}} \right) \\
&\hspace{3cm}+L_{\Psi,t+1} \bigg(\inf_{\pi \in \Pi(\PP_1,\PP_2)}  \int_{\Omloc \times \Omloc} \|\omega_{1, \operatorname{loc}}-\omega_{2, \operatorname{loc}}\|\,\pi(\D \omega_{1, \operatorname{loc}}, \D \omega_{2, \operatorname{loc}}) \bigg)^{\alpha_{t+1}}\\
&= 
 L_{\Psi,t+1} \cdot \left(  \sum_{i=1}^t\big(\left\|\omega_i-\widetilde{\omega}_i\right\|^{\alpha_{t+1}}+\|a_{i-1}-\widetilde{a}_{i-1}\|^{\alpha_{t+1}}\big)\right)
+L_{\Psi,t+1}\|\aloc^*-\widetilde{\aloc}^* \|^{\alpha_{t+1}}
+L_{\Psi,t+1} \big(\operatorname{d}_{W_1}(\PP_1,\PP_2)\big)^{\alpha_{t+1}},
\end{align*}
where   in the last step we applied the definition of the $1$-Wasserstein distance. We now use Assumption~\ref{asu_P}~(iii) and Assumption~\ref{asu_A}~(ii), set 
$L_{\Psi,t}:= 2 L_{\Psi,t+1} \max\left\{L_{\mathcal{A},t}^{\alpha_{t+1}}+L_{\mathcal{P},t}^{\alpha_{t+1}},1\right\}\geq 0$, 
$\alpha_t:=\alpha_{t+1}\in (0,1]$,  
and obtain
\begin{align*}
&\Psi_t\left(\omega^t,a^t\right)-\Psi_t\left(\widetilde{\omega}^{t},\widetilde{a}^t\right) \\
&\leq
 L_{\Psi,t+1} \cdot \left(  \sum_{i=1}^t\left\|\omega_i-\widetilde{\omega}_i\right\|^{\alpha_{t+1}}+\|a_{i-1}-\widetilde{a}_{i-1}\|^{\alpha_{t+1}}\right)
 +L_{\Psi,t+1}\big(L_{\mathcal{A},t}^{\alpha_{t+1}}+L_{\mathcal{P},t}^{\alpha_{t+1}}\big)\left(\sum_{i=1}^t \|\omega_i- \widetilde{\omega}_t\|\right)^{\alpha_{t+1}}\\
&\leq 
L_{\Psi,t+1} \cdot \left(  \sum_{i=1}^t\left\|\omega_i-\widetilde{\omega}_i\right\|^{\alpha_{t+1}}
+\|a_{i-1}-\widetilde{a}_{i-1}\|^{\alpha_{t+1}}\right)
+L_{\Psi,t+1}\big(L_{\mathcal{A},t}^{\alpha_{t+1}}+L_{\mathcal{P},t}^{\alpha_{t+1}}\big)\left(\sum_{i=1}^t \|\omega_i- \widetilde{\omega}_t\|^{\alpha_{t+1}}\right)\\
&\leq 
L_{\Psi,t} \cdot \left(  \sum_{i=1}^t\left\|\omega_i-\widetilde{\omega}_i\right\|^{\alpha_{t}}+\|a_{i-1}-\widetilde{a}_{i-1}\|^{\alpha_{t}}\right).
\end{align*}
By interchanging the roles of $\Psi_t\left(\omega^t,a^t\right)$ and $\Psi_t\left(\widetilde{\omega}^{t},\widetilde{a}^t\right) $ we obtain
\[
\left|\Psi_t\left(\omega^t,a^t\right)-\Psi_t\left(\widetilde{\omega}^{t},\widetilde{a}^t\right)\right|\leq L_{\Psi,t} \cdot \left(  \sum_{i=1}^t\left\|\omega_i-\widetilde{\omega}_i\right\|^{\alpha_{t}}+\|a_{i-1}-\widetilde{a}_{i-1}\|^{\alpha_{t}}\right).
\]
 This shows \eqref{eq_Lipschitz_2_thm_1}.
\end{proof}
}

As a consequence of Lemma~\ref{lem_appendix_backwards_iteration} we obtain the following corollary for the case $t=0$.

\begin{cor}\label{cor_appendix_backwards_iteration}
Let  $\Omloc \times \cA_0 \ni (\omega,a)\mapsto \Psi_{1}(\omega,a) \in \R$ 
Assume that there exists some $C_{\Psi,1}\geq 1$,
{ $\alpha_1 \in (0,1]$,} and $L_{\Psi,1}\geq 0$ such that for all $(\omega,a), (\widetilde{\omega},\widetilde{a}) \in \Omloc\times \cA_0$  we have
\begin{equation}\label{eq_linear_growth_cor_1}
\left|\Psi_{1}(\omega,a)\right| \leq C_{\Psi,1} \cdot \left(1+ \| \omega\|^p\right) 
\end{equation}
as well as
\begin{equation}\label{eq_Lipschitz_1_cor_1}
{ 
\left|\Psi_{1}\left(\omega,a\right)-\Psi_{1}\left(\widetilde{\omega},\widetilde{a}\right)\right| \leq L_{\Psi,1} \left(||\omega-\widetilde{\omega}\|^{\alpha_1 }+\|a-\widetilde{a}\|^{\alpha_1}\right).
}
\end{equation}
Then, the following holds.
\begin{itemize}
\item[(i)]  There exists a measurable map $\cA_0 \ni a \mapsto \widetilde{\PP}_0^*(a) \in \mathcal{M}_1(\Omloc)$ satisfying for all $a \in \cA_0$ that  $\widetilde{\PP}_0^*(a) \in \mathcal{P}_0$ and
\[
\E_{\widetilde{\PP}_0^*(a)} \left[\Psi_{1} \left(\cdot, a\right)\right]=\inf_{\PP \in \mathcal{P}_0} \E_{\PP} \left[\Psi_{1} \left(\cdot, a\right)\right].
\]
\item[(ii)] There exists some  $a_0^* \in \cA_0$ such that 
\[
\Psi_0:=\sup_{\widetilde{a} \in \cA_0}\inf_{\PP\in \mathcal{P}_1}\E_{\PP}\left[\Psi_{1}\left(\cdot, \widetilde{a})\right)\right] = \inf_{\PP\in \mathcal{P}_1}\E_{\PP}\left[\Psi_{1}\left(\cdot, a_0^*)\right)\right].
\]
\end{itemize}
\end{cor}
\begin{proof}
This follows by the same arguments as the proof of Lemma~\ref{lem_appendix_backwards_iteration}.
\end{proof}
{
\begin{rem}\label{rem:backwards_iteration}
Let Assumption~\ref{asu_A}, Assumption~\ref{asu_psi}, and Assumption~\ref{asu_P} hold, and let $p \in \N_0$ be the integer from Assumption~\ref{asu_psi}. Moreover, for every $t=T-1,\dots, 1,0$ let $\Psi_t: \Omega^t \times \cA^t \to \R$ be recursively defined via \eqref{eq_defn_J_t}--\eqref{eq_defn_Psi_0}. Then, we see from Lemma~\ref{lem_appendix_backwards_iteration}, Corollary~\ref{cor_appendix_backwards_iteration}, and their proofs that each 
$\Psi_t: \Omega^t \times \cA^t \to \R$ satisfies \eqref{eq_linear_growth_thm_2} and \eqref{eq_Lipschitz_2_thm_1} with 
\begin{equation}\label{eq:constants_backwards_iteration}
\begin{split}
	C_{\Psi,t}&:= 2^{T-t} C_{\Psi}\prod_{s=t}^{T-1} C_{\mathcal{P},s}\geq 1,\\
	\alpha_t&:=\alpha \in (0,1],\\
	L_{\Psi,t}&:= 2^{T-t} L_{\Psi} \prod_{s=t}^{T-1} \max\left\{L_{\mathcal{A},s}^{\alpha}+L_{\mathcal{P},s}^{\alpha},1\right\}\geq 0.
\end{split}
\end{equation} 
\end{rem}
}
Now we are able to present  the proof of Theorem~\ref{thm_main_result}.
\vspace{0.3cm}
\begin{proof}[Proof of Theorem~\ref{thm_main_result}]~
\begin{itemize}
\item[(i)]
Since Assumption~\ref{asu_psi} is fulfilled for $\Psi = \Psi_T$, we obtain by a recursive application of Lemma~\ref{lem_appendix_backwards_iteration}~(iii) for $t=T-1,\dots,0$ that $\Psi_{t+1}$  satisfies the requirements of Lemma~\ref{lem_appendix_backwards_iteration}, i.e., $\Psi_{t+1}$ fulfills \eqref{eq_linear_growth_thm_1} and \eqref{eq_Lipschitz_1_thm_1}. Then,  the assertion follows by Lemma~\ref{lem_appendix_backwards_iteration}~(i) and (ii) for $t=T-1,\dots,1$ and by Corollary~\ref{cor_appendix_backwards_iteration}~(i) and (ii) for $t=0$.
\item[(ii)]
Let $t \in \{0,\dots,T-1\}$.
In the following, for $\omega^t =(\omega_1,\dots,\omega_t) \in \Omega^t$ we denote for any $0 \leq s \leq t$ by $\omega^s:=(\omega_1,\dots,\omega_s) \in \Omega^s$ the first $s$ coordinates of $\omega^t$.

Note that by definition of the kernel $\PP_t^*$ we have for all $\omega^t \in \Omega^t$ that
\begin{equation}\label{eq_min_eq_iii}
\inf_{\PP \in \mathcal{P}_t(\omega^t)} \E_{\PP} \left[\Psi_{t+1} \left(\omega^t \otimes_t \cdot,~ (a_s^*(\omega^s))_{s=0,\dots,t}  \right)\right] = \E_{\PP_t^*(\omega^t)} \left[\Psi_{t+1} \left(\omega^t \otimes_t \cdot,~ (a_s^*(\omega^s))_{s=0,\dots,t}  \right)\right].
\end{equation}
Then, by definition of $\ab^*$, $\PP_t^*$ and  of $\Psi_t$, we obtain for all $\omega^t \in \Omega^t$ that
\begin{align*}
\Psi_t\left(\omega^t,(a_s^*(\omega^s))_{s=0,\dots,t-1}\right) &=  \sup_{\widetilde{a} \in A} J_t\left(\omega^t, ~((a_s^*(\omega^s))_{s=0,\dots,t-1},\widetilde{a})\right)\\
&= J_t\left(\omega^t, ~((a_s^*(\omega^s))_{s=0,\dots,t-1},a^*_t(\omega^t)\right)\\
&=\inf_{\PP \in \mathcal{P}_t(\omega^t)} \E_{\PP} \left[\Psi_{t+1} \left(\omega^t \otimes_t \cdot,~ (a_s^*(\omega^s))_{s=0,\dots,t}  \right)\right]\\
&= \E_{\PP_t^*(\omega^t)} \left[\Psi_{t+1} \left(\omega^t \otimes_t \cdot,~ (a_s^*(\omega^s))_{s=0,\dots,t}  \right)\right] .
\end{align*}
Further, by definition of $\ab^*$, we have for all $\PP=\PP_0 \otimes \PP_1 \otimes \cdots \otimes \PP_{T-1} \in \mathfrak{P}$ that
\begin{equation}\label{P_star_ineq}
\begin{aligned}
&\E_{\PP} \left[\Psi_{t+1} \left(a_0^*,\dots,a_t^* \right)\right]\\
&=\int_{\Omega^{t+1}} \Psi_{t+1}\left((\omega^t,\omega),~(a_s^*(\omega^s))_{s=0,\dots,t} \right) \PP_0\otimes \cdots \otimes \PP_t (\D (\omega^t, \omega)) \\
&=\int_{\Omega^{t}} \left(\int_{\Omloc} \Psi_{t+1}\left((\omega^t,\omega),~(a_s^*(\omega^s))_{s=0,\dots,t} \right) \PP_t(\omega^t ; \D \omega) \right) \PP_0\otimes \cdots \otimes \PP_{t-1} (\D \omega^t)\\
&\geq \int_{\Omega^{t}} \inf_{\widetilde{\PP}_t \in \mathcal{P}_t(\omega^t)} \left(\int_{\Omloc} \Psi_{t+1}\left((\omega^t,\omega),~(a_s^*(\omega^s))_{s=0,\dots,t} \right)\widetilde{\PP}_t(\omega^t ; \D \omega) \right) \PP_0\otimes \cdots \otimes \PP_{t-1} (\D \omega^t)\\
&= \int_{\Omega^{t}} \Psi_{t}\left(\omega^t,~(a_s^*(\omega^s))_{s=0,\dots,t-1} \right) \PP_0\otimes \cdots \otimes \PP_{t-1} (\D \omega^t)\\
&=\E_{\PP} \left[\Psi_{t} \left(a_0^*,\dots,a_{t-1}^* \right)\right].
\end{aligned}
\end{equation}
Repeatedly applied, the above inequality \eqref{P_star_ineq} yields
\begin{equation}\label{P_star_ineq_2}
\E_{\PP} \left[\Psi_{T} \left(a_0^*,\dots,a_{T-1}^* \right)\right] \geq \Psi_0,
\end{equation}
and since $\PP\in \mathfrak{P}$ was chosen arbitrarily we also get, by using $\Psi = \Psi_T$, that
\begin{equation}\label{eq_ineq_proof_Psi_t_Psi_01}
\inf_{\PP\in \mathfrak{P}}\E_{\PP} \left[\Psi \left(a_0^*,\dots,a_{T-1}^* \right)\right] =\inf_{\PP\in \mathfrak{P}}\E_{\PP} \left[\Psi_{T} \left(a_0^*,\dots,a_{T-1}^* \right)\right] \geq \Psi_0.
\end{equation}
By considering the measure
\[
\PP^*: = \PP_0^* \otimes \cdots  \otimes\PP_{T-1}^*  \in \mathfrak{P},
\]
and by using \eqref{eq_min_eq_iii}, we obtain equality in \eqref{P_star_ineq} and \eqref{P_star_ineq_2}. This means, together with \eqref{eq_ineq_proof_Psi_t_Psi_01}, that  we have
\begin{equation}\label{eq_ineq_proof_Psi_t_Psi_0}
\inf_{\PP\in \mathfrak{P}}\E_{\PP} \left[\Psi \left(a_0^*,\dots,a_{T-1}^* \right)\right] \geq \Psi_0 = \E_{\PP^*}\left[\Psi \left(a_0^*,\dots,a_{T-1}^* \right)\right].
\end{equation}
Now let $(a_s(\cdot))_{s=0,\dots,T-1}\in \mathfrak{A} $ be an arbitrary control. Then, we have for all $\omega^t \in \Omega^t$ that
\begin{equation}\label{eq_ineq_proof2}
\begin{aligned}
&\hspace{-1cm}\E_{\widetilde{\PP}_t^*(\omega^t, (a_s(\omega^s))_{s=0,\dots,t})} \left[ \Psi_{t+1}\left(\omega^t \otimes_t \cdot, (a_s(\omega^s))_{s=0,\dots,t} \right) \right] \\
&= \inf_{\PP_t \in \mathcal{P}_t(\omega^t)} \E_{\PP_t} \left[ \Psi_{t+1}\left(\omega^t \otimes_t \cdot, (a_s(\omega^s))_{s=0,\dots,t} \right) \right]\\
&\leq \sup_{\widetilde{a} \in A_t}\inf_{\PP_t \in \mathcal{P}_t(\omega^t)} \E_{\PP_t} \left[ \Psi_{t+1}\left(\omega^t \otimes_t \cdot, \left((a_s(\omega^s))_{s=0,\dots,t-1},\widetilde{a}\right) \right) \right]\\
&= \sup_{\widetilde{a} \in A_t}J_t\left(\omega^t, \left((a_s(\omega^s))_{s=0,\dots,t-1},\widetilde{a}\right)\right) \\
&= \Psi_t \left( \omega^t, \left((a_s(\omega^s))_{s=0,\dots,t-1}\right)\right).
\end{aligned}
\end{equation}
Let $\PP=\PP_0  \otimes \cdots  \otimes\PP_{T-1} \in \mathfrak{P}$ be arbitrary and define 
\[
\Omega^t\ni \omega^t \mapsto {\PP}_t^{*,a}(\omega^t):=\widetilde{\PP}_t^*\left(\omega^t,~(a_s(\omega^s))_{s=0,\dots,t}\right) \in \mathcal{M}_1(\Omloc).
\]
Then \eqref{eq_ineq_proof2} implies
\begin{equation}\label{eq_ineq_proof3}
\begin{aligned}
\E_{\PP}\left[\Psi_{t}(a_0,\dots,a_{t-1})\right]&=\E_{\PP_0  \otimes \cdots \otimes \PP_{t-1}}\left[\Psi_{t}(a_0,\dots,a_{t-1})\right]\\
&\geq \E_{\PP_0 \otimes \PP_1 \otimes \cdots \otimes \PP_{t-1} \otimes {\PP}_t^{*,a}}\left[\Psi_{t+1}(a_0,\dots,a_{t})\right]\\
&\geq \inf_{\PP^{\prime} \in \mathfrak{P}} \E_{\PP^{\prime}}\left[\Psi_{t+1}(a_0,\dots,a_{t})\right].
\end{aligned}
\end{equation}
Hence, as $\PP \in \mathfrak{P}$ was arbitrary, we obtain from \eqref{eq_ineq_proof3} that
\begin{equation}\label{eq_ineq_inf_psi_ineq_psi}
\inf_{\PP \in \mathfrak{P}} \E_{\PP}\left[\Psi_{t+1}(a_0,\dots,a_{t})\right] \leq \inf_{\PP \in \mathfrak{P}} \E_{\PP} \left[\Psi_{t}(a_0,\dots,a_{t-1})\right].
\end{equation}
By a repeated application of the latter inequality \eqref{eq_ineq_inf_psi_ineq_psi} one sees that 
\begin{equation}
\inf_{\PP \in \mathfrak{P}} \E_{\PP}\left[\Psi(a_0,\dots,a_{T-1})\right] =\inf_{\PP \in \mathfrak{P}} \E_{\PP}\left[\Psi_{T}(a_0,\dots,a_{T-1})\right] \leq \Psi_{0}.
\end{equation}
Since  $(a_0,\dots,a_{T-1}) \in \mathfrak{A}$ was also chosen arbitrarily we obtain that
\begin{equation}
\sup_{a\in \mathfrak{A}}\inf_{\PP \in \mathfrak{P}} \E_{\PP}\left[\Psi(a_0,\dots,a_{T-1})\right] \leq \Psi_{0},
\end{equation}
which implies together with \eqref{eq_ineq_proof_Psi_t_Psi_0} that
\[
\sup_{a\in \mathfrak{A}}\inf_{\PP \in \mathfrak{P}} \E_{\PP}\left[\Psi(a_0,\dots,a_{T-1})\right] = \inf_{\PP\in \mathfrak{P}}\E_{\PP} \left[\Psi \left(a_0^*,\dots,a_{T-1}^* \right)\right] = \Psi_{0} = \E_{\PP^*}\left[\Psi \left(a_0^*,\dots,a_{T-1}^* \right)\right].
\]
\end{itemize}

\end{proof}

\subsection{Proofs of Section~\ref{sec_ambiguity}}
\subsubsection{Proofs of Section~\ref{sec_ambiguity_wasserstein}}
Before reporting the proof of Theorem~\ref{thm_wasserstein_ambiguity}
 we present several auxiliary results. 
 { 
 	%
 	\begin{lem}\label{lem:Lipschitz-implies-LHC}
 		Let $(\mathcal{X},d_{\mathcal{X}})$ and  $(\mathcal{Y},d_{\mathcal{Y}})$ be metric spaces  and let $\Lambda:\mathcal{X}\twoheadrightarrow \mathcal{Y}$ be a correspondence satisfying the following: there exists a constant $L\geq 0$ such that for each $x, \widetilde{x} \in \mathcal{X}$ and each $y\in \Lambda(x)$ there exists $\widetilde{y} \in \Lambda (\widetilde{x})$ such that
 		\begin{equation}\label{eq:Lipschitz-set-valued-condition}
 			d_{\mathcal{Y}}(y,\widetilde{y}) 
 			\leq
 			L \cdot d_{\mathcal{X}}(x,\widetilde{x}).
 		\end{equation}
 		Then $\Lambda:\mathcal{X}\twoheadrightarrow \mathcal{Y}$ is lower-hemicontinuous.
 	\end{lem}
 	\begin{proof}
 		We aim to apply \cite[Theorem~17.21]{Aliprantis}. To that end let $(x_n)_{n \in \N}\subseteq \mathcal{X}$ which converges to some $x\in \mathcal{X}$ and let $y \in \Lambda(x)$. By \eqref{eq:Lipschitz-set-valued-condition}, there exists for each $n \in \N$ some $y_n \in \Lambda(x_n)$ satisfying
 		\begin{equation*}
 			d_{\mathcal{Y}}(y,y_n) 
 			\leq
 			L \cdot d_{\mathcal{X}}(x,x_n).
 		\end{equation*}
 		Since the sequence $(x_n)_{n \in \N}$  converges to $x$, we  derive that also the sequence $(y_n)_{n \in \N}$ converges to $y$. Hence we can directly apply \cite[Theorem~17.21]{Aliprantis} to conclude that $\Lambda:\mathcal{X}\twoheadrightarrow \mathcal{Y}$ is lower-hemicontinuous.
 	\end{proof}
 }

{The statements of the following four simple lemmas about  Wasserstein spaces  are known in the literature especially for the case $p=0$, but we nevertheless present them for the sake of completeness.}

We use the notation $\Pi(\mu,\nu) \subset \mathcal{M}_1(\mathcal{X} \times \mathcal{X})$ to describe the set of all joint distributions of probability measures $\mu,\nu \in \mathcal{M}_1(\mathcal{X})$ on some suitable space $\mathcal{X}$.
\begin{lem}\label{lem_rel_compact}
Let $p\in \N_0$ and let $\left(\mathcal{X},\|\cdot \|_\mathcal{X}\right)$ be a separable Banach space. Let $\left(\mu_1^{(n)},\mu_2^{(n)}\right)_{n \in \N}\subseteq \mathcal{M}_1^p(\mathcal{X}) \times \mathcal{M}_1^p(\mathcal{X})$ {converge} in $\tau_p \times \tau_p$ to some $(\mu_1^\infty, \mu_2^\infty) \in \mathcal{M}_1^p(\mathcal{X}) \times \mathcal{M}_1^p(\mathcal{X}) $. Then $\bigcup_{n \in \N} \Pi(\mu_1^{(n)}, \mu_2^{(n)})$ is relatively compact w.r.t. $\tau_{p}\times \tau_p$.
\end{lem}
\begin{proof}
For the case $p=0$, this result follows from, e.g., the proof of \cite[Lemma 5.2]{clement2008wasserstein}. Hence, w.l.o.g.\ we can assume $p>0$. By \cite[Proof of Lemma 5.2]{clement2008wasserstein} we further know that $\bigcup_{n \in \N} \Pi(\mu_1^{(n)}, \mu_2^{(n)})$ is weakly relatively compact. Let $\|\cdot \|_{\mathcal{X} \times \mathcal{X}} $ be the norm on $\mathcal{X}\times \mathcal{X}$ defined by 
\[
\|(x,y) \|_{\mathcal{X} \times \mathcal{X}}:= \|x\|_\mathcal{X}+\|y\|_\mathcal{X},\qquad (x,y) \in  \mathcal{X} \times \mathcal{X}.
\]
Then, we have for all $x,y \in \mathcal{X}$ that $$\|(x,y)\|_{\mathcal{X} \times \mathcal{X}}^p = \|(x,0)+(0,y)\|_{\mathcal{X} \times \mathcal{X}}^p \leq 2^p \left( \|(x,0)\|_{\mathcal{X} \times \mathcal{X}}^p+\|(0,y)\|_{\mathcal{X} \times \mathcal{X}}^p \right)= 2^p \left( \|x\|_{\mathcal{X}}^p+\|y\|_{\mathcal{X}}^p \right),$$
as well as $\|(x,y)\|_{\mathcal{X} \times \mathcal{X}} \geq \max\{\|x\|_\mathcal{X},\|y\|_\mathcal{X}\}$.  Hence, we have for all $R>0$
\begin{align}
&\sup_{\pi \in \bigcup_{n \in \N} \Pi(\mu_1^{(n)}, \mu_2^{(n)})} \int_{\{(x,y) \in \mathcal{X} \times \mathcal{X}: \|(x,y)\|_{\mathcal{X}\times \mathcal{X}}>R\}}\|(x,y)\|_{\mathcal{X}\times \mathcal{X}}^p \,\pi(\D x, \D y) \label{eq_weak_compact_1}\\
\leq &2^p \cdot \sup_{\pi \in \bigcup_{n \in \N} \Pi(\mu_1^{(n)}, \mu_2^{(n)})} \int_{\{(x,y) \in \mathcal{X} \times \mathcal{X}: \|(x,y)\|_{\mathcal{X}\times \mathcal{X}}>R\}}\|x\|_\mathcal{X}^p \,\pi(\D x, \D y) \notag\\
&\hspace{2cm}+2^p \cdot \sup_{\pi \in \bigcup_{n \in \N} \Pi(\mu_1^{(n)}, \mu_2^{(n)})}  \int_{\{(x,y) \in \mathcal{X} \times \mathcal{X}: \|(x,y)\|_{\mathcal{X} \times \mathcal{X}} >R\}}\|y\|_\mathcal{X}^p \,\pi(\D x, \D y)  \notag \\
\leq &2^p \cdot \sup_{n \in \N} \int_{\{x \in \mathcal{X}: \|x\|_\mathcal{X}>R\}}\|x\|_\mathcal{X}^p \, \mu_1^{(n)} (\D x)+2^p \cdot \sup_{n \in \N}  \int_{\{y \in \mathcal{X}: \|y\|_\mathcal{X}>R\}}\|y\|_\mathcal{X}^p \, \mu_2^{(n)} (\D y). \notag
\end{align}
Now, the convergence $\lim_{n \rightarrow \infty}\left(\mu_1^{(n)},\mu_2^{(n)}\right) = (\mu_1^\infty, \mu_2^\infty) $  in $\tau_p \times \tau_p$ implies with \cite[Theorem 2.2.1 (3)]{panaretos2020invitation} both that 
$$\lim_{R \rightarrow \infty}\sup_{n \in \N} \int_{\{x \in \mathcal{X}: \|x\|_\mathcal{X}>R\}}\|x\|_\mathcal{X}^p \,\mu_1^{(n)} (\D x) =0,$$
 and that 
 $$\lim_{R \rightarrow \infty}\sup_{n \in \N}  \int_{\{y \in \mathcal{X}: \|y\|_\mathcal{X}>R\}}\|y\|_\mathcal{X}^p\, \mu_2^{(n)} (\D y) =0.$$
Therefore, the integral in \eqref{eq_weak_compact_1} vanishes as $R \rightarrow \infty$, and the relative compactness of $\bigcup_{n \in \N} \Pi(\mu_1^{(n)}, \mu_2^{(n)})$ in $\tau_p \times \tau_p$ follows directly by \cite[Proposition 2.2.3]{panaretos2020invitation}.
\end{proof}

\begin{lem}\label{lem_int_lsc}
Let $\left(\mathcal{X},\|\cdot \|_\mathcal{X}\right)$ be a separable Banach space, and let $p \in \N_0$ and $q\in \N$ with $q\geq p$. Moreover, let $f \in C_q(\mathcal{X},\R)$ be non-negative. Then
\begin{align*}
\mathcal{M}_1^q(\mathcal{X}) &\rightarrow \R\\
\mu &\mapsto \int_\mathcal{X} f(x) \,\mu(\D x)
\end{align*}
is lower semicontinuous w.r.t.\,$\tau_p$.
\end{lem}
\begin{proof}
If $p=q$, then the map is continuous by definition of $\tau_p=\tau_q$. Hence, w.l.o.g.\ we can assume $q>p$. We define for each $k \in \N$ the function
\[
\mathcal{X} \ni x \mapsto f_k(x):=f(x)\wedge \frac{k \cdot f(x)}{1+\|x\|_\mathcal{X}^{q-p}}.
\]
One sees directly that $f_k \geq 0$ and that $f_k \in C_p(\mathcal{X}, \R)$ for all $k \in \N$. Moreover, we have $f_k \uparrow f$ monotonically as $k \rightarrow \infty$.
Now, let $(\mu^{(n)})_{n \in \N} \subseteq \mathcal{M}_1^q(\mathcal{X})$ and $\mu \in \mathcal{M}_1^q(\mathcal{X})$ such that $\mu^{(n)} \rightarrow \mu$ in $\tau_p$. Then, we have by the monotone convergence theorem, since $f_k \in C_p(\mathcal{X}, \R)$, and since $f_k \leq f$ that
\begin{align*}
\int_\mathcal{X} f(x) \mu(\D x) 
= \lim_{k \rightarrow \infty} \int_\mathcal{X} f_k(x)\,\mu(\D x) 
 = \lim_{k \rightarrow \infty} \lim_{n \rightarrow \infty} \int_\mathcal{X} f_k(x)\,\mu^{(n)}(\D x) 
 \leq  \liminf_{n \rightarrow \infty} \int_\mathcal{X} f(x)\, \mu^{(n)}(\D x).
\end{align*}
\end{proof}

\begin{lem}\label{lem_lsc} 
Let $\left(\mathcal{X},\|\cdot \|_\mathcal{X}\right)$ be a separable Banach space, and let $p \in \N_0$ and $q\in \N$ with $q\geq p$.
Then, the map 
\begin{align*}
\mathcal{M}_1^q(\mathcal{X}) \times \mathcal{M}_1^q(\mathcal{X})  &\rightarrow [0,\infty)\\
(\mu_1,\mu_2) &\mapsto d_{W_q}(\mu_1,\mu_2)
\end{align*}
is lower semicontinuous in $\tau_p\times \tau_p$.
\end{lem}
\begin{proof}
{
Since $\tau_0\subseteq \tau_p$, the result now follows directly from in \cite[Corollary 5.3]{clement2008wasserstein} which has proven the lower semicontinuity of that map with respect to $\tau_0\times \tau_0$.}
\end{proof}

\begin{lem}\label{lem_wasserstein_compact}
Let $p \in \N_0$ and $q\in \N$ such that $q>p$. Moreover, let $\left(\mathcal{X},\|\cdot \|_\mathcal{X}\right)$ be a separable Banach space. Then, for any $\widehat{\mu} \in \mathcal{M}_1^q(\mathcal{X})$ and $\varepsilon>0$, the set
\[
\mathcal{B}_{\varepsilon}^{(q)}\left(\widehat{\mu}\right): = \left\{ \mu \in \mathcal{M}_1^q(\mathcal{X})~\middle|~ d_{W_q}\left(\mu, \widehat{\mu}\right)\leq \varepsilon \right\}
\]
is compact w.r.t.\,$\tau_p$.
\end{lem}
\begin{proof}
The case $p=0$ follows, e.g., from \cite[Theorem 1]{yue2022linear}. Hence, w.l.o.g.\ we can assume $p>0$. First, to see that the set is relatively compact w.r.t.\,$\tau_p$, note that by \cite[Lemma 1]{yue2022linear} we have
\[
\sup_{\mu \in \mathcal{B}_{\varepsilon}^{(q)}(\widehat{\mu})} \int_\mathcal{X} \|x\|_\mathcal{X}^p \|x\|_\mathcal{X}^{q-p} \,\mu(\D x) 
=
 \sup_{\mu \in \mathcal{B}_{\varepsilon}^{(q)}(\widehat{\mu})} \int_\mathcal{X} \|x\|_\mathcal{X}^q \,\mu(\D x) < \infty.
\]
Hence, since $[0,\infty) \ni r \mapsto r^{q-p}$ is monotonically divergent as $q>p$, we get relative compactness w.r.t.\,$\tau_p$ from, e.g., \cite[Proposition 2.2.3]{panaretos2020invitation}.

To see that $\mathcal{B}_{\varepsilon}^{(q)}\left(\widehat{\mu}\right)$ is closed, we apply Lemma~\ref{lem_lsc}, and obtain that 
\begin{equation}\label{eq_map_proof_lem_1}
\begin{aligned}
\mathcal{M}_1^q(\mathcal{X}) &\rightarrow \R\\
\mu &\mapsto d_{W_q}(\widehat{\mu},\mu)
\end{aligned}
\end{equation}
is lower semicontinuous. Hence, $\mathcal{B}_{\varepsilon}^{(q)}\left(\widehat{\mu}\right)$ being a superlevel set of \eqref{eq_map_proof_lem_1} is closed.
\end{proof}
{Next, the following lemma in  Euclidean geometry will play a crucial role in order to prove that Assumption~\ref{asu_P}~(iii) holds both for the Wasserstein uncertainty case (see Theorem~\ref{thm_wasserstein_ambiguity}) and the parameter uncertainty case (see Theorem~\ref{thm_param_assumptions}).
\begin{lem}\label{lem_EuclideanGeometry}
Let $s\in \N$ and let $\mathbb{S}\subseteq \R^s$ be a non-empty convex subset.
For any arbitrary fixed $\lambda \in [0,1]$ define
\begin{equation}\label{eq_defn_v}
	\begin{aligned}
		v^{(\lambda)}_1: \mathbb{S} \times \mathbb{S} \times \mathbb{S} &\rightarrow \mathbb{S} \times \mathbb{S} \times \mathbb{S}\\
		(a,b,c) &\mapsto (a,b,\lambda a + (1-\lambda)c),\\
		v_2: \mathbb{S} \times \mathbb{S} \times \mathbb{S} &\rightarrow \mathbb{S} \\
		(a,b,c) &\mapsto \begin{cases} a &\text{ if } a=b=c,\\
			\frac{\|c-a\|}{\|c-a\|+\|b-a\|} \cdot  c+\left(1-\frac{\|c-a\|}{\|c-a\|+\|b-a\|}\right) \cdot b &\text{ else, }
		\end{cases}\\
		v^{(\lambda)}: \mathbb{S} \times \mathbb{S} \times \mathbb{S} &\rightarrow \mathbb{S}
		\\(a,b,c) &\mapsto v_2\big(v^{(\lambda)}_1(a,b,c)\big).
	\end{aligned}
\end{equation} Then for any $(a,b,c)\in \mathbb{S} \times \mathbb{S} \times \mathbb{S}$ the following holds
\begin{align}
	\|v^{(\lambda)}(a,b,c)-c\|&\leq \|b-a\| + \lambda \|c-a\|, \label{eq_v_ineq_1} \\
	\|v^{(\lambda)}(a,b,c)-b\|&\leq (1-\lambda) \|c-a\|. \label{eq_v_ineq_2}
\end{align}
\end{lem}
\begin{proof}[Proof of Lemma~\ref{lem_EuclideanGeometry}]
First,	note that by the convexity of $\mathbb{S}$, the above maps are well-defined.

	In the case $a=b=c$ the inequalities \eqref{eq_v_ineq_1} and \eqref{eq_v_ineq_2} are trivial, so without loss of generality ($a=b=c$) does not hold. Then, 
	 note that, by definition of $v_2$, we have
	\begin{equation}
		\label{eq:v_2-part1}
		\begin{split}
			\|v_2(a,b,c)-c\| & = \left\| -\left(1-\frac{\|c-a\|}{\|c-a\|+\|b-a\|} \right)\cdot  c+\left(1-\frac{\|c-a\|}{\|c-a\|+\|b-a\|}\right) \cdot b \right\|\\
			& = \left\| -\left(\frac{\|b-a\|}{\|c-a\|+\|b-a\|} \right)\cdot  c+\left(\frac{\|b-a\|}{\|c-a\|+\|b-a\|}\right) \cdot b \right\|\\
			& \leq  \|b-c\| \cdot \frac{\|b-a\|}{\|c-a\|+\|b-a\|} \\
			& \leq  \left(\|b-a\| +\|a-c\|\right) \cdot \frac{\|b-a\|}{\|c-a\|+\|b-a\|} = \|b-a\|.
		\end{split}
	\end{equation}
	Similarly, 
	note that 
	\begin{equation}
		\label{eq:v_2-part2}
		\begin{split}
			\|v_2(a,b,c)-b\| & = \left\|\frac{\|c-a\|}{\|c-a\|+\|b-a\|} \cdot  c-\frac{\|c-a\|}{\|c-a\|+\|b-a\|} \cdot b \right\|\\
			& \leq  \|b-c\| \cdot \frac{\|c-a\|}{\|c-a\|+\|b-a\|}  \leq  \|c-a\|.
		\end{split}
	\end{equation}
	Therefore, using \eqref{eq:v_2-part1}, we see that indeed
	\begin{equation*}
		\begin{split}
			\|v^{(\lambda)}(a,b,c)-c\| 
			& = 
			\|v_2(a,b,\lambda a + (1-\lambda)c)-c\|
			\\
			& \leq 
			\|v_2(a,b,\lambda a + (1-\lambda)c)-(\lambda a + (1-\lambda)c)\|
			+ 
			\| (\lambda a + (1-\lambda)c)-c\|
			\\
			& = 
			\| b-a \| + \lambda \|c-a \|.
		\end{split}
	\end{equation*}
	Moreover, by using \eqref{eq:v_2-part2}, we also obtain that indeed
	\begin{equation*}
		\begin{split}
			\|v^{(\lambda)}(a,b,c)-b\|
			&=
			\|v_2(a,b,\lambda a + (1-\lambda)c)-b\|\\	
			&\leq 
			\|(\lambda a + (1-\lambda)c)-a\|\\
			&= (1-\lambda) \|c-a \|.
		\end{split}
	\end{equation*}
	\end{proof}
}
\begin{proof}[Proof of Theorem~\ref{thm_wasserstein_ambiguity}]
Let $t\in \{0,1,\dots,T\}$ be fixed. We verify that the conditions (i)-(iv) of Assumption~\ref{asu_P} are fulfilled.
\begin{itemize}
\item[(i)] The ambiguity set $\mathcal{P}_t(\omega^t)$ contains, by definition, for all $\omega^t \in \Omega^t$ the reference measure $\widehat{\PP}_t(\omega^t)$ and is hence non-empty. Moreover, by, e.g., \cite[Lemma 1]{yue2022linear} we have $\mathcal{P}_t(\omega^t) \subseteq  \mathcal{M}_1^q(\Omloc)\subseteq \mathcal{M}_1^{\max\{1,p\}}(\Omloc)$ for all $\omega^t \in \Omega^t$ as $p<q$.\\
The compactness of $\mathcal{P}_t(\omega^t)$ w.r.t.\,$\tau_p$ follows from Lemma~\ref{lem_wasserstein_compact} as $p<q$.\\
{To show the upper hemicontinuity we mainly follow the lines of the proof of \cite[Proposition 3.1]{neufeld2023mdp}. Therefore, the goal is to apply the characterization of upper hemicontinuity provided in \cite[Theorem 17.20]{Aliprantis}. Let $({\omega^t}^{(n)})_{n\in \N} \subseteq \Omega^t$ such that ${\omega^t}^{(n)} \rightarrow {\omega^t}\in \Omega^t$ for $n \rightarrow \infty$. Further, consider a sequence $(\PP^{(n)})_{n \in \N}$ such that $\PP^{(n)} \in  \mathcal{B}^{(q)}_{\varepsilon_t({\omega^t}^{(n)})}\left(\widehat{\PP}_t({\omega^t}^{(n)})\right)=\mathcal{P}_t({\omega^t}^{(n)})$ for all $n \in \N$, i.e., we have $\left({\omega^t}^{(n)},\PP^{(n)}\right)_{n \in \N} \subseteq \operatorname{Gr} \mathcal{P}_t$, where $\operatorname{Gr} \mathcal{P}_t$ denotes the \emph{graph} of $\mathcal{P}_t$.

Let $(\delta_n)_{n \in \N} \subseteq (0,1)$ with $\lim_{n \rightarrow \infty} \delta_n = 0$. Note that, since $\Omega^t \ni \omega^t \mapsto \widehat{\PP}_t(\omega^t)$ is, by assumption, $L_{\widehat{\mathbb{P}},t}$-Lipschitz continuous we have $$\lim_{n \rightarrow \infty} d_{W_q}\left(\widehat{\PP}_t(\omega^t), \widehat{\PP}_t({\omega^t}^{(n)})\right) \leq \lim_{n \rightarrow \infty} L_{\widehat{\mathbb{P}},t} \cdot \| \omega^t- {\omega^t}^{(n)} \|=0.$$ Hence, there exists a subsequence $\left(\widehat{\PP}_t({\omega^t}^{(n_k)})\right)_{k \in \N}$ such that 
\begin{equation}\label{eq_proof_uhc_1}
d_{W_q}\left(\widehat{\PP}_t({\omega^t}), \widehat{\PP}_t({\omega^t}^{(n_k)})\right) \leq \delta_k \cdot 
\overline{\varepsilon}_t,
\quad \text{ for all } k \in \N.
\end{equation}
This implies for each $\PP^{(n_k)}$, $k \in \N$, that 
\begin{align*}
d_{W_q}\left(\widehat{\PP}_t({\omega^t}), \PP^{(n_k)}\right)
&\leq d_{W_q}\left(\widehat{\PP}_t({\omega^t}), \widehat{\PP}_t({\omega^t}^{(n_k)})\right)+d_{W_q}\left( \widehat{\PP}_t({\omega^t}^{(n_k)})), \PP^{(n_k)}\right)\\
& \leq \delta_k \cdot \overline{\varepsilon}_t 
 +\varepsilon_t({\omega^t}^{(n_k)}) \\
&\leq 2 \overline{\varepsilon}_t.
\end{align*}
Hence, $\PP^{(n_k)} \in \mathcal{B}_{2\overline{\varepsilon}_t}^{(q)}(\widehat{\PP}_t({\omega^t}))$ for all $k \in \N$. According to Lemma~\ref{lem_wasserstein_compact}, $\mathcal{B}_{2\overline{\varepsilon}_t}^{(q)}(\widehat{\PP}_t({\omega^t}))$ is compact in $\tau_p$, implying the existence of a subsequence $(\PP^{(n_{k_\ell})})_{\ell \in \N}$ such that $\PP^{(n_{k_\ell})} \xrightarrow{\tau_p} \PP$ as $\ell \rightarrow \infty$ for some $\PP \in  \mathcal{B}_{2\overline{\varepsilon}_t}^{(q)}(\widehat{\PP}_t({\omega^t}))$. In particular, since by assumption $\widehat{\PP}_t({\omega^t})$ possesses finite $q$-th moments, $\PP$ has also finite $q$-th moments, see \cite[Lemma 1]{yue2022linear}. It remains to prove that $\PP \in \mathcal{B}^{(q)}_{\varepsilon_t(\omega^t)}(\widehat{\PP}_t({\omega^t}))$.
{
To that end, note that since $\mathcal{M}_1^q(\Omloc)\ni\mu \mapsto d_{W_q}(\widehat{\PP}_t(\omega^t),\mu)$ is lower semicontinuous in $\tau_p$, see Lemma~\ref{lem_lsc}, and $\PP^{(n_{k_\ell})} \xrightarrow{\tau_p} \PP$, we obtain from
\eqref{eq_proof_uhc_1}, that $\PP^{(n_{k_\ell})} \in  \mathcal{B}^{(q)}_{\varepsilon_t({\omega^t}^{(n_{k_\ell})})}\left(\widehat{\PP}_t({\omega^t}^{(n_{k_\ell})})\right)$,
 and the continuity of $\varepsilon_t:\Omega^t \to [0,\overline{\varepsilon}_t]$ that 
\begin{align*}
	d_{W_q}\left(\widehat{\PP}_t(\omega^t),\PP\right) 
	&\leq
	\liminf_{\ell \rightarrow \infty} d_{W_q}\left(\widehat{\PP}_t(\omega^t), \PP^{({n_k}_\ell)}\right) \\
		&\leq
	\liminf_{\ell \rightarrow \infty} \Big[ d_{W_q}\left(\widehat{\PP}_t(\omega^t), \widehat{\PP}_t({\omega^t}^{({n_k}_\ell)})\right) 
	+
d_{W_q}\left(\widehat{\PP}_t({\omega^t}^{({n_k}_\ell)}), \PP^{({n_k}_\ell)}\right)
	\Big] \\
	&\leq 
	\liminf_{\ell \rightarrow \infty} \Big(\delta_{k_{\ell}}\cdot \overline{\varepsilon}_t
	+ \varepsilon_t({\omega^t}^{({n_k}_\ell)})\Big)\\
	&=\varepsilon_t(\omega^t),
\end{align*}
and hence indeed $\PP \in \mathcal{B}_{\varepsilon_t(\omega^t)}^{(q)}(\widehat{\PP}_t(\omega^t))$.}
The assertion that $\mathcal{P}$ is upper hemicontinuous follows now with the characterization of upper hemicontinuity provided in ~\cite[Theorem 17.20]{Aliprantis}.

Finally, the lower hemicontinuity of $\mathcal{P}_t:\Omega^t \twoheadrightarrow (\mathcal{M}_1^p(\Omloc),\tau_p)$ will automatically follow once we have proven below that Assumption~\ref{asu_P}~(iii) holds, but where we show the stronger statement that \eqref{eq_condition_P} holds with respect to $d_{W_{\max\{1,p\}}}$ instead of $d_{W_1}$,
cf.\ Lemma~\ref{lem:Lipschitz-implies-LHC}.
}

\item[(ii)] 
{{Define the constant\footnote{We use here the notation $\omega_a^t=(\omega_{a,1},\dots,\omega_{a,t})\in \Omega^t$.}
		\begin{equation*}
 C_{\mathcal{P},t}:= \max\left\{2^{p-1}\bigg({\overline{\varepsilon}_t}
 + 
 \inf_{\omega_a^t \in \Omega^t}\bigg\{
 \bigg(\int_{\Omloc }\|z\|^p \widehat{\PP}_t(\omega_a^t) (\D z) \bigg)^{\nicefrac{1}{p}}
 +
 L_{\widehat{\mathbb{P}},t} \sum_{i=1}^t \|\omega_{a,i}\|
 \bigg\}\bigg)^p,
 2^{p-1}L_{\widehat{\mathbb{P}},t}^pt^{p-1},~ 1\right\} <\infty.
\end{equation*}
 For all  $\omega^t \in \Omega^t$ and $\PP \in \mathcal{P}_t(\omega^t)$ let $\pi_p(\D x, \D y) \in \Pi(\PP,\widehat{\PP}_t(\omega^t))$ be the optimal coupling of $\PP$ and $\widehat{\PP}_t(\omega^t)$ w.r.t.\ the $p$-Wasserstein distance $d_{W_p}$.
Then, by Minkowski's inequality we have
\begin{equation}\label{eq_prop_22_1}
\begin{aligned}
\bigg(\int_{\Omloc} \|x\|^p\PP( \D x)\bigg)^{\nicefrac{1}{p}} &= \bigg(\int_{\Omloc \times \Omloc} \| x \|^p \pi_p( \D  x, \D y)\bigg)^{\nicefrac{1}{p}} \\
&\leq 
\bigg(\int_{\Omloc \times \Omloc}\|x-y\|^p \pi_p (\D x, \D y)\bigg)^{\nicefrac{1}{p}}
+
\bigg(\int_{\Omloc}\|y\|^p \widehat{\PP}_t(\omega^t) (\D y)\bigg)^{\nicefrac{1}{p}}  \\
&= d_{W_p}(\PP, \widehat{\PP}_t(\omega^t)) 
+
\bigg(\int_{\Omloc}\|y\|^p \widehat{\PP}_t(\omega^t) (\D y)\bigg)^{\nicefrac{1}{p}} \\
&\leq 
{\overline{\varepsilon}_t}
+ 
\bigg(\int_{\Omloc}\|y\|^p \widehat{\PP}_t(\omega^t) (\D y)\bigg)^{\nicefrac{1}{p}}.
\end{aligned}
\end{equation}
Now 
for any arbitrary $\omega_a^t \in \Omega^t$
let $\pi_{\widehat{\PP}(\omega_a^t)}\in \Pi\left(\widehat{\PP}_t(\omega^t),\widehat{\PP}_t(\omega_a^t)\right)$ be the optimal coupling of $\widehat{\PP}_t(\omega^t)$ and $\widehat{\PP}_t(\omega_a^t)$ w.r.t.\,the $p$-Wasserstein distance $d_{W_p}$. Then, by \eqref{eq_prop_22_1}, Minkowski's inequality, the Lipschitz continuity of $\Omega^t\ni\omega^t\mapsto \widehat{\PP}_t(\omega^t)\in (\mathcal{M}_1^q(\Omloc),\tau_q)$, and that $p<q$
 we have
\begin{equation}
	\begin{split}
\label{eq_prop_22_1+}
\bigg(\int_{\Omloc} \|x\|^p\PP( \D x)\bigg)^{\nicefrac{1}{p}}
&\leq 
{\overline{\varepsilon}_t}
+ 
\bigg(\int_{\Omloc \times \Omloc }\|y\|^p \pi_{\widehat{\PP}_t(\omega_a^t)} (\D y, \D  z)\bigg)^{\nicefrac{1}{p}}\\
&\leq 
{\overline{\varepsilon}_t}
+ 
\bigg(\int_{\Omloc }\|z\|^p \widehat{\PP}_t(\omega_a^t) (\D z) \bigg)^{\nicefrac{1}{p}}
+
\bigg(\int_{\Omloc \times \Omloc }\|y-z\|^p \pi_{\widehat{\PP}_t(\omega_a^t)} (\D x, \D  y)\bigg)^{\nicefrac{1}{p}} \\
&=
{\overline{\varepsilon}_t}
+ 
\bigg(\int_{\Omloc }\|z\|^p \widehat{\PP}_t(\omega_a^t) (\D z) \bigg)^{\nicefrac{1}{p}}
+
d_{W_p}(\widehat{\PP}_t(\omega^t),\widehat{\PP}_t(\omega_a^t)) \\
&\leq
{\overline{\varepsilon}_t}
+ 
\bigg(\int_{\Omloc }\|z\|^p \widehat{\PP}_t(\omega_a^t) (\D z) \bigg)^{\nicefrac{1}{p}}
+
d_{W_q}(\widehat{\PP}_t(\omega^t),\widehat{\PP}_t(\omega_a^t)) \\
&\leq
{\overline{\varepsilon}_t}
 +
\bigg(\int_{\Omloc }\|z\|^p \widehat{\PP}_t(\omega_a^t) (\D z) \bigg)^{\nicefrac{1}{p}}
+
L_{\widehat{\mathbb{P}},t} \sum_{i=1}^t \|\omega_i-\omega_{a,i}\| \\
&\leq
{\overline{\varepsilon}_t}
+
\bigg(\int_{\Omloc }\|z\|^p \widehat{\PP}_t(\omega_a^t) (\D z) \bigg)^{\nicefrac{1}{p}}
+
L_{\widehat{\mathbb{P}},t} \sum_{i=1}^t \|\omega_{a,i}\|
+
L_{\widehat{\mathbb{P}},t} \sum_{i=1}^t \|\omega_i\|.
\end{split}
\end{equation}
Since  $\omega_a^t \in \Omega^t$ was arbitrary, we conclude that
\begin{equation}
	\label{eq_prop_22_1++}
\begin{split}
\bigg(\int_{\Omloc} \|x\|^p\PP( \D x)\bigg)^{\nicefrac{1}{p}}
&\leq
{\overline{\varepsilon}_t}
+ 
\inf_{\omega_a^t \in \Omega^t}\bigg\{
\bigg(\int_{\Omloc }\|z\|^p \widehat{\PP}_t(\omega_a^t) (\D z) \bigg)^{\nicefrac{1}{p}}
+
L_{\widehat{\mathbb{P}},t} \sum_{i=1}^t \|\omega_{a,i}\|
\bigg\}
+
L_{\widehat{\mathbb{P}},t} \sum_{i=1}^t \|\omega_i\|.
\end{split}
\end{equation}
This implies that

\begin{equation}
	\label{eq_prop_22_1+++}
\begin{split}
	\int_{\Omloc} \|x\|^p\PP( \D x)
	&\leq
	2^{p-1}\bigg({\overline{\varepsilon}_t}
	+ 
	\inf_{\omega_a^t \in \Omega^t}\bigg\{
	\bigg(\int_{\Omloc }\|z\|^p \widehat{\PP}_t(\omega_a^t) (\D z) \bigg)^{\nicefrac{1}{p}}
	+
	L_{\widehat{\mathbb{P}},t} \sum_{i=1}^t \|\omega_{a,i}\|
	\bigg\}\bigg)^p
	+
	2^{p-1}L_{\widehat{\mathbb{P}},t}^p\Big( \sum_{i=1}^t \|\omega_i\|\Big)^p\\
		&\leq
	2^{p-1}\bigg({\overline{\varepsilon}_t}
	+ 
	\inf_{\omega_a^t \in \Omega^t}\bigg\{
	\bigg(\int_{\Omloc }\|z\|^p \widehat{\PP}_t(\omega_a^t) (\D z) \bigg)^{\nicefrac{1}{p}}
	+
	L_{\widehat{\mathbb{P}},t} \sum_{i=1}^t \|\omega_{a,i}\|
	\bigg\}\bigg)^p
	+
	2^{p-1}L_{\widehat{\mathbb{P}},t}^pt^{p-1}\sum_{i=1}^t \|\omega_i\|^p\\
&\leq C_{\mathcal{P},t} \bigg(1+\sum_{i=1}^t \|\omega_i\|^p\bigg).	
\end{split}
\end{equation}
}}
\item[(iii)]

Let $\omega^t, \widetilde{\omega}^t \in \Omega^t$ and let $\PP_1 \in \mathcal{P}_t(\omega^t)=\mathcal{B}_{\varepsilon_t(\omega^t)}^{(q)} \left(\widehat{\PP}_t(\omega^t)\right)$. Without loss of generality assume $\PP_1 \not\in \mathcal{P}_t(\widetilde{\omega}^t)$. 

{
To construct $\PP_2 \in \mathcal{M}_1(\Omloc)$, 
in the following, we denote $\widehat{\PP}_1:= \widehat{\PP}_t(\omega^t)$, $\widehat{\PP}_2:= \widehat{\PP}_t(\widetilde{\omega}^t)$ as well as 
$\widehat{\Omloc}_{,1}:= \Omloc$, $\widehat{\Omloc}_{,2}:= \Omloc$, ${\Omloc}_{,1}:= \Omloc$ , ${\Omloc}_{,2}:= \Omloc$, i.e., we simply consider copies of $\Omloc$. Then, we consider a probability measure $\gamma:= \gamma(\D a, \D b, \D c) \in \mathcal{M}_1 \left(\widehat{\Omloc}_{,1} \times \widehat{\Omloc}_{,2} \times {\Omloc}_{,1} \right)$ which fulfills the following marginal constraints. 
\begin{enumerate}
	\item The marginal of $\gamma$ on $\widehat{\Omloc}_{,1}$ is $\widehat{\PP}_1$.
	\item The marginal of $\gamma$ on $\widehat{\Omloc}_{,2}$ is $\widehat{\PP}_2$.
	\item The marginal of $\gamma$ on ${\Omloc}_{,1}$ is ${\PP}_1$.
	\item The marginal of $\gamma$ on $\widehat{\Omloc}_{,1} \times \widehat{\Omloc}_{,2}$ minimizes the $q$-Wasserstein distance between $\widehat{\PP}_1$ and $\widehat{\PP}_2$.
	\item The marginal of $\gamma$ on $\widehat{\Omloc}_{,1} \times {\Omloc}_{,1}$ minimizes the $q$-Wasserstein distance between $\widehat{\PP}_1$ and ${\PP}_1$.
\end{enumerate}
The existence of such a probability measure follows by the \emph{Gluing Lemma}, see, e.g. \cite[Lemma 7.6.]{villani2021topics}.
We then choose $\lambda \in [0,1]$ by
\begin{equation}
	\label{eq:lambda}
\lambda:=\begin{cases}
	\frac{\max\{\varepsilon_t(\omega^t)-\varepsilon_t(\widetilde{\omega}^t),0\}}{\varepsilon_t(\omega^t)}   &  \mbox{ if } \varepsilon_t(\omega^t) \neq 0,\\
	0							& \mbox{ else,}
\end{cases}
\end{equation}
 use the map $v^{(\lambda)}:\Omloc \times \Omloc \times \Omloc \rightarrow \Omloc$ introduced in \eqref{eq_defn_v} in Lemma~\ref{lem_EuclideanGeometry} (with $\mathbb{S}\leftarrow \Omloc$ in the notation of Lemma~\ref{lem_EuclideanGeometry}) with respect to this $\lambda \in [0,1]$ in \eqref{eq:lambda} to define
\[
\widetilde{\gamma}:= \gamma \circ \left(\operatorname{id}_{{\Omloc}_{,1} \times {\Omloc}_{,2} \times \widehat{\Omloc}_{,1} },~v^{(\lambda)}\right)^{-1} \in \mathcal{M}_1\left(\widehat{\Omloc}_{,1} \times \widehat{\Omloc}_{,2} \times {\Omloc}_{,1} \times {\Omloc}_{,2} \right).
\]
Then, we set $\PP_2 \in \mathcal{M}_1( {\Omloc}_{,2}) = \mathcal{M}_1(\Omloc)$  to be defined as the marginal of $\widetilde{\gamma}$ on ${\Omloc}_{,2} $.

Next, we claim $\PP_2 \in \mathcal{B}_{\varepsilon_t(\widetilde{\omega}^t)}^{(q)} \left(\widehat{\PP}_2\right)$. 
Note that, by construction, the second and fourth marginals  of $\widetilde{\gamma}$ are $\widehat{\PP}_2$ and $\PP_2$, respectively.
Hence, by \eqref{eq_v_ineq_2}, by the definition of $\gamma$, and as $\PP_1 \in \mathcal{B}_{\varepsilon_t(\omega^t)}^{(q)} \left(\widehat{\PP}_1\right)$ we have that 
\begin{equation}
	\label{eq:radius-P2}
	\begin{split}
\operatorname{d}_{W_q} \left(\widehat{\PP}_2, \PP_2 \right) & =\inf_{\pi \in \Pi(\widehat{\PP}_2, \PP_2)}  \bigg(\int_{\Omloc \times \Omloc } \|d-b\|^q \,\pi(\D d, \D b)\bigg)^{\nicefrac{1}{q}} \\
&\leq \bigg(\int_{\Omloc \times \Omloc \times \Omloc \times \Omloc} \| d-b\|^q \,\widetilde{\gamma}(\D a, \D b, \D c, \D d)\bigg)^{\nicefrac{1}{q}}\\
&= \bigg(\int_{\Omloc \times \Omloc \times \Omloc} \|v^{(\lambda)}(a,b,c)-b\|^q \,{\gamma}(\D a, \D b, \D c)\bigg)^{\nicefrac{1}{q}}\\
&\leq  \bigg(\int_{\Omloc \times \Omloc \times \Omloc} (1-\lambda)^q\|c-a\|^q \,{\gamma}(\D a, \D b, \D c)\bigg)^{\nicefrac{1}{q}}\\
& = (1-\lambda)\operatorname{d}_{W_q} \left(\widehat{\PP}_1, \PP_1 \right).
\end{split}
\end{equation}
Moreover, if $\varepsilon_t(\omega^t)=0$, then $\PP_1= \widehat{\PP}_1$, which hence by the above inequality \eqref{eq:radius-P2} ensures that in this case $\PP_2=\widehat{\PP}_2 \in
\mathcal{B}_{\varepsilon_t(\widetilde{\omega}^t)}^{(q)} \left(\widehat{\PP}_2\right)$. Now if $\varepsilon_t(\omega^t)\neq 0$, then \eqref{eq:radius-P2} assures that indeed
\begin{align*}
	\operatorname{d}_{W_q} \left(\widehat{\PP}_2, \PP_2 \right)
	\leq  (1-\lambda)\operatorname{d}_{W_q} \left(\widehat{\PP}_1, \PP_1 \right)
\leq \big(1-\tfrac{\max\{\varepsilon_t(\omega^t)-\varepsilon_t(\widetilde{\omega}^t),0\}}{\varepsilon_t(\omega^t)}\big) \varepsilon_t(\omega^t)
\leq \varepsilon_t(\widetilde{\omega}^t).
\end{align*}
Next, we claim that 
\begin{equation}\label{eq_wasserstein_smaller}
\operatorname{d}_{W_q} \left(\PP_1, \PP_2 \right)
 \leq 
 	\operatorname{d}_{W_q} \left(\widehat{\PP}_1, \widehat{\PP}_2 \right)
 	+\lambda \operatorname{d}_{W_q} \left(\widehat{\PP}_1, \PP_1 \right).
\end{equation}
Indeed, by construction, the third and fourth marginals  of $\widetilde{\gamma}$ are $\PP_1$ and $\PP_2$, respectively.
Hence, by \eqref{eq_v_ineq_1}, Minkowski's inequality, and by the definition of $\gamma$ we have
\begin{align*}
\operatorname{d}_{W_q} \left({\PP}_1, \PP_2 \right) 
&\leq  
\bigg(\int_{\Omloc \times \Omloc \times \Omloc \times \Omloc} \| c-d\|^q \,\widetilde{\gamma}(\D a, \D b, \D c, \D d)\bigg)^{\nicefrac{1}{q}}\\
&= 
 \bigg(\int_{\Omloc \times \Omloc \times \Omloc} \| c-v^{(\lambda)}(a,b,c)\|^q \,{\gamma}(\D a, \D b, \D c)\bigg)^{\nicefrac{1}{q}} \\
&\leq 
\bigg( \int_{\Omloc \times \Omloc \times \Omloc} \big(\| b-a \| + \lambda \|c-a \|\big)^q {\gamma}(\D a, \D b, \D c)\bigg)^{\nicefrac{1}{q}}\\
&\leq 
\bigg( \int_{\Omloc \times \Omloc \times \Omloc} \| b-a \|^q \, {\gamma}(\D a, \D b, \D c)\bigg)^{\nicefrac{1}{q}}
	+
	\bigg( \int_{\Omloc \times \Omloc \times \Omloc} \lambda^q\|c-a \|^q {\gamma}(\D a, \D b, \D c)\bigg)^{\nicefrac{1}{q}}\\
& = \operatorname{d}_{W_q} \left(\widehat{\PP}_1, \widehat{\PP}_2 \right)
	+\lambda \operatorname{d}_{W_q} \left(\widehat{\PP}_1, \PP_1 \right).
\end{align*}
Finally, recall that  $\widehat{\PP}_1:= \widehat{\PP}_t(\omega^t)$, $\widehat{\PP}_2:= \widehat{\PP}_t(\widetilde{\omega}^t)$. If $\varepsilon_t(\omega^t)=0$, then $\lambda=0$ by \eqref{eq:lambda} and hence the assumption that
	$\operatorname{d}_{W_q}(\widehat{\PP}_t(\omega^t), \widehat{\PP}_t(\widetilde{\omega}^t)) \leq L_{\widehat{\mathbb{P}},t} \cdot \left( \sum_{i=1}^t \|\omega_i - \widetilde{\omega}_i \|  \right)$ implies \eqref{eq_condition_P}.  
If $\varepsilon_t(\omega^t)\neq 0$, then	the assumption that
 $|\varepsilon_t(\omega^t)-\varepsilon_t(\widetilde{\omega}^t)| \leq L_{\varepsilon,t} \cdot \left( \sum_{i=1}^t \|\omega_i - \widetilde{\omega}_i \|  \right)$  and the definition of $\lambda$ given in \eqref{eq:lambda}
 assure that
\begin{equation}\label{eq_final_argument_proof_wasserstein}
	\begin{split}
\operatorname{d}_{W_1}({\PP}_1,{\PP}_2) 
\leq
\operatorname{d}_{W_{\max\{1,p\}}}({\PP}_1,{\PP}_2)
&\leq   \operatorname{d}_{W_q}({\PP}_1,{\PP}_2)  \\
&\leq	\operatorname{d}_{W_q} \left(\widehat{\PP}_1, \widehat{\PP}_2 \right)
+\lambda \operatorname{d}_{W_q} \left(\widehat{\PP}_1, \PP_1 \right)\\
&\leq  L_{\widehat{\mathbb{P}},t} \cdot \left( \sum_{i=1}^t \|\omega_i - \widetilde{\omega}_i \| \right) 
+ \tfrac{\max\{\varepsilon_t(\omega^t)-\varepsilon_t(\widetilde{\omega}^t),0\}}{\varepsilon_t(\omega^t)} \cdot \varepsilon_t(\omega^t)\\
&\leq 
(L_{\widehat{\mathbb{P}},t} + L_{\varepsilon,t})  \cdot \left( \sum_{i=1}^t \|\omega_i - \widetilde{\omega}_i \| \right).
\end{split}
\end{equation}
This shows 
that Assumption~\ref{asu_P}~(iii) holds with $L_{\mathcal{P},t}:=L_{\widehat{\mathbb{P}},t} + L_{\varepsilon,t}\geq 0$, as well as by 
Lemma~\ref{lem:Lipschitz-implies-LHC}
that  $\mathcal{P}_t:\Omega^t \twoheadrightarrow (\mathcal{M}_1^p(\Omloc),\tau_p)$ is lower hemicontinuous.
}
\item[(iv)] The ambiguity set $\mathcal{P}_0$ contains, by definition, the reference measure $\widehat{\PP}_0$ and is hence non-empty. Moreover, by, e.g., \cite[Lemma 1]{yue2022linear} we have $\mathcal{P}_0 \subseteq  \mathcal{M}_1^q(\Omloc)\subseteq \mathcal{M}_1^{\max\{1,p\}}(\Omloc)$ since $q>p$.\\
The compactness of $\mathcal{P}_0$ w.r.t.\,$\tau_p$ follows from Lemma~\ref{lem_wasserstein_compact}. Let $\PP \in \mathcal{P}_0$ and denote by $\pi_p \in \Pi\left(\PP,\widehat{\PP}_0\right)$ the optimal coupling between $\PP$ and $\widehat{\PP}_0$ w.r.t.\ the $p$-Wasserstein distance $d_{W_p}$. {{Then, by Minkowski's inequality we have
\begin{equation}\label{eq_prop_22_part_4}
\begin{aligned}
\bigg(\int_{\Omloc} \|x\|^p\PP( \D x)\bigg)^{\nicefrac{1}{p}} &= \int_{\Omloc \times \Omloc} \| x \|^p \pi_p( \D  x, \D y) \\
&\leq 
\bigg(\int_{\Omloc \times \Omloc}\|x-y\|^p \pi_p (\D x, \D y)\bigg)^{\nicefrac{1}{p}}
+
\bigg(\int_{\Omloc}\|y\|^p \widehat{\PP}_0 (\D y)\bigg)^{\nicefrac{1}{p}}  \\
&= d_{W_p}(\PP, \widehat{\PP}_0)
+ 
\bigg(\int_{\Omloc}\|y\|^p \widehat{\PP}_0 (\D y)\bigg)^{\nicefrac{1}{p}}  \\
&\leq \varepsilon_0+ 
\bigg(\int_{\Omloc}\|y\|^p \widehat{\PP}_0 (\D y)\bigg)^{\nicefrac{1}{p}} < \infty.
\end{aligned}
\end{equation}
}}
\end{itemize}

\end{proof}
%
%
{
\subsubsection{Proofs of Section~\ref{sec_ambiguity_parametric}}
In this section, we provide the proof of Theorem~\ref{thm_param_assumptions} which provides a general setup for parametric uncertainty satisfying Assumption~\ref{asu_P}.
  Moreover, we provide the proofs of Proposition~\ref{prop_normal_family} and Proposition~\ref{prop_exponential_family} being concrete examples satisfying the setup of Theorem~\ref{thm_param_assumptions}.

We start with the following crucial lemma which states that set-valued maps in Euclidean spaces defined as balls around a reference element, where both the reference element and the radius depend (Lipschitz-) continuously on the input, are compact and continuous, as well as satisfy a stability/continuity condition like \eqref{eq_condition_P} and \eqref{eq_condition_A}. 
%
\begin{lem}\label{lem_parameter-ball-correspondence}
Let $\mathcal{X}$ be a metric space, let $s\in \N$, and let $\mathbb{S}\subseteq \R^s$ be non-empty, closed, and convex. Moreover, let $\widehat{\theta}: \mathcal{X} \to \mathbb{S}$ be a continuous map  and let
$\varepsilon: \mathcal{X} \to [0,\overline{\varepsilon}]$ be a continuous map with respect to some fixed $\overline{\varepsilon}>0$. Define the correspondence $\Gamma:\mathcal{X}\twoheadrightarrow \mathbb{S}$ by
\begin{equation*}
	\mathcal{X} \ni x \twoheadrightarrow  
	\Gamma(x):=\Big\{\theta \in \mathbb{S}\,\Big|\, \|\theta -\widehat{\theta}(x)\|\leq \varepsilon(x) \Big\}.
\end{equation*}
Then the following holds.
\begin{enumerate}
	\item[(i)]  $\Gamma:\mathcal{X}\twoheadrightarrow \mathbb{S}$ is non-empty, compact-valued, and continuous.
\end{enumerate}
If, in addition to the above assumptions, both maps $\widehat{\theta}: \mathcal{X} \to \mathbb{S}$   and 
$\varepsilon: \mathcal{X} \to [0,\overline{\varepsilon}]$ are Lipschitz-continuous with 
 Lipschitz-constants $L_{\widehat{\theta}}\geq 0$ and $L_{\varepsilon}\geq 0$, respectively, then the following holds additionally.
\begin{enumerate}
	\item[(ii)] 
	For all $x, \widetilde{x} \in \mathcal{X}$ and for all $\theta \in \Gamma(x)$ there exists some $\widetilde{\theta} \in \Gamma(\widetilde{x})$ such that
	\begin{equation}
		\label{eq_condition_Gamma}
		\Vert \theta -\widetilde{\theta}\Vert  \leq 
	(L_{\widehat{\theta}}+L_{\varepsilon})	\|x-\widetilde{x}\|.
	\end{equation}
\end{enumerate}
\end{lem}
\begin{proof}[Proof of Lemma~\ref{lem_parameter-ball-correspondence}]
To see that Lemma~\ref{lem_parameter-ball-correspondence}~(i) holds, note 
	first that non-emptiness of $\Gamma(x)$ for each $x \in \mathcal{X}$ follows by definition. Now, to see that 
	 $\Gamma:\mathcal{X}\twoheadrightarrow \mathbb{S}$ is compact-valued and upper-hemicontinuous, we aim to apply \cite[Theorem~17.20]{Aliprantis}. To that end, let $\big((x^n,\theta^n)\big)_{n \in \N}\subseteq \operatorname{Gr} \Gamma$, where $\operatorname{Gr} \Gamma:=\big\{(x,\theta) \in \mathcal{X}\times \mathbb{S}\,\big|\, \theta \in \Gamma(x)\}$ denotes the \emph{graph} of $\Gamma$, and let $x \in \mathcal{X}$ such that $\lim_{n \to \infty} d_{\mathcal{X}}(x^n,x)=0$, where $d_{\mathcal{X}}(\cdot,\cdot)$ denotes the metric on $\mathcal{X}$.
	 Note that for every $n \in \N$
	 \begin{equation*}
	 	\Vert \theta^n -\widehat{\theta}(x)\Vert 
	 	\leq 
	 	\Vert \theta^n -\widehat{\theta}(x^n)\Vert  + \Vert \widehat{\theta}(x^n) -\widehat{\theta}(x)\Vert 
	 	\leq \varepsilon(x^n) + \Vert \widehat{\theta}(x^n) -\widehat{\theta}(x)\Vert.
	 \end{equation*}
	 Since by assumption both $\varepsilon: \mathcal{X} \to [0,\overline{\varepsilon}]$ and $\widehat{\theta}: \mathcal{X} \to \mathbb{S}$ are continuous, we 
	 see that for every $n \in \N$ sufficiently large, we have that
	 \begin{equation*}
	 	\Vert \theta^n -\widehat{\theta}(x)\Vert 
	 	\leq 2\overline{\varepsilon}.
	 \end{equation*}
By applying the Bolzano–Weierstrass theorem, we hence derive that there is a subsequence  $(\theta^{n_k})_{k\in \N}$ which converges to some $\theta \in \mathbb{S}$ since $\mathbb{S}\subseteq \R^s$ is closed. Moreover, again by the continuity of $\varepsilon: \mathcal{X} \to [0,\overline{\varepsilon}]$ and $\widehat{\theta}: \mathcal{X} \to \mathbb{S}$, we deduce that
\begin{equation*}
	\Vert \theta- \widehat{\theta}(x) \Vert 
	=
	 \lim_{k \to \infty} \Vert \theta^{n_k}- \widehat{\theta}(x^{n_k}) \Vert 
	 \leq
	 \lim_{k \to \infty} \varepsilon(x^{n_k})
	 = \varepsilon(x).
\end{equation*}
By applying \cite[Theorem~17.20]{Aliprantis}, we conclude that  $\Gamma:\mathcal{X}\twoheadrightarrow \R^s$ is indeed compact-valued and upper hemicontinuous.

To prove the lower hemicontinuity of  $\Gamma:\mathcal{X}\twoheadrightarrow \R^s$, we aim to apply \cite[Theorem~17.21]{Aliprantis}.
To that end, define the set-valued map $\accentset{\circ}{\Gamma}:\mathcal{X} \twoheadrightarrow \mathbb{S}$ by setting for every $x \in \mathcal{X}$
\begin{equation*}
	\accentset{\circ}{\Gamma}(x):=
	\begin{cases}
		\Big\{\theta \in \mathbb{S}\,\Big|\, \|\theta -\widehat{\theta}(x)\|< \varepsilon(x) \Big\}   & \mbox{ if } \ \varepsilon(x)>0,\\
		\widehat{\theta}(x)    &  \mbox{ if } \ \varepsilon(x)=0.
	\end{cases} 
\end{equation*}
Note that by construction, we have for each $x \in \mathcal{X}$ that the closure of $\accentset{\circ}{\Gamma}(x)$ is equal to $\Gamma(x)$. Therefore, by \cite[Lemma~17.22]{Aliprantis}, it suffices to prove that $\accentset{\circ}{\Gamma}:\mathcal{X} \twoheadrightarrow \mathbb{S}$ is lower hemicontinuous in order to prove the lower hemicontinuity of $\Gamma:\mathcal{X}\twoheadrightarrow \R^s$.
 
To that end, let $(x^n)_{n \in \N}\subseteq \mathcal{X}$ and $x\in\mathcal{X}$ such that  $\lim_{n \to \infty} d_{\mathcal{X}}(x^n,x)=0$ and let $\theta \in \accentset{\circ}{\Gamma}(x)$.

If $\varepsilon(x)= 0$, then $\theta = \widehat{\theta}(x)$. Hence in this case, the continuity of $\widehat{\theta}: \mathcal{X} \to \mathbb{S}$ implies that the sequence  $\widehat{\theta}(x^n)\in \accentset{\circ}{\Gamma}(x^n)$, $n \in\N$, converges to $\widehat{\theta}(x)=\theta$.

If $\varepsilon(x)\neq  0$, then $\accentset{\circ}{\Gamma}(x)=\Big\{\theta \in \mathbb{S}\,\Big|\, \|\theta -\widehat{\theta}(x)\|< \varepsilon(x) \Big\}$.
 In this case, $\accentset{\circ}{\Gamma}(x)\subseteq \R^s$ is open, hence  there exists some $0< \delta< \varepsilon(x)$ such that $\|\theta -\widehat{\theta}(x)\|< \varepsilon(x)-\delta$.
  We define for each $n \in \N$ 
\[
\theta^{n}:=
\begin{cases}
	\widehat{\theta}(x^n),&\text{ if } \
	\|\widehat{\theta}(x^n)-\widehat{\theta}(x)\Vert \geq \tfrac{\delta}{2}, \\
	\theta, &\text{else}.
\end{cases}
\]
Now, by continuity of $\varepsilon: \mathcal{X} \to [0,\overline{\varepsilon}]$, there exists an $N_1\in \N$ such that for all $n \geq N_1$ we have  $|\varepsilon(x^n)-\varepsilon(x)|<\tfrac{\delta}{2}$.
We claim that $\theta^{n} \in  \accentset{\circ}{\Gamma}(x^n)$ for every $n \geq N_1$. Indeed, note first that if $n \geq N_1$ such that	$\|\widehat{\theta}(x^n)-\widehat{\theta}(x)\Vert \geq \tfrac{\delta}{2}$, then  $\theta^{n}=	\widehat{\theta}(x^n) \in \accentset{\circ}{\Gamma}(x^n)$. Moreover, if $n \geq N_1$ such that $\|\widehat{\theta}(x^n)-\widehat{\theta}(x)\Vert < \tfrac{\delta}{2}$ 
then  $\theta^{n}=\theta$, and hence by the triangle inequality
\begin{equation*}
\Vert \theta^n -\widehat{\theta}(x^n)\Vert
= \Vert \theta -\widehat{\theta}(x^n)\Vert
\leq  \Vert \theta -\widehat{\theta}(x)\Vert + \Vert \widehat{\theta}(x) -\widehat{\theta}(x^n)\Vert
 < (\varepsilon(x)-\delta)+\tfrac{\delta}{2}
 \leq (\varepsilon(x^n)-\tfrac{\delta}{2})+\tfrac{\delta}{2}
 =\varepsilon(x^n),
\end{equation*}
which demonstrates that $\theta^n \in \accentset{\circ}{\Gamma}(x^n)$.
Moreover, by the continuity of the map $\widehat{\theta}: \mathcal{X} \to \mathbb{S}$, there exists some $N\geq N_1$ such that for all $n \geq N$ we have 
$\|\widehat{\theta}(x^n)-\widehat{\theta}(x)\Vert < \tfrac{\delta}{2}$.
Therefore, by construction, we have  for all $n \geq N$ that $\theta^n=\theta$ and thus, in particular $\theta^n \to \theta$ for $n \rightarrow \infty$. The 
 lower hemicontinuity of $\accentset{\circ}{\Gamma}$  (and hence also of $\Gamma$) now follows from ~\cite[Theorem 17.21]{Aliprantis}.
 
Next we aim to prove that Lemma~\ref{lem_parameter-ball-correspondence}~(ii) holds. 
To that end, let $x,\widetilde{x}\in \mathcal{X}$ and let $\theta \in \Gamma(x)$. 
For the ease of notation we set 
$a : = \widehat{\theta}(x) \in \mathcal{X}$,
$b:= \widehat{\theta}(\widetilde{x}) \in \mathcal{X}$,
$c:= \theta \in \mathcal{X}$ and define
\begin{equation}
	\label{eq:lambda-param}
	\lambda:=\begin{cases}
		\frac{\max\{\varepsilon(x)-\varepsilon(\widetilde{x}),0\}}{\varepsilon(x)}   &  \mbox{ if } \varepsilon(x) \neq 0,\\
		0							& \mbox{ else.}
	\end{cases}
\end{equation}
Moreover, we let $v^{(\lambda)}:\mathbb{S} \times \mathbb{S} \times \mathbb{S} \to \mathbb{S}$ be the
map defined in \eqref{eq_defn_v} in Lemma~\ref{lem_EuclideanGeometry} with respect to that chosen $\lambda \in [0,1]$, 
and then set
\begin{equation*}
	\widetilde{\theta}:= v^{(\lambda)}(a,b,c)
	=v^{(\lambda)}\big(\widehat{\theta}(x),\widehat{\theta}(\widetilde{x}),\theta\big).
\end{equation*}
By \eqref{eq_v_ineq_2} in Lemma~\ref{lem_EuclideanGeometry} and the choice of $\lambda$ in \eqref{eq:lambda-param}, we have that
\begin{equation}
	\label{eq_v_ineq_2-param}
	\begin{split}
		\|\widetilde{\theta}-\widehat{\theta}(\widetilde{x})\|
		=  
		\big\|v^{(\lambda)}\big(\widehat{\theta}(x),\widehat{\theta}(\widetilde{x}),\theta\big)-\widehat{\theta}(\widetilde{x}) \big\|
		\leq (1-\lambda) \big\|\theta-\widehat{\theta}(x) \big\|
		\leq \varepsilon(\widetilde{x}).
	\end{split}
\end{equation}
%
Moreover, by \eqref{eq_v_ineq_1} in Lemma~\ref{lem_EuclideanGeometry}, 
the choice of $\lambda$ in \eqref{eq:lambda-param}, and the Lipschitz-continuity of  both maps $\widehat{\theta}: \mathcal{X} \to \mathbb{S}$   and 
$\varepsilon: \mathcal{X} \to [0,\overline{\varepsilon}]$, 
we see that
\begin{equation}
	\label{eq_v_ineq_1-param}
	\begin{split}
		\| \widetilde{\theta} -\theta\| 
		&\leq
		\| \widehat{\theta}(\widetilde{x})- \widehat{\theta}(x)\| 
		+ \lambda \|\theta -\widehat{\theta}(x)\| \\
		&\leq
		\| \widehat{\theta}(\widetilde{x})- \widehat{\theta}(x)\| 
		+ \lambda \varepsilon(x)\\
		&	\leq
		\| \widehat{\theta}(\widetilde{x})- \widehat{\theta}(x)\| 
		+ |\varepsilon(\widetilde{x})-\varepsilon(x)|\\
		&	\leq
		(L_{\widehat{\theta}} + L_{\varepsilon})\| \widetilde{x}-x\|.
	\end{split}
\end{equation}
\end{proof}
}

\begin{proof}[Proof of Theorem~\ref{thm_param_assumptions}]
We verify all four properties of Assumption~\ref{asu_P}.\\
To see that Assumption~\ref{asu_P}~(iv) holds, note that we can w.l.o.g. assume that $p \geq 1$, as for $p=0$, Assumption~\ref{asu_P}~(iv) holds trivially. 
{{For the case $p\geq 1$ let $\PP \in \mathcal{P}_0$. Then, by definition of $\mathcal{P}_{0}$ in \ref{eq_a5}, there exists some $\theta \in \Theta_0$ with $\|\theta - \widehat{\theta}_0 \| \leq \varepsilon_0$ such that $\PP  \equiv \PP_\theta$. Next, let $\Pi_0$ be the optimal coupling of $\PP_\theta$ and $\PP_{\widehat{\theta}_0}$ w.r.t.\ the $p$-Wasserstein distance $d_{W_p}$. Then, by using Minkowski's inequality and \ref{eq_a2} we obtain that
\begin{align*}
\bigg(\int_{\Omloc} \|x\|^p \PP(\D x)\bigg)^{\nicefrac{1}{p}} 
&= 
\bigg(\int_{\Omloc \times \Omloc} \|x \|^p \Pi_0(\D x,\D y)\bigg)^{\nicefrac{1}{p}}  \\
&\leq \bigg( \int_{\Omloc} \|y\|^p \PP_{\widehat{\theta}_0} (\D y) \bigg)^{\nicefrac{1}{p}}  
+
\operatorname{d}_{W_p}(\PP_\theta, \PP_{\widehat{\theta}_0})\\
&\leq 
\bigg( \int_{\Omloc} \|y\|^p \PP_{\widehat{\theta}_0} (\D y) \bigg)^{\nicefrac{1}{p}}  
+ 
 L_{P_\theta,t}  \cdot  \|\theta-\widehat{\theta}_0\| \\
&\leq \bigg( \int_{\Omloc} \|y\|^p \PP_{\widehat{\theta}_0} (\D y) \bigg)^{\nicefrac{1}{p}}
 +
 L_{P_\theta,t} \cdot \varepsilon_0 
 =: C_{\mathcal{P},0}^{\nicefrac{1}{p}}<\infty.
\end{align*}
}}
Moreover, $\Theta_0^{(\varepsilon_0)} :=\left\{\theta \in \Theta_0~\middle|~\|\theta-\widehat{\theta}_0 \| \leq \varepsilon_0 \right\} \subseteq \R^{D_0}$ is compact as $\Theta_0$ is closed.
The map
\begin{align*}
\Psi:\Theta_0 &\rightarrow \left(\mathcal{M}_1^p(\Omloc), \tau_p\right)\\
\theta &\mapsto \PP_{\theta}
\end{align*}
is continuous by \ref{eq_a2}, and $\mathcal{P}_0 = \Psi\left(\Theta_0^{\varepsilon_0}\right)$ is an image of a compact set under a continuous map and thus compact.
The non-emptiness of $\mathcal{P}_0$ follows as $\Theta_0^{(\varepsilon_0)}$ is non-empty by \ref{eq_a1}. Hence, we have shown that Assumption~\ref{asu_P}~(iv) is satisfied. 

{Next, we show that Assumption~\ref{asu_P}~(i) is fulfilled. To this end, let $t\in \{1,\dots,T-1\}$ be fixed. The non-emptiness of $\mathcal{P}_{t}(\omega^t)$ for each $\omega^t \in \Omega^t$ follows by definition and by \ref{eq_a1}. We define the correspondence
\begin{equation}
	\label{eq:P-param-def-Theta-t}
	\begin{split}
\Omega^t &\twoheadrightarrow  \Theta_t \\
\omega^t &\twoheadrightarrow \Theta_t^{(\varepsilon_t)}(\omega^t):=\left\{\theta \in \Theta_t~\middle|~ \|\theta-\widehat{\theta}_t(\omega^t) \| \leq \varepsilon_t(\omega^t) \right\}.
\end{split}
\end{equation}
Since by assumptions~\ref{eq_a3}~and~\ref{eq_a4} the maps $\Omega^t \ni \omega^t \mapsto \widehat{\theta}_t(\omega^t)$ and $\Omega^t \ni \omega^t \mapsto \varepsilon_t(\omega^t)$ are continuous, and by assumption~\ref{eq_a1} $\Theta_t\subseteq \R^{D_t}$ is non-empty, convex, and closed,  Lemma~\ref{lem_parameter-ball-correspondence} guarantees that $\Theta_t^{(\varepsilon_t)}: \Omega^t \twoheadrightarrow \R^{D_t}$  is a non-empty, compact-valued, continuous correspondence. Moreover,  for any $\omega^t\in \Omega^t$ we have 
\begin{equation}\label{eq:P-param-Ass1-identity}
	\mathcal{P}_t(\omega^t)=\left\{\PP_{\theta} \in \mathcal{M}_1^{\max\{1,p\}}(\Omloc)~\middle|~ \theta \in \Theta_t^{(\varepsilon_t)}(\omega^t)\right\}.
	\end{equation}
 Now, to see that
	$\mathcal{P}_{t}: \Omega^t\twoheadrightarrow \mathcal{M}_1^{\max\{1,p\}}(\Omloc)$ is compact-valued and upper-hemicontinuous, we aim to apply \cite[Theorem~17.20]{Aliprantis}. To that end, let $\left({\omega^t}^{(n)},\PP^{(n)}\right)_{n \in \N} \subseteq \operatorname{Gr} \mathcal{P}_t$, where $\operatorname{Gr} \mathcal{P}_t$ denotes the \emph{graph} of $\mathcal{P}_t$ and let $\omega^t \in \Omloc$ such that ${\omega^t}^{(n)} \to \omega^t$ when $n \to \infty$.
	Then by definition of $\operatorname{Gr} \mathcal{P}_t$, we have for each $n \in \N$ that there exists $\theta^n \in \Theta_t^{(\varepsilon_t)}({\omega^t}^{(n)})$ such that $\PP^{(n)}=\PP_{\theta^n}$. 
	In particular, we see that for each $n \in \N$ we have $\left({\omega^t}^{(n)},\theta^n\right)_{n \in \N} \subseteq \operatorname{Gr} \Theta_t^{(\varepsilon_t)}$. Since we have shown above that $\Theta_t^{(\varepsilon_t)}: \Omega^t \twoheadrightarrow  \Theta_t$ is a non-empty, compact-valued, and upper hemicontinuous correspondence, \cite[Theorem~17.20]{Aliprantis} ensures that there exists a subsequence $(\theta^{n_k})_{k\in \N}$ converging to some $\theta \in  \Theta_t^{(\varepsilon_t)}(\omega^t)$. This, the fact that $\PP^{(n_k)}=\PP_{\theta^{n_k}}$ for each $k \in \N$, and that by assumption~\ref{eq_a2} the map $(\Theta_t\ni \theta \mapsto \PP_\theta)$ is continuous imply that the subsequence $(\PP^{(n_k)})_{k \in \N}$ converges to $\PP_{\theta}\in \mathcal{M}_1^{\max\{1,p\}}(\Omloc)$. The identity \eqref{eq:P-param-Ass1-identity} now ensures that $\PP_{\theta}\in 	\mathcal{P}_t(\omega^t)$, hence  \cite[Theorem~17.20]{Aliprantis} indeed guarantees that $\mathcal{P}_{t}: \Omega^t\twoheadrightarrow \mathcal{M}_1^{\max\{1,p\}}(\Omloc)$ is compact-valued and upper-hemicontinuous. 
	Finally, the lower hemicontinuity  
	will automatically follow once we have proven below that Assumption~\ref{asu_P}~(iii) holds, but where we show 
	the stronger  statement
	that \eqref{eq_condition_P} holds 
	with respect to $d_{W_{\max\{1,p\}}}$ instead of $d_{W_1}$, cf.\ 
	Lemma~\ref{lem:Lipschitz-implies-LHC}.
}
\\

To see that Assumption~\ref{asu_P}~(ii) holds we again assume w.l.o.g. $p \geq 1$. In this case, let $\omega^t=(\omega^t_1,\dots,\omega^t_t)\in \Omega^t$ and $\PP \in \mathcal{P}_t(\omega^t)$. By definition of $\mathcal{P}_t$ in \ref{eq_a4}, there exists some $\theta \in \Theta_t$ such that $\| \theta - \widehat{\theta}_t(\omega^t) \| \leq \varepsilon_t(\omega^t)$ and such that $\PP = \PP_{\theta}$. Let $\pi_t(\D x, \D y)$ denote an optimal coupling of $\PP$ and $\PP_{{\widehat{\theta}_t(\omega^t)}}$ with respect to the Wasserstein-$p$ distance $d_{W_p}$. Then, 
{{by Minkowski's inequality we have 
\begin{equation}\label{eq_proof_param_1}
\begin{aligned}
\bigg(\int_{\Omloc} \|x \|^p \PP(\D x)\bigg)^{\nicefrac{1}{p}}
 &= 
 \bigg(\int_{\Omloc \times \Omloc} \|x \|^p \pi_t(\D x, \D y)\bigg)^{\nicefrac{1}{p}} \\
 &\leq 
 \bigg(\int_{\Omloc \times \Omloc} \|x-y\|^p \pi_t(\D x, \D y)\bigg)^{\nicefrac{1}{p}} 
 + 
 \bigg(\int_{\Omloc} \|y \|^p\PP_{{\widehat{\theta}_t(\omega^t)}} (\D y)\bigg)^{\nicefrac{1}{p}}
\\
&= 
\operatorname{d}_{W_p}\left(\PP,\PP_{{\widehat{\theta}_t(\omega^t)}} \right)
 + 
 \bigg(\int_{\Omloc} \|y \|^p\PP_{{\widehat{\theta}_t(\omega^t)}} (\D y)\bigg)^{\nicefrac{1}{p}}.
\end{aligned}
\end{equation}
By Assumption \ref{eq_a2}, as $p \geq 1$, we have 
\begin{equation}\label{eq_proof_param_2}\operatorname{d}_{W_p}\left(\PP,\PP_{{\widehat{\theta}_t(\omega^t)}} \right)
	=
	\operatorname{d}_{W_p}\left(\PP_\theta,\PP_{{\widehat{\theta}_t(\omega^t)}} \right)
	 \leq
	  L_{P_{\theta},t} \cdot \| \theta - \widehat{\theta}_t(\omega^t) \|
	   \leq 
	   { L_{P_{\theta},t} \cdot \varepsilon_t(\omega^t)
	    \leq 
	    L_{P_{\theta},t} \cdot \overline{\varepsilon}_t. }
\end{equation}
Next, 
for any arbitrary $\omega_{\operatorname{ref}}^{*,t} = (\omega_{\operatorname{ref},1}^{*},\dots,\omega_{\operatorname{ref},t}^{*}) \in \Omega^t$,
  let $\widetilde{\pi}(\D y,\D x)$ be an optimal coupling of $\PP_{{\widehat{\theta}_t(\omega^t)}}$ and $\PP_{{\widehat{\theta}_t(\omega_{\operatorname{ref}}^{*,t})}} $ with respect to $d_{W_p}$. 
  Then, due to Minkowski's inequality, \ref{eq_a2}, and \ref{eq_a3}, we have
\begin{equation*}
\begin{aligned}
\bigg(\int_{\Omloc} \|y \|^p \PP_{{\widehat{\theta}_t(\omega^t)}} (\D y)  \bigg)^{\nicefrac{1}{p}}
&= \bigg(\int_{\Omloc \times \Omloc} \|y\|^p\widetilde{\pi}(\D y,\D z)\bigg)^{\nicefrac{1}{p}} \\
&\leq
 \bigg(\int_{\Omloc}  \|z\|^p \PP_{{\widehat{\theta}_t(\omega_{\operatorname{ref}}^{*,t})}} (\D z)\bigg)^{\nicefrac{1}{p}}
 +
 \operatorname{d}_{W_p}\left(\PP_{{\widehat{\theta}_t(\omega^t)}},~\PP_{{\widehat{\theta}_t(\omega_{\operatorname{ref}}^{*,t})}}\right)\\
&\leq
\bigg(\int_{\Omloc}  \|z\|^p \PP_{{\widehat{\theta}_t(\omega_{\operatorname{ref}}^{*,t})}} (\D z)\bigg)^{\nicefrac{1}{p}}
 + 
  L_{P_{\theta},t}  \cdot \left \|\widehat{\theta}_t(\omega^t)- \widehat{\theta}_t(\omega_{\operatorname{ref}}^{*,t})\right\| \\
&\leq 
\bigg(\int_{\Omloc}  \|z\|^p \PP_{{\widehat{\theta}_t(\omega_{\operatorname{ref}}^{*,t})}} (\D z)\bigg)^{\nicefrac{1}{p}}
+
 L_{P_{\theta},t}  \cdot L_{\widehat{\theta},t} \cdot \sum_{i=1}^t \|\omega_i- \omega_{\operatorname{ref},i}^{*}\|  \\
 &\leq 
 \bigg(\int_{\Omloc}  \|z\|^p \PP_{{\widehat{\theta}_t(\omega_{\operatorname{ref}}^{*,t})}} (\D z)\bigg)^{\nicefrac{1}{p}}
+
L_{P_{\theta},t}  \cdot L_{\widehat{\theta},t} \cdot \sum_{i=1}^t \big(\| \omega_{\operatorname{ref},i}^{*}\|  + \|\omega_i\|\big). 
\end{aligned}
\end{equation*}
Since  $\omega_{\operatorname{ref}}^{*,t} = (\omega_{\operatorname{ref},1}^{*},\dots,\omega_{\operatorname{ref},t}^{*}) \in \Omega^t$ was arbitrarily chosen, we obtain that
\begin{equation}
	\begin{aligned}
	\label{eq_proof_param_3}
&\bigg(\int_{\Omloc} \|y \|^p \PP_{{\widehat{\theta}_t(\omega^t)}} (\D y)  \bigg)^{\nicefrac{1}{p}}\\
&\leq 
 \inf_{\omega_{\operatorname{ref}}^{*,t}\in \Omega^t}\bigg\{\bigg(\int_{\Omloc}  \|z\|^p \PP_{{\widehat{\theta}_t(\omega_{\operatorname{ref}}^{*,t})}} (\D z)\bigg)^{\nicefrac{1}{p}}
+
L_{P_{\theta},t}  \cdot L_{\widehat{\theta},t} \cdot \sum_{i=1}^t\| \omega_{\operatorname{ref},i}^{*}\| \bigg\}
+
L_{P_{\theta},t}  \cdot L_{\widehat{\theta},t} \cdot \sum_{i=1}^t  \|\omega_i\|. 
\end{aligned}
\end{equation}
Therefore, we see that \eqref{eq_proof_param_1}, \eqref{eq_proof_param_2}, and \eqref{eq_proof_param_3} together imply that 
\begin{align*}
\int_{\Omloc} \|x \|^p \PP(\D x) 
&\leq
2^{p-1}\Bigg(
L_{P_{\theta},t} \cdot {\overline{\varepsilon}_t}
+
 \inf_{\omega_{\operatorname{ref}}^{*,t}\in \Omega^t}\bigg\{\bigg(\int_{\Omloc}  \|z\|^p \PP_{{\widehat{\theta}_t(\omega_{\operatorname{ref}}^{*,t})}} (\D z)\bigg)^{\nicefrac{1}{p}}
+
L_{P_{\theta},t}  \cdot L_{\widehat{\theta},t} \cdot \sum_{i=1}^t\| \omega_{\operatorname{ref},i}^{*}\| \bigg\}\Bigg)^p\\
& \quad
+
2^{p-1}L_{P_{\theta},t}^p  \cdot L_{\widehat{\theta},t}^p \cdot  t^{p-1}\sum_{i=1}^t  \|\omega_i\|^p\\
&\leq C_{\mathcal{P},t} \cdot \left(1+\sum_{i=1}^t \|\omega_i\|^p\right),
\end{align*}
for 
\begin{equation*}
\begin{aligned}
C_{\mathcal{P},t}:=\max \Bigg\{&
2^{p-1}\Bigg(
L_{P_{\theta},t} \cdot {\overline{\varepsilon}_t}
+
\inf_{\omega_{\operatorname{ref}}^{*,t}\in \Omega^t}\bigg\{\bigg(\int_{\Omloc}  \|z\|^p \PP_{{\widehat{\theta}_t(\omega_{\operatorname{ref}}^{*,t})}} (\D z)\bigg)^{\nicefrac{1}{p}}
+
L_{P_{\theta},t}  \cdot L_{\widehat{\theta},t} \cdot \sum_{i=1}^t\| \omega_{\operatorname{ref},i}^{*}\| \bigg\}\Bigg)^p,\\
& \quad
,2^{p-1}L_{P_{\theta},t}^p  \cdot L_{\widehat{\theta},t}^p \cdot t^{p-1},
~1\Bigg\}<\infty,	
\end{aligned}
\end{equation*}
which shows Assumption~\ref{asu_P}~(ii).
}}

{It remains to show Assumption~\ref{asu_P}~(iii), but with respect to $d_{W_{\max\{1,p\}}}$ instead of $d_{W_1}$ in \eqref{eq_condition_P}. To that end, let $\omega^t=(\omega_1,\dots,\omega_t),~\widetilde{\omega}^t=(\widetilde{\omega}_1,\dots,\widetilde{\omega}_t) \in \Omega^t$ and let $\PP \in \mathcal{P}_t(\omega^t)$. By definition of $\mathcal{P}_t$ in \ref{eq_a4},  there exists some $\theta \in \Theta_t$ with $\| \theta - \widehat{\theta}(\omega^t)\| \leq \varepsilon_t(\omega^t)$ such that $\PP_{\theta} \equiv \PP \in \mathcal{P}_t(\omega^t)$. 
Now, recall the representation \eqref{eq:P-param-Ass1-identity} for the correspondence  $\mathcal{P}_t$ with respect to the correspondence 
$\Theta_t^{(\varepsilon_t)}$ defined in
\eqref{eq:P-param-def-Theta-t}. By definition, $\theta \in \Theta_t^{(\varepsilon_t)}(\omega^t)$. Due to \ref{eq_a3} and \ref{eq_a4}, we can  apply Lemma~\ref{lem_parameter-ball-correspondence}~(ii) (with $\Gamma \leftarrow\Theta_t^{(\varepsilon_t)}$, $\mathcal{X} \leftarrow \Omega^t$) to obtain  some $\widetilde{\theta} \in \Theta_t^{(\varepsilon_t)}(\widetilde{\omega}^t)$
such that
\begin{equation*}
	\Vert \theta -\widetilde{\theta}\Vert \leq \big(L_{\widehat{\theta},t} + L_{\varepsilon,t}\big) \cdot \sum_{i=1}^t \| \omega_i-\widetilde{\omega}_i\|.
\end{equation*}
Therefore, we see that $\PP_{\widetilde{\theta}} \in \mathcal{P}_t (\widetilde{\omega}^t)$, as well as by \ref{eq_a2} that
\begin{equation}
	\label{eq:Asu_2.3(iii)_parameter}
\begin{split}
\operatorname{d}_{W_{1}}\left(\PP_{{\theta}}, \PP_{\widetilde{\theta}} \right) 
\leq 
\operatorname{d}_{W_{\max \{1,p\}}}\left(\PP_{{\theta}}, \PP_{\widetilde{\theta}} \right) 
\leq 
L_{\PP_{{\theta}},t} \cdot \| \widetilde{\theta} -\theta\| 
\leq 
L_{\PP_{{\theta}},t} \cdot \big(L_{\widehat{\theta},t} + L_{\varepsilon,t}\big) \cdot \sum_{i=1}^t \| \omega_i-\widetilde{\omega}_i\|.
\end{split}
\end{equation}
This shows that Assumption~\ref{asu_P}~(iii) is indeed fulfilled with $L_{\mathcal{P},t}:=L_{\PP_{{\theta}},t} \cdot \big(L_{\widehat{\theta},t} + L_{\varepsilon,t}\big)\geq0$.}
\end{proof}
%
%
%
{
\begin{proof}[Proof of Proposition~\ref{prop_normal_family}]
	We verify assumptions ~\ref{eq_a1} -- \ref{eq_a5}.
\begin{itemize}
	\item[\ref{eq_a1}] By assumption
	 $\Theta_t= \R^d \times [0,\infty)^d$ is non-empty, convex, and closed.
	 \item[\ref{eq_a2}] First of all, note that for any $p\in\{0,1,2\}$ and any
	 $(\boldsymbol{\mu}, \boldsymbol{\sigma})\in \R^d \times [0,\infty)^d$ we have that $\PP_{(\boldsymbol{\mu}, \boldsymbol{\sigma})} \in \mathcal{M}_1^{\max\{1,p\}}(\R^d)$.
	Moreover, since $p\in\{0,1,2\}$, we have by \cite{dowson1982frechet} that for any $(\boldsymbol{\mu}_1, \boldsymbol{\sigma}_1), (\boldsymbol{\mu}_2, \boldsymbol{\sigma}_2) \in \R^d \times [0,\infty)^d$
	\begin{equation*}
		\begin{split}
		d_{W_{\max\{1,p\}}}\left(\PP_{(\boldsymbol{\mu}_1, \boldsymbol{\sigma}_1)},\PP_{(\boldsymbol{\mu}_2, \boldsymbol{\sigma}_2)}\right) 
		&\leq 
			d_{W_{2}}\left(\PP_{(\boldsymbol{\mu}_1, \boldsymbol{\sigma}_1)},\PP_{(\boldsymbol{\mu}_2, \boldsymbol{\sigma}_2)}\right) \\
			&= 
			\sqrt{\|\boldsymbol{\mu}_1-\boldsymbol{\mu}_2\|^2 + \mbox{tr}\big(\big(\mbox{diag}(\boldsymbol{\sigma}_1)-\mbox{diag}(\boldsymbol{\sigma}_2)\big)^2\big)}\\
				&\leq
			\|\boldsymbol{\mu}_1-\boldsymbol{\mu}_2\| + \sqrt{\mbox{tr}\big(\big(\mbox{diag}(\boldsymbol{\sigma}_1)-\mbox{diag}(\boldsymbol{\sigma}_2)\big)^2\big)}
			\\
				&= 
			\|\boldsymbol{\mu}_1-\boldsymbol{\mu}_2\| +
			\|\boldsymbol{\sigma}_1-\boldsymbol{\sigma}_2\|.
			\end{split}	
			\end{equation*}
			\item[\ref{eq_a3}] By construction, we see that both  $\widehat{\boldsymbol{\mu}}_t:\Omega^t \to \R^d$ defined in \eqref{eq_defn_mu_hat} and 
			$\widehat{\boldsymbol{\sigma}}_t=(\widehat{\sigma}_{t,1},\dots,\widehat{\sigma}_{t,d}):\Omega^t \to [0,\infty)^d$ defined in \eqref{eq_defn_sigma_hat} are Lipschitz continuous.
			\item[\ref{eq_a4}] Follows by definition of $\mathcal{P}_t$.
			\item[\ref{eq_a5}] Follows by definition of $\mathcal{P}_0$.
\end{itemize}	
\end{proof}
}
\begin{proof}[Proof of Proposition~\ref{prop_exponential_family}]
We verify assumptions ~\ref{eq_a1} -- \ref{eq_a5}.
\begin{itemize}
\item[\ref{eq_a1}] By assumption $\Theta_t= \Omloc = [0,\infty)$ is non-empty, convex, and closed.
\item[\ref{eq_a2}] Let $ t\in \{0,1,\dots,T-1\}$ and let $\theta_1,\theta_2 \in \Theta_t$ with $\theta_1 \neq \theta_2$. We distinguish two cases.\\
\underline{Case 1: $\theta_1,\theta_2 \neq 0$}\\
 Then, we have the representation
\begin{equation}\label{eq_proof_prop_exp_1}
d_{W_{\max\{1,p\}}}\left(\PP_{\theta_1},\PP_{\theta_2}\right) = \left(\int_0^1 \left|F_{\PP_{\theta_1}}^{-1}(y)-F_{\PP_{\theta_2}}^{-1}(y)\right|^{{\max\{1,p\}}} \D y \right)^{\tfrac{1}{\max\{1,p\}}},
\end{equation}
where $[0,1] \ni y \mapsto F_{\PP_{\theta_i}}^{-1}(y):=\inf\{x \in \R~|~\PP_{\theta_i}\left((-\infty,x]\right) \geq y\}$ denotes the quantile function of $\PP_{\theta_i}$ for $i=1,2$, see, e.g. \cite[Equation (3.5)]{ruschendorf2007monge} or \cite{vallender1974calculation}. For any $i=1,2$ note that the cumulative distribution function is given by $F_{\PP_{\theta_i}}(x) =1-e^{-\frac{1}{\theta_i}x},~~x\in [0,\infty)$. Hence, the quantile function computes as
\[
F_{\PP_{\theta_i}}^{-1}(y) = \begin{cases} \infty&\text{ if }y = 1,\\
-\theta_i \log(1-y), &\text{ if }y \in (0,1),\\
-\infty, &\text{ if }y = 0.
\end{cases}
\]
We set 
\begin{equation}\label{eq_defn_LPT}
L_{P_\theta,t}:= \left[\max\{1,p\}!\right]^{\tfrac{1}{\max\{1,p\}}} < \infty .
\end{equation}
By \eqref{eq_proof_prop_exp_1}, we obtain 
\begin{align*}
d_{W_{\max\{1,p\}}}\left(\PP_{\theta_1},\PP_{\theta_2}\right) &= \left(\int_0^1 \left|\theta_1 \log(1-y)-\theta_2\log(1-y)\right|^{{\max\{1,p\}}} \D y \right)^{\tfrac{1}{\max\{1,p\}}}\\
&\leq |\theta_1-\theta_2| \left(\int_0^1 |\log(1-y)|^{\max\{1,p\}}  \D y\right)^{\tfrac{1}{\max\{1,p\}}} \\
&= |\theta_1-\theta_2|  \cdot L_{P_\theta,t}.
\end{align*}
\underline{Case 2: $\theta_1 = 0$ or $\theta_2 = 0$}
\\
Without loss of generality $\theta_1 = 0$ and $\theta_2 >0$. Then, we have 
\[
F_{\PP_{\theta_1}}^{-1}(y)=F_{\delta_{\{0\}}}^{-1}(y)=\begin{cases}
0 &\text{ if } y \in (0,1],\\
-\infty &\text{ if } y = 0.
\end{cases}
\]
Therefore, by \eqref{eq_proof_prop_exp_1} and by the definition of $L_{P_\theta,t}$ in \eqref{eq_defn_LPT}, we obtain
\begin{align*}
d_{W_{\max\{1,p\}}}\left(\PP_{\theta_1},\PP_{\theta_2}\right) &= \left(\int_0^1 \left|0-\theta_2\log(1-y)\right|^{{\max\{1,p\}}} \D y \right)^{\tfrac{1}{\max\{1,p\}}}\\
&= \theta_2 \left(\int_0^1 |\log(1-y)|^{\max\{1,p\}}  \D y\right)^{\tfrac{1}{\max\{1,p\}}} \\
&= |\theta_1-\theta_2|  \cdot L_{P_\theta,t}.
\end{align*}
\item[\ref{eq_a3}] Let $ t\in \{1,\dots,T-1\}$, and let $\omega^t, \widetilde{\omega}^t \in \Omega^t$. Then, we have
\[
\left\| \widehat{\theta}_t (\omega^t)- \widehat{\theta}_t (\widetilde{\omega}^t ) \right\| \leq \frac{1}{t} \sum_{i=1}^t\|\omega_i-\widetilde{\omega}_i\|.
\]
\item[\ref{eq_a4}] Follows by definition of $\mathcal{P}_t$.
\item[\ref{eq_a5}] Follows by definition of $\mathcal{P}_0$.
\end{itemize}
\end{proof}
%
{
\subsubsection{Proofs of Section~\ref{sec_controls}}
In this section we provide the proof of Proposition~\ref{prop_controls}.

\begin{proof}[Proof of Proposition~\ref{prop_controls}]
	By assumption, $\mathcal{A}_0\subseteq \R^{m_0}$ is non-empty and compact, hence Assumption~\ref{asu_A}~(iii) is satisfied. To see that Assumptions~\ref{asu_A}~(i)~\&~(ii) are satisfied, we distinguish between the three different assumptions (i),~(ii),~and~(iii).
\begin{itemize}
	\item [(i)] Follows directly since $\mathcal{A}_t$ is a constant mapping taking value of one non-empty compact set in $\R^{m_t}$.
	\item[(ii)] For every $t \in \{1,\dots,T-1\}$   Assumption~\ref{asu_A}~(i)~\&~(ii) holds directly from Lemma~\ref{lem_parameter-ball-correspondence}~(i)~\&~(ii) (with $\Gamma\leftarrow \mathcal{A}_t$, $\mathcal{X}\leftarrow \Omega^t$ in the notation of Lemma~\ref{lem_parameter-ball-correspondence}). 
\item[(iii)] Fix $t \in \{1,\dots,T-1\}$.
To see  that Assumption~\ref{asu_A}~(i) holds, note first that by construction $\mathcal{A}_t$ is a non-empty and compact valued correspondence.
Moreover, to see that $\mathcal{A}_t$ is upper-hemicontinuous, let
$\big({\omega^t}^{(n)},a^{(n)}\big)_{n \in \N} \subseteq \operatorname{Gr} \mathcal{A}_t$, where $\operatorname{Gr} \mathcal{A}_t$ denotes the \emph{graph} of $\mathcal{A}_t$, and assume that $\big({\omega^t}^{(n)}\big)_{n \in \N}$ converges to some $\omega^t \in \Omega^t$.  
For any $n \in \N$ we denote $a^{(n)}=(a^{(n)}_{1},\dots,a^{(n)}_{m_t})\in \R^{m_t}$. 
Then, since by assumption for each $j=1,\dots, m_t$ both  $\underline{a}_{t,j}:\Omega^t\mapsto \R$ and $\overline{a}_{t,j}:\Omega^t\mapsto \R$ 
are  Lipschitz continuous, there exists a constant $L_{a}>0$ such that for each $j=1,\dots, m_t$ we have
\begin{equation*}
\begin{split}
-\infty 
&<
\underline{a}_{t,j}(\omega^t)-L_{a}\sup_{\ell\in \N} \big(\Vert {\omega^t}^{(\ell)} -{\omega^t} \Vert\big)\\
&\leq \underline{a}_{t,j}({\omega^t}^{(n)})\\
&\leq a^{(n)}_{j}\\
&\leq \overline{a}_{t,j}({\omega^t}^{(n)})\\
&\leq \overline{a}_{t,j}(\omega^t)+L_{a}\sup_{\ell\in \N} \big(\Vert {\omega^t}^{(\ell)} -{\omega^t} \Vert\big)
<\infty.
\end{split}
\end{equation*}
Therefore, we see that 
 $(a^{(n)})_{n\in \N}$ is a bounded sequence in $\R^{m_t}$. This implies that there exists a subsequence  $(a^{(n_k)})_{k\in \N}$ which converges to some $a=(a_1,\dots,a_{m_t})\in \R^{m_t}$. Since for each $j=1,\dots,m_t$ and $k \in \N$ we have that 
 $\underline{a}_{t,j}({\omega^t}^{(n_k)}) 
 \leq 
  a^{(n_k)}_j 
 \leq 
 \overline{a}_{t,j}({\omega^t}^{(n_k)})$, 
 the continuity of both  $\underline{a}_{t,j}:\Omega^t\mapsto \R$ and $\overline{a}_{t,j}:\Omega^t\mapsto \R$ ensures that also
  $\underline{a}_{t,j}({\omega^t}) 
 \leq 
 a_j 
 \leq 
 \overline{a}_{t,j}({\omega^t})$.
 This shows that  $a \in \mathcal{A}_t(\omega^t)$, which in turn by  the characterization of upper hemicontinuity provided in Lemma~\cite[Theorem 17.20]{Aliprantis} demonstrates that $\mathcal{A}_t$ is indeed upper-hemicontinuous.
 Therefore, we see that indeed  Assumption~\ref{asu_A}~(i) holds.
 
 Finally, to see that Assumption~\ref{asu_A}~(ii) holds, fix any $\omega^t$, $\widetilde{\omega}^t \in \Omega^t$ and $a \in \mathcal{A}_t(\omega^t)$. Define $\widetilde{a}=(\widetilde{a}_1,\dots,\widetilde{a}_{m_t})\in \R^{m_t}$ by setting for each $j=1,\dots, m_t$
  \begin{equation}
 	\label{eq:Lipschiz-control}
 	\begin{split}
 		\widetilde{a}_j:&= 
 		a_j
 		+
 		\big(\underline{a}_{t,j}(\widetilde{\omega}^t) -a_j\big)^+ 
 		-
 		\big(a_j-\overline{a}_{t,j}(\widetilde{\omega}^t) \big)^+ \\
 		&= 
 		\underline{a}_{t,j}(\widetilde{\omega}^t) \mathbf{1}_{\{a_j<\underline{a}_{t,j}(\widetilde{\omega}^t)\}} 
 		+ 
 		a_j 
 		\mathbf{1}_{\{\underline{a}_{t,j}(\widetilde{\omega}^t)\leq a_j \leq \overline{a}_{t,j}(\widetilde{\omega}^t) \}}
 		+ 
 		\overline{a}_{t,j}(\widetilde{\omega}^t)
 		\mathbf{1}_{\{ a_j > \overline{a}_{t,j}(\widetilde{\omega}^t) \}}.
 	\end{split}
 	\end{equation}
 	We claim that \eqref{eq_condition_A} holds. 
 	Indeed, on 
 	$\{\underline{a}_{t,j}(\widetilde{\omega}^t)\leq a_j \leq \overline{a}_{t,j}(\widetilde{\omega}^t) \}\subseteq \Omega^t$ we have $\widetilde{a}_j= a_j$ by construction. 
 	Moreover, on the set  $\{a_j<\underline{a}_{t,j}(\widetilde{\omega}^t)\}\subseteq \Omega^t$ 
 	we
 	have 
 	 $|a_j-\widetilde{a}_j|= \underline{a}_{t,j}(\widetilde{\omega}^t)-a_j \leq \underline{a}_{t,j}(\widetilde{\omega}^t)-\underline{a}_{t,j}(\omega^t)$. Furthermore, on the set $\{ a_j > \overline{a}_{t,j}(\widetilde{\omega}^t) \}\subseteq \Omega^t$
 	 we have 
    $|a_j-\widetilde{a}_j|=  a_j-\overline{a}_{t,j}(\widetilde{\omega}^t) 
    \leq \overline{a}_{t,j}(\omega^t)-\overline{a}_{t,j}(\widetilde{\omega}^t) $.
Therefore, since by assumption for each $j=1,\dots, m_t$ both  $\underline{a_{t,j}}:\Omega^t\mapsto \R$ and $\overline{a_{t,j}}:\Omega^t\mapsto \R$ 
are  Lipschitz continuous, there exists a constant $L_{a}>0$ such that
\begin{equation*}
	\begin{split}
\Vert a- \widetilde{a} \Vert
&\leq \sum_{j=1}^{m_t} |a_j-\widetilde{a}_j|\\
&\leq \sum_{j=1}^{m_t}\big( |\underline{a}_{t,j}(\omega^t)-\underline{a}_{t,j}(\widetilde{\omega}^t)|
+ |\overline{a}_{t,j}(\omega^t)-\overline{a}_{t,j}(\widetilde{\omega}^t)|
\big)\\
&\leq \sum_{j=1}^{m_t} 2 L_{a} \sum_{i=1}^t \Vert \omega_i -\widetilde{\omega}_i \Vert
=L_{\mathcal{A},t} \sum_{i=1}^t \Vert \omega_i -\widetilde{\omega}_i \Vert,
\end{split}
\end{equation*}
where $L_{\mathcal{A},t}:= 2 m_t  L_{a}>0$. Hence, we see that indeed that Assumption~\ref{asu_A}~(ii) holds, which finishes the proof.
\end{itemize}
\end{proof}
}
%
%
{
\subsection{Proof of Section~\ref{sec_RobustVSNonRobust}}
In this subsection we provide the proof of Theorem~\ref{thm:RobustVSNonRobust}.
\begin{proof}[Proof of Theorem~\ref{thm:RobustVSNonRobust}]
	We start with the proof of (1).\\
	First, observe that since by assumption each $\Omega^t \ni \omega^t \mapsto \widehat{\PP}_t(\omega^t) \in \mathcal{M}_1^{\max\{1,p\}}(\Omloc)$ is  $L_{\widehat{\mathbb{P}},t}$-Lipschitz continuous  w.r.t.\ the $d_{W_1}$-metric, we have for the choice $\mathcal{P}_t(\omega^t):=\widehat{\mathbb{P}}_t(\omega^t)$ for each $\omega^t \in \Omega^t$, $t\in \{1,\dots,T-1\}$ and $\mathcal{P}_0:=\widehat{\mathbb{P}}_0$  (i.e.\ the special case where the set-valued map $\mathcal{P}_t$ is in fact a singleton defined by the kernel $\widehat{\PP}_t$) that Assumption~\ref{asu_P} is satisfied with $L_{\mathcal{P},t}=L_{\widehat{\mathbb{P}},t}$ and (by following \eqref{eq_prop_22_1}--\eqref{eq_prop_22_1+++}) that
	$$
	\mathcal{C}_{\mathcal{P},t}=
	\max\left\{2^{p-1}\bigg(
	\inf_{\omega_a^t \in \Omega^t}\bigg\{
	\bigg(\int_{\Omloc }\|z\|^p \widehat{\PP}_t(\omega_a^t) (\D z) \bigg)^{\nicefrac{1}{p}}
	+
	L_{\widehat{\mathbb{P}},t} \sum_{i=1}^t \|\omega_{a,i}\|
	\bigg\}\bigg)^p,
	2^{p-1}L_{\widehat{\mathbb{P}},t}^pt^{p-1},~ 1\right\}<\infty.$$
	Likewise,  since by assumption each $\Omega^t \ni \omega^t \mapsto \PP^{\rm{TR}}_t(\omega^t) \in \mathcal{M}_1^{\max\{1,p\}}(\Omloc)$ is  $L_{\mathbb{P}^{\rm{TR}},t}$-Lipschitz continuous w.r.t.\ the $d_{W_1}$-metric, we have for the choice $\mathcal{P}_t(\omega^t):=\mathbb{P}^{\rm{TR}}_t(\omega^t)$ for each $\omega^t \in \Omega^t$, $t\in \{1,\dots,T-1\}$ and $\mathcal{P}_0:=\mathbb{P}^{\rm{TR}}_0$ that Assumption~\ref{asu_P} is satisfied with  $L_{\mathcal{P},t}=L_{\mathbb{P}^{\rm{TR}},t}$ and $\mathcal{C}_{\mathcal{P},t}$ as above but with $L_{\widehat{\mathbb{P}},t}$ replaced by $L_{\mathbb{P}^{\rm{TR}},t}$.

Now, for each  $t=T-1,\dots,1,0$ define\footnote{We denote for any $t=T-1,\dots,1$ the set $C^\alpha(\Omega^t \times \mathcal{A}^t):=\{F:\Omega^t \times \mathcal{A}^t \to \R \, | F \mbox{ is $\alpha$-H\"older continuous}\}$} the following operators 
$\widehat{\T}_t:C^\alpha(\Omega^{t+1} \times \mathcal{A}^{t+1}) \to C^\alpha(\Omega^t \times \mathcal{A}^t)$ 
and 
$\T^{\rm{TR}}_t:C^\alpha(\Omega^{t+1} \times \mathcal{A}^{t+1}) \to C^\alpha(\Omega^t \times \mathcal{A}^t)$
 by setting for any
  $v \in C^\alpha(\Omega^{t+1} \times \mathcal{A}^{t+1})$ 
  and $(\omega^{t},a^{t})\in \Omega^t \times \mathcal{A}^t$

\begin{align}
(\widehat{\T}_tv)(\omega^{t},a^{t}):=& \sup_{a \in \mathcal{A}_t(\omega^t)} \E_{\widehat{\PP}_t(\omega^t)}\big[v\big(\omega^t \otimes_t \cdot, (a^t,a)\big)\big], \label{eq:def:widehatT}\\
(\T^{\rm{TR}}_tv)(\omega^{t},a^{t}):=& \sup_{a \in \mathcal{A}_t(\omega^t)} \E_{\PP^{\rm{TR}}_t(\omega^t)}\big[v\big(\omega^t \otimes_t \cdot, (a^t,a)\big)\big].
\label{eq:def:trueT}
\end{align}
Note that from Lemma~\ref{lem_appendix_backwards_iteration} and Remark~\ref{rem:backwards_iteration}, we see that the operators are well-defined. Now, for each $t=1,\dots,T$ define the operator\footnote{We set $C^\alpha(\Omega^{0} \times \mathcal{A}^{0}):=\R$.}
$\widehat{\T}^{t}:C^\alpha(\Omega^{T} \times \mathcal{A}^{T}) \to C^\alpha(\Omega^{T-t} \times \mathcal{A}^{T-t})$ 
as well as the operator
$\T^{\rm{TR},t}:C^\alpha(\Omega^{T} \times \mathcal{A}^{T}) \to C^\alpha(\Omega^{T-t} \times \mathcal{A}^{T-t})$
by setting for any
$v \in C^\alpha(\Omega^{T} \times \mathcal{A}^{T})$ 
\begin{align}
\widehat{\T}^tv&:=	\widehat{\T}_{T-t} \circ \dots \circ \widehat{\T}_{T-1} v,
\label{eq:def:widehatT-global}\\
\T^{\rm{TR},t}v&:=	\T^{\rm{TR}}_{T-t} \circ \dots \circ \T^{\rm{TR}}_{T-1} v.  
\label{eq:def:trueT-global}
\end{align}
By comparing the recursive iteration \eqref{eq_defn_J_t}--\eqref{eq_defn_Psi_0} with the ones in \eqref{eq:def:widehatT}--\eqref{eq:def:trueT-global}  and applying the dynamic programming principle in Theorem~\ref{thm_main_result} we see that
\begin{equation}\label{eq:backward-iteration-identity}
	\widehat{\mathcal{V}}= \widehat{\T}^T\Psi 
	\qquad \mbox{ and } \qquad
	\mathcal{V}^{\rm{TR}} = \T^{\rm{TR},T}\Psi.
\end{equation}

Now, we claim that for each $t=1,\dots,T$ we have for all $(\omega^{T-t},a^{T-t})\in \Omega^{T-t}\times \mathcal{A}^{T-t}$ that\footnote{We use the usual convention that $\prod_{\emptyset}=1$, i.e.\ the product indexed by the emptyset is $1$.}
\begin{equation}
	\label{eq:upper-bound-recursion}
	\begin{split}
&\Big|(\T^{\rm{TR},t}\Psi)(\omega^{T-t},a^{T-t})-(\widehat{\T}^t\Psi)(\omega^{T-t},a^{T-t})\Big|\\
&\leq 
 L_{\Psi}
\sum_{s={T-t}}^{T-1} \bigg[
\Big(2^{T-(s+1)}  \prod_{u=s+1}^{T-1} \max\left\{L_{\mathcal{A},u}^{\alpha}+L_{\widehat{\PP},u}^{\alpha},1\right\}\Big)
\cdot
 \mu^{\rm{err},\alpha}_{s,T-t}(\omega^{T-t})
\bigg].
\end{split}
\end{equation}
We prove \eqref{eq:upper-bound-recursion} by (forward) induction. To that end, to see that \eqref{eq:upper-bound-recursion} holds for $t=1$ note that by definition of the operators in \eqref{eq:def:widehatT}--\eqref{eq:def:trueT-global} we have for any $(\omega^{T-1},a^{T-1})\in \Omega^{T-1}\times \mathcal{A}^{T-1}$ that
\begin{equation*}
	\begin{split}
	&\Big|(\T^{\rm{TR},1}\Psi)(\omega^{T-1},a^{T-1})-(\widehat{\T}^1\Psi)(\omega^{T-1},a^{T-1})\Big|\\
	&=
	\Big|(\T^{\rm{TR}}_{T-1}\Psi)(\omega^{T-1},a^{T-1})-(\widehat{\T}_{T-1}\Psi)(\omega^{T-1},a^{T-1})\Big|\\
	&=
	\bigg|\sup_{a \in \mathcal{A}_{T-1}(\omega^{T-1})} \E_{\PP^{\rm{TR}}_{T-1}(\omega^{T-1})}\big[\Psi\big(\omega^{T-1} \otimes_{T-1} \cdot, (a^{T-1},a)\big)\big]\\
	& \quad \quad -
	\sup_{a \in \mathcal{A}_{T-1}(\omega^{T-1})} \E_{\widehat{\PP}_{T-1}(\omega^{T-1})}\big[\Psi\big(\omega^{T-1} \otimes_{T-1} \cdot, (a^{T-1},a)\big)\big]
	\bigg|.
	\end{split}
\end{equation*} 
Let $\Pi_{T-1}(d\omega_{T},d\widetilde\omega_{T})\in \mathcal{M}_1(\Omloc\times \Omloc)$ be the optimal coupling of $\PP^{\rm{TR}}_{T-1}(\omega^{T-1})$ and $\widehat{\PP}_{T-1}(\omega^{T-1})$ with respect to $d_{W_1}(\cdot,\cdot)$. Then, the H\"older continuity of $\Psi$ imposed in \eqref{eq_Lipschitz_1} in Assumption~\ref{asu_psi} and Jensen's inequality therefore imply that 
\begin{equation}
	\label{eq:upper-bound-recursion-time1}
	\begin{split}
	&\Big|(\T^{\rm{TR},1}\Psi)(\omega^{T-1},a^{T-1})-(\widehat{\T}^1\Psi)(\omega^{T-1},a^{T-1})\Big|\\
	&\leq 	
	\sup_{a \in \mathcal{A}_{T-1}(\omega^{T-1})} \int_{\Omloc \times \Omloc}
	\big| \Psi(\omega^{T-1}\otimes_{T-1}\omega_T,a^{T-1},a)-
			\Psi(\omega^{T-1}\otimes_{T-1}\widetilde\omega_T,a^{T-1},a) \big|	\,	\Pi_{T-1}(d\omega_{T},d\widetilde\omega_{T})\\
	&\leq 
	L_\Psi \int_{\Omloc \times \Omloc}
			\Vert\omega_T-\widetilde\omega_T\Vert^\alpha		\, \Pi_{T-1}(d\omega_{T},d\widetilde\omega_{T})\\
	&\leq L_\Psi \cdot \big(d_{W_1}\big(\PP^{\rm{TR}}_{T-1}(\omega^{T-1}),\widehat{\PP}_{T-1}(\omega^{T-1})\big)\big)^\alpha
\end{split}
\end{equation}
Moreover, note that by the definition in \eqref{eq:def:err-alpha}
\begin{equation}\label{def-quantity-UB-inducStart}
	\big(d_{W_1}\big(\PP^{\rm{TR}}_{T-1}(\omega^{T-1}),\widehat{\PP}_{T-1}(\omega^{T-1})\big)\big)^\alpha = \mu^{\rm{err},\alpha}_{T-1,T-1}(\omega^{T-1}).
\end{equation}
This and \eqref{eq:upper-bound-recursion-time1}  hence show that
\begin{equation*}
\Big|(\T^{\rm{TR},1}\Psi)(\omega^{T-1},a^{T-1})-(\widehat{\T}^1\Psi)(\omega^{T-1},a^{T-1})\Big|
\leq L_\Psi \cdot	\mu^{\rm{err},\alpha}_{T-1,T-1}(\omega^{T-1}),
\end{equation*}
which indeed coincides with \eqref{eq:upper-bound-recursion} for $t=1$. Now, for the induction step assume that \eqref{eq:upper-bound-recursion} holds for some $t-1\geq 1$ and we aim to show that then \eqref{eq:upper-bound-recursion} also holds for $t$.
To that end, observe that by definition of the operators in \eqref{eq:def:widehatT}--\eqref{eq:def:trueT-global} we have for any $(\omega^{T-t},a^{T-t})\in \Omega^{T-t}\times \mathcal{A}^{T-t}$ that
\begin{equation}
		\label{eq:upper-bound-recursion-timet-Step1}
\begin{split}
	&\Big|(\T^{\rm{TR},t}\Psi)(\omega^{T-t},a^{T-t})-(\widehat{\T}^t\Psi)(\omega^{T-t},a^{T-t})\Big|\\
	&=
	\Big|(\T^{\rm{TR}}_{T-t}\circ\T^{\rm{TR},t-1}\Psi)(\omega^{T-t},a^{T-t})
			-(\widehat{\T}_{T-t}\circ\widehat{\T}^{t-1}\Psi)(\omega^{T-t},a^{T-t})\Big|\\
	&\leq 
	\Big|(\T^{\rm{TR}}_{T-t}\circ\T^{\rm{TR},t-1}\Psi)(\omega^{T-t},a^{T-t})
	-(\T^{\rm{TR}}_{T-t}\circ\widehat{\T}^{t-1}\Psi)(\omega^{T-t},a^{T-t})\Big|\\
	& \quad + \Big|(\T^{\rm{TR}}_{T-t}\circ\widehat{\T}^{t-1}\Psi)(\omega^{T-t},a^{T-t})
	-(\widehat{\T}_{T-t}\circ\widehat{\T}^{t-1}\Psi)(\omega^{T-t},a^{T-t})\Big|.
\end{split}
\end{equation}
Now, note that by definition of the operator $\T^{\rm{TR}}_{T-t}$ in \eqref{eq:def:trueT}, the induction hypothesis, and defintion \eqref{eq:def:err-alpha}
 we see that
\begin{equation}
	\label{eq:upper-bound-recursion-timet-Step2a}
	\begin{split}
&\Big|(\T^{\rm{TR}}_{T-t}\circ\T^{\rm{TR},t-1}\Psi)(\omega^{T-t},a^{T-t})
-(\T^{\rm{TR}}_{T-t}\circ\widehat{\T}^{t-1}\Psi)(\omega^{T-t},a^{T-t})\Big|\\
&\leq	
\sup_{a \in \mathcal{A}_{T-t}(\omega^{T-t})} \E_{\PP^{\rm{TR}}_{T-t}(\omega^{T-t})}\Big[\big|\T^{\rm{TR},t-1}\Psi\big(\omega^{T-t} \otimes_{T-t} \cdot, (a^{T-t},a)\big) 
- \widehat{\T}^{t-1}\Psi\big(\omega^{T-t} \otimes_{T-t} \cdot, (a^{T-t},a)\big)\big|
\Big]\\
&\leq  L_{\Psi}
\sum_{s={T-t+1}}^{T-1} \bigg[
\Big(2^{T-(s+1)}  \prod_{u=s+1}^{T-1} \max\left\{L_{\mathcal{A},u}^{\alpha}+L_{\widehat{\PP},u}^{\alpha},1\right\}\Big)
\cdot \E_{\PP^{\rm{TR}}_{T-t}(\omega^{T-t})}\big[\mu^{\rm{err},\alpha}_{s,T-t+1}(\omega^{T-t}\otimes_{T-t} \cdot)\big] \bigg]\\
&= L_{\Psi}
\sum_{s={T-t+1}}^{T-1} \bigg[
\Big(2^{T-(s+1)}  \prod_{u=s+1}^{T-1} \max\left\{L_{\mathcal{A},u}^{\alpha}+L_{\widehat{\PP},u}^{\alpha},1\right\}\Big)
\cdot
\mu^{\rm{err},\alpha}_{s,T-t}(\omega^{T-t}) \bigg].
	\end{split}
\end{equation}
Moreover, observe that by definition of the operators in \eqref{eq:def:widehatT} and \eqref{eq:def:widehatT-global} together with Lemma~\ref{lem_appendix_backwards_iteration} and Remark~\ref{rem:backwards_iteration}, we have that $\widehat{\T}^{t-1}\Psi$ satisfies \eqref{eq_Lipschitz_2_thm_1}, i.e., is $\alpha$-H\"older continuous with corresponding constant $L_{\Psi,T-t+1}=2^{t-1} L_{\Psi} \prod_{u=T-t+1}^{T-1} \max\left\{L_{\mathcal{A},u}^{\alpha}+L_{\widehat{\PP},u}^{\alpha},1\right\}$ (using that $L_{\mathcal{P},u}=L_{\widehat{\PP},u}$, as here  $\mathcal{P}_u:=\widehat{\mathbb{P}}_u$).
In addition, let $\Pi_{T-t}(d\omega_{T-t+1},d\widetilde\omega_{T-t+1})\in \mathcal{M}_1(\Omloc\times \Omloc)$ be the optimal coupling of $\PP^{\rm{TR}}_{T-t}(\omega^{T-t})$ and $\widehat{\PP}_{T-t}(\omega^{T-t})$ with respect to $d_{W_1}(\cdot,\cdot)$. Then, using Jensen's inequality, we derive that
\begin{equation}
	\begin{split}
		\label{eq:upper-bound-recursion-timet-Step2b-1}
		&\Big|(\T^{\rm{TR}}_{T-t}\circ\widehat{\T}^{t-1}\Psi)(\omega^{T-t},a^{T-t})
		-(\widehat{\T}_{T-t}\circ\widehat{\T}^{t-1}\Psi)(\omega^{T-t},a^{T-t})\Big| \\
		& \leq 
		\sup_{a \in \mathcal{A}_{T-t}(\omega^{T-t})}
		\int_{\Omloc \times \Omloc} \Big|
		\widehat{\T}^{t-1}\Psi\big(\omega^{T-t}\otimes_{T-t}\omega_{T-t+1},(a^{T-t},a)\big)\\
		& \qquad \qquad \qquad \qquad \qquad \qquad-
		\widehat{\T}^{t-1}\Psi\big(\omega^{T-t}\otimes_{T-t}\widetilde{\omega}_{T-t+1},(a^{T-t},a)\big)
		\Big| \, \Pi_{T-t}(d\omega_{T-t+1},d\widetilde\omega_{T-t+1})\\
		&\leq
		 L_{\Psi,T-t+1}
		\int_{\Omloc \times \Omloc}
		\Vert\omega_{T-t+1}-\widetilde\omega_{T-t+1}\Vert^\alpha
		\, \Pi_{T-t}(d\omega_{T-t+1},d\widetilde\omega_{T-t+1})\\
		&\leq 
		L_{\Psi,T-t+1} \big(d_{W_1}\big(\PP^{\rm{TR}}_{T-t}(\omega^{T-t}),\widehat{\PP}_{T-t}(\omega^{T-t})\big)\big)^\alpha.
	\end{split}
\end{equation}
Moreover, note that by the definition in \eqref{eq:def:err-alpha}
\begin{equation}\label{def-quantity-UB-inducStep}
	\big(d_{W_1}\big(\PP^{\rm{TR}}_{T-t}(\omega^{T-t}),\widehat{\PP}_{T-t}(\omega^{T-t})\big)\big)^\alpha = \mu^{\rm{err},\alpha}_{T-t,T-t}(\omega^{T-t}).
\end{equation}
This, \eqref{eq:upper-bound-recursion-timet-Step2b-1}, and the definition of $L_{\Psi,T-t+1}$ hence show that
\begin{equation}
		\label{eq:upper-bound-recursion-timet-Step2b-2}
	\begin{split}
&\Big|(\T^{\rm{TR}}_{T-t}\circ\widehat{\T}^{t-1}\Psi)(\omega^{T-t},a^{T-t})
-(\widehat{\T}_{T-t}\circ\widehat{\T}^{t-1}\Psi)(\omega^{T-t},a^{T-t})\Big|\\
	&\leq 
	2^{t-1} L_{\Psi} \prod_{u=T-t+1}^{T-1} \max\left\{L_{\mathcal{A},u}^{\alpha}+L_{\widehat{\PP},u}^{\alpha},1\right\}
		\mu^{\rm{err},\alpha}_{T-t,T-t}(\omega^{T-t}).
		\end{split}
\end{equation}
Therefore, combining \eqref{eq:upper-bound-recursion-timet-Step1}, \eqref{eq:upper-bound-recursion-timet-Step2a}, and \eqref{eq:upper-bound-recursion-timet-Step2b-2} demonstrates that indeed for any $(\omega^{T-t},a^{T-t})\in \Omega^{T-t}\times \mathcal{A}^{T-t}$
\begin{equation*}
	\begin{split}
&\Big|(\T^{\rm{TR},t}\Psi)(\omega^{T-t},a^{T-t})-(\widehat{\T}^t\Psi)(\omega^{T-t},a^{T-t})\Big|\\
& \leq 
L_{\Psi} \sum_{s={T-t+1}}^{T-1} \bigg[
\Big(2^{T-(s+1)}  \prod_{u=s+1}^{T-1} \max\left\{L_{\mathcal{A},u}^{\alpha}+L_{\widehat{\PP},u}^{\alpha},1\right\}\Big)
\cdot
\mu^{\rm{err},\alpha}_{s,T-t}(\omega^{T-t}) \bigg]\\
& \quad +
	L_{\Psi} 2^{t-1}  \prod_{u=T-t+1}^{T-1} \max\left\{L_{\mathcal{A},u}^{\alpha}+L_{\widehat{\PP},u}^{\alpha},1\right\}
\mu^{\rm{err},\alpha}_{T-t,T-t}(\omega^{T-t})\\
&=
L_{\Psi} \sum_{s={T-t}}^{T-1} \bigg[
\Big(2^{T-(s+1)}  \prod_{u=s+1}^{T-1} \max\left\{L_{\mathcal{A},u}^{\alpha}+L_{\widehat{\PP},u}^{\alpha},1\right\}\Big)
\cdot
\mu^{\rm{err},\alpha}_{s,T-t}(\omega^{T-t}) \bigg].
		\end{split}
\end{equation*} 
Hence we have proven \eqref{eq:upper-bound-recursion}. 
We can now deduce (1) directly from \eqref{eq:backward-iteration-identity} together with \eqref{eq:upper-bound-recursion} for $t=T$.
%
%

To deduce (2), first note that  since by assumption
$\PP^{\rm{TR}}_t(\omega^t)\in \mathcal{B}_{\varepsilon_t(\omega^t)}^{(q)}\big(\widehat{\PP}_t(\omega^t)\big)$ holds for each $\omega^t \in \Omega^t$ and $t=1,\dots.T-1$, as well as that
$\PP^{\rm{TR}}_0\in \mathcal{B}_{\varepsilon_0}^{(q)}\big(\widehat{\PP}_0\big)$, we immediately get that $0\leq \mathcal{V}^{\rm{TR}}-\mathcal{V}$.

Now we define for each $t=T-1,\dots,1,0$ the 
operator
 $\T^{WC}_t:C^\alpha(\Omega^{t+1} \times \mathcal{A}^{t+1}) \to C^\alpha(\Omega^t \times \mathcal{A}^t)$
by setting for any
$v \in C^\alpha(\Omega^{t+1} \times \mathcal{A}^{t+1})$ 
and $(\omega^{t},a^{t})\in \Omega^t \times \mathcal{A}^t$
\begin{equation}
	\label{eq:def:worstcaseT}
(\T^{WC}_tv)(\omega^{t},a^{t}):= \sup_{a \in \mathcal{A}_t(\omega^t)} \E_{\widetilde{\PP}_t^*(\omega^t,(a^{t},a))}\big[v\big(\omega^t \otimes_t \cdot, (a^t,a)\big)\big], 
\end{equation}
where 
$\Omega^t \times \mathcal{A}^{t+1} \ni (\omega^t,(a^{t},a)) \mapsto  \widetilde{\PP}_t^*(\omega^t,(a^{t},a)) \in  \mathcal{P}_t(\omega^t)$  is the local worst-case measure at time~$t$ defined in \eqref{eq_thm_iip_def} from the dynamic programming principle in Theorem~\ref{thm_main_result}. Moreover, for each $t=1,\dots,T$ we define the operator $\T^{WC,t}:C^\alpha(\Omega^{T} \times \mathcal{A}^{T}) \to C^\alpha(\Omega^{T-t} \times \mathcal{A}^{T-t})$
by setting for any
$v \in C^\alpha(\Omega^{T} \times \mathcal{A}^{T})$ 
\begin{equation}
		\label{eq:def:worstcaseT-global}
\T^{WC,t}v:=	\T^{WC}_{T-t} \circ \dots \circ \T^{WC}_{T-1} v. 
\end{equation}
By definition of the operators \eqref{eq:def:worstcaseT} and \eqref{eq:def:worstcaseT-global} together with the dynamic programming principle in Theorem~\ref{thm_main_result} we see that $\mathcal{V}=\T^{WC,T}\Psi$. As a consequence, we can follow the exact same arguments as for the proof of (1) line by line, but with 
$\widehat{\PP}_t \leftarrow \widetilde{\PP}_t^*$,
 $\widehat{\T}_t \leftarrow \T^{WC}_t$,
$ \widehat{\T}^{t} \leftarrow \T^{WC,t}$,
$L_{\widehat{\mathbb{P}},t} \leftarrow L_{\widehat{\mathbb{P}},t} + L_{\varepsilon,t}$ 
(since by  \eqref{lem_parameter-ball-correspondence}, we have here $L_{\mathcal{P},t}= L_{\widehat{\mathbb{P}},t} + L_{\varepsilon,t}$),
 as well as noticing that due to the assumption that $\PP^{\rm{TR}}_t(\omega^t) \in\mathcal{B}_{\varepsilon_t}^{(q)} \big(\widehat{\PP}_t(\omega^t)\big)$
 we have in \eqref{def-quantity-UB-inducStart} and \eqref{def-quantity-UB-inducStep} 
 in this case by the definition in \eqref{eq:def:epsilon-alpha} for any 
$(\omega^{T-t},(a^{T-t},a))\in \Omega^{T-t}\times \cA^{T-t+1}$ that 
\begin{equation*}
\big(d_{W_1}\big(\PP^{\rm{TR}}_{T-t}(\omega^{T-t}),\widetilde{\PP}^*_{T-t}(\omega^{T-t},(a^{T-t},a))\big)\big)^\alpha
\leq
 2^\alpha\big(\varepsilon_{T-t}(\omega^{T-t})\big)^\alpha
  = 2^\alpha \mu^{\varepsilon,\alpha}_{T-t,T-t}(\omega^{T-t}).
\end{equation*}
This allows us to conclude (2).

Finally, to see that (3) holds, we first note that since for each $t=1,\dots, T-1$ and $\omega^t\in \Omega^t$, $\PP^{\rm{TR}}_t(\omega^t)=\PP_{\theta^{\rm{TR}}_t(\omega^t)}$ holds   for some
	$\theta^{\rm{TR}}_t: \Omega^t \rightarrow \Theta_t$
	satisfying for each $\omega^t\in \Omega^t$ that 
	$\| \theta^{\rm{TR}}_t(\omega^t) - \widehat{\theta}_t(\omega^t)\| \leq \varepsilon_t(\omega^t)$, and that $\PP^{\rm{TR}}_0=\PP_{\theta^{\rm{TR}}_0}$ holds for some $\theta^{\rm{TR}}_0 \in \Theta_0$ satisfying 
	$\| \theta^{\rm{TR}}_0 - \widehat{\theta}_0\| \leq \varepsilon_0$, 
	 we immediately get that $0\leq \mathcal{V}^{\rm{TR}}-\mathcal{V}$.

Now for the upper bound, we can re-use for each $t=T-1,\dots,1,0$ the  operator
$\T^{WC}_t:C^\alpha(\Omega^{t+1} \times \mathcal{A}^{t+1}) \to C^\alpha(\Omega^t \times \mathcal{A}^t)$ as well as for each $t=1,\dots,T$ the operator
$\T^{WC,t}:C^\alpha(\Omega^{T} \times \mathcal{A}^{T}) \to C^\alpha(\Omega^{T-t} \times \mathcal{A}^{T-t})$ introduced
in the proof of (2) above and then follow the exact same arguments as for the proof of (1) line by line, 
but with 
$\widehat{\PP}_t \leftarrow \widetilde{\PP}_t^*$,
$\widehat{\T}_t \leftarrow \T^{WC}_t$,
$ \widehat{\T}^{t} \leftarrow \T^{WC,t}$,
$L_{\widehat{\mathbb{P}},t} \leftarrow L_{\PP_{{\theta}},t} \cdot (L_{\widehat{\theta},t} + L_{\varepsilon,t})$ 
(since by  \eqref{eq:Asu_2.3(iii)_parameter}, we have here $L_{\mathcal{P},t}:=L_{\PP_{{\theta}},t} \cdot (L_{\widehat{\theta},t} + L_{\varepsilon,t})$\,),
as well as noticing that due to the assumption that $\PP^{\rm{TR}}_t(\omega^t)=\PP_{\theta^{\rm{TR}}_t(\omega^t)}$ holds   for some
$\theta^{\rm{TR}}_t: \Omega^t \rightarrow \Theta_t$
satisfying for each $\omega^t\in \Omega^t$ that 
$\| \theta^{\rm{TR}}_t(\omega^t) - \widehat{\theta}_t(\omega^t)\| \leq \varepsilon_t(\omega^t)$,
as well as that by the property of the worst-case measure $\widetilde{\PP}_t^*$ defined in \eqref{eq_thm_iip_def} satisfying $\widetilde{\PP}_t^*(\omega^t,a^{t+1}) \in \mathcal{P}_t(\omega^t)$ there exists a map $\theta^{*}_t:\Omega^t\times \cA^{t+1} \to \Theta_t$ such that  $\|\theta^*_t(\omega^t,a^{t+1})-\widehat{\theta}_t(\omega^t)\| \leq \varepsilon_t(\omega^t)$ and $\widetilde{\PP}_t^*(\omega^t,a^{t+1})=\PP_{\theta^*_t(\omega^t,a^{t+1})}$ 
hold for every $(\omega^t,a^{t+1})\in \Omega^t \times \cA^{t+1}$,
we have in
\eqref{def-quantity-UB-inducStart} and \eqref{def-quantity-UB-inducStep} 
in this case for any 
$(\omega^{T-t},(a^{T-t},a))\in \Omega^{T-t}\times \cA^{T-t+1}$ that  by \ref{eq_a2}
\begin{equation}
	\begin{split}
	\big(d_{W_1}\big(\PP^{\rm{TR}}_{T-t}(\omega^{T-t}),\widetilde{\PP}^*_{T-t}(\omega^{T-t},(a^{T-t},a))\big)\big)^\alpha
	&\leq
	\big(d_{W_1}\big(\PP^{\rm{TR}}_{\theta^{\rm{TR}}_{T-t}(\omega^{T-t})},\PP_{\theta^*_{T-t}(\omega^{T-t},(a^{T-t},a))}\big)\big)^\alpha\\
	&\leq 
	L_{\PP_{\theta},T-t}^\alpha \cdot \|\theta^{\rm{TR}}_{T-t}(\omega^{T-t})-\theta^*_{T-t}(\omega^{T-t},(a^{T-t},a)) \|^\alpha\\
&\leq	2^\alpha L_{\PP_{\theta},T-t}^\alpha\cdot \big(\varepsilon_{T-t}(\omega^{t-t})\big)^\alpha\\
	&= 2^\alpha L_{\PP_{\theta},T-t}^\alpha \cdot \mu^{\varepsilon,\alpha}_{T-t,T-t}(\omega^{T-t}).
	\end{split}
\end{equation}
This allows us to conclude (3).
\end{proof}
}
%
%

\subsection{{Proofs and additional results of Section~\ref{sec_applications}}}

\begin{proof}[{Proof of Proposition~\ref{prop_assumptions_exa1}}]
To verify Assumption~\ref{asu_P}, we aim to apply Theorem~\ref{thm_wasserstein_ambiguity} with $p=0,q=1$. This means we need to show that for all $\omega^t, \widetilde{\omega}^t \in \Omega^t$ there exists some $L_B>0$ such that we have $\operatorname{d}_{W_1}(\widehat{\PP}_t(\omega^t) , \widehat{\PP}_t(\widetilde{\omega}^t)  ) \leq L_B \cdot \left(\sum_{i=1}^t \|\omega_i- \widetilde{\omega}_i\|\right)$. First, for any function $f: \Omega^t \rightarrow \R$, we define the quantity 
$
\|f\|_{\operatorname{Lip}} :=\sup_{\omega^t \neq \widetilde{\omega}^t} \frac{|f(\omega^t)-f(\widetilde{\omega}^t)|}{\sum_{i=1}^t \|\omega_i-\widetilde{\omega_i}\| }.
$
Note that $f: \Omega^t \rightarrow \R$ is Lipschitz-continuous if and only if $\|f\|_{\operatorname{Lip}}  < \infty$.
 Next, note that by construction $\Omega^t \ni \omega^t \mapsto \pi_s(\omega^t)$ is Lipschitz-continuous for all $s=t,\cdots,N-1$ with  Lipschitz constant $L_{\pi_s}:=\|\pi_s\|_{\operatorname{Lip}}$ since the partial derivatives of $\pi_s$ exist and are bounded on $\Omega^t$.  We use this observation to define $L_{\pi}:=\max_{s=t,\dots,N-1} L_{\pi_s}$. Then, we apply the \emph{Kantorovich--Rubinstein duality} (see, e.g. \cite[Remark 6.5]{villani2009optimal}) to compute
\begin{align*}
\operatorname{d}_{W_1}(\widehat{\PP}_t(\omega^t) , \widehat{\PP}_t(\widetilde{\omega}^t)  ) &= \sup_{\substack{f: \Omega^t \rightarrow \R, \\ \|f \|_{\operatorname{Lip} \leq 1}}} \left\{\int_{\Omega^t} f(x)\widehat{\PP}_t(\omega^t; dx) -   \int_{\Omega^t} f(x)\widehat{\PP}_t(\widetilde{\omega}^t; dx) \right\}\\
&= \sup_{\substack{f: \Omega^t \rightarrow \R, \\ \|f \|_{\operatorname{Lip} \leq 1}}} \left\{ \sum_{s=t}^{N-1} f(\mathcal{R}_{s+1}) \pi_s(\omega^t) -    \sum_{s=t}^{N-1} f(\mathcal{R}_{s+1}) \pi_s(\widetilde{\omega}^t) \right\}\\
&= \sup_{\substack{f: \Omega^t \rightarrow \R, \\ \|f \|_{\operatorname{Lip} \leq 1}}} \bigg\{ \sum_{s=t}^{N-1} \left(f(\mathcal{R}_{s+1})-f(0)\right)\left( \pi_s(\omega^t) -\pi_s(\widetilde{\omega}^t)  \right) +   f(0) \sum_{s=t}^{N-1} \left( \pi_s(\omega^t) -\pi_s(\widetilde{\omega}^t) \right) \bigg\}\\
&= \sup_{\substack{f: \Omega^t \rightarrow \R, \\ \|f \|_{\operatorname{Lip} \leq 1}}} \bigg\{ \sum_{s=t}^{N-1} \left(f(\mathcal{R}_{s+1})-f(0)\right)\left( \pi_s(\omega^t) -\pi_s(\widetilde{\omega}^t)  \right)\bigg\} \\
&\leq  \sup_{\substack{f: \Omega^t \rightarrow \R, \\ \|f \|_{\operatorname{Lip} \leq 1}}} \bigg\{ \sum_{s=t}^{N-1} |\mathcal{R}_{s+1}| L_{\pi} \cdot \left(\sum_{i=1}^t \|\omega_i- \widetilde{\omega}_i\|\right)\bigg\} \\
&\leq (N-t) \cdot C \cdot  L_{\pi} \cdot \left(\sum_{i=1}^t \|\omega_i- \widetilde{\omega}_i\|\right).
\end{align*}

{ Analogue, using the dual representation for $\operatorname{d}_{W_1}(\cdot, \cdot)$, we get
\begin{align*}
\operatorname{d}_{W_1}(\widehat{\PP}^{\rm{ada.}}(\omega^t) ,\widehat{\PP}^{\rm{ada.}}(\widetilde{\omega}^t)  ) 
&= \sup_{\substack{f: \Omega^t \rightarrow \R, \\ \|f \|_{\operatorname{Lip} \leq 1}}} \left\{  \frac{1}{N+t}\sum_{i=1}^{t} \left(f(\omega_i)  -    f(\widetilde{\omega}_i)\right)\right\}\\
&\leq \frac{1}{N+t} \left(\sum_{i=1}^t \|\omega_i- \widetilde{\omega}_i\|\right).
\end{align*}

Hence, according to Theorem~\ref{thm_wasserstein_ambiguity},  Assumption~\ref{asu_P} is fulfilled for both, the ambiguity sets defined in  \eqref{eq_P_hedging} and the ambiguity sets defined in \eqref{eq_P_hedging_adaptive}.
}

Next, to verify Assumption~\ref{asu_psi}, note first that $\Omloc = {[-C,C]^d}$ is compact and $p=0$. Hence, Assumption~\ref{asu_psi}~(ii) follows directly as $(\omega,a) \mapsto \Psi(\omega,a)$ defined in \eqref{eq_Psi_hedging} is continuous. Therefore, we only need to verify Assumption~\ref{asu_psi}~(i). 
{To this end, we first note that 
\begin{equation}\label{eq_hedge_payoff_proof_1}
\Omega \times \mathcal{A}^T \ni \left((\omega_1,\dots,\omega_T ),~\left(d_0,\Delta_0,\dots,\Delta_{T-1}\right)\right) \mapsto  d_0+\sum_{i=1}^d\sum_{j=0}^{T-1}   \Delta_j^i \left\{S_0^i \left(\prod_{k=1}^{j}(\omega_k^i+1)\right)\cdot \omega_{j+1}^i\right\}
\end{equation}
is continuously differentiable, and $\Omega \times \mathcal{A}^T$ is compact, hence \eqref{eq_hedge_payoff_proof_1} is Lipschitz continuous.
Moreover, the map
\begin{equation}\label{eq_hedge_payoff_proof_2}
	\Omega \ni (\omega_1,\dots,\omega_T) \mapsto \Big(S_0(\omega_1+1), S_0{\textstyle\prod\limits_{k=1}^{2}(\omega_k+1), \dots, S_0\prod\limits_{k=1}^{T}} (\omega_k+1)\Big)\in \R^{d \cdot T}
\end{equation}
is continuously differentiable, hence also Lipschitz continuous as $\Omega$ is compact.
Therefore, since by assumption the payoff function of the derivative $\Phi$ is $\beta$-Hölder-continuous and $\Omloc$ here is compact, we also have that the map $\Upsilon:\Omega \times \mathcal{A}^T \to \R$  defined for any $(\omega^T,a^T):=\left((\omega_1,\dots,\omega_T ),~\left(d_0,\Delta_0,\dots,\Delta_{T-1}\right)\right) \in \Omega \times \mathcal{A}^T$ by
\begin{equation}\label{eq_hedge_payoff_proof_3}
	\begin{split}
	&\Upsilon\left((\omega_1,\dots,\omega_T ),~\left(d_0,\Delta_0,\dots,\Delta_{T-1}\right)\right)\\
	&:=  d_0+\sum_{i=1}^d\sum_{j=0}^{T-1}   \Delta_j^i \left\{S_0^i \left(\prod_{k=1}^{j}(\omega_k^i+1)\right)\cdot \omega_{j+1}^i\right\}
	-\Phi \Big(S_0(\omega_1+1), S_0{\textstyle\prod\limits_{k=1}^{2}(\omega_k+1), \dots, S_0\prod\limits_{k=1}^{T}} (\omega_k+1)\Big)
\end{split}
\end{equation}
is $\beta$-H\"older continuous. Moreover, recall that $U$ defined in \eqref{eq_u_prospect} is $\mathfrak{a}$-H\"older continuous.
This and the fact that $\Psi(\omega,a)=-U(\Upsilon(\omega,a))$ holds for any 
$(\omega,a) \in \Omega \times \mathcal{A}^T$, see \eqref{eq_Psi_hedging}, imply that $\Psi$ is  $\alpha$-H\"older continuous with $\alpha:=\mathfrak{a} \cdot \beta\in (0,1)$.

}
\end{proof}
{
\begin{lem}\label{lem:epsilon_multidim}
Assume the setting described in Section~\ref{subsec_hedging}, and let $d \geq 2$. Then, we have for all $t = 0,1,\dots,T-1$ and for all $\alpha \in (0,1)$ that
\begin{equation}
\PP^* \left(  W_1(\widehat{\PP}^{\rm{ada.}}_t,\PP^*)\geq \widehat{\varepsilon}_t \right) \leq \alpha
\end{equation}
with \begin{align*}
\widetilde{\varepsilon}_t:= \frac{64}{3\alpha} \bigg[\gamma^*_t+(N+t)^{-1/2} \bigg( &(C\tfrac{\sqrt{d}}{2}-\gamma^*_t)+\log\left(C \tfrac{\sqrt{d}}{2\gamma^*_t}\right) 2C \sqrt{d} \lceil d/2 \rceil \\
&+  \sum_{k=2}^{\lceil d/2 \rceil} {{\lceil d/2 \rceil} \choose k } (2C \sqrt{d})^k \bigg(\tfrac{\left(C\tfrac{\sqrt{d}}{2}\right)^{1-k}-{\gamma^*_t}^{1-k}}{1-k}\bigg)\bigg)\bigg]
\end{align*}
and $\gamma^*_t = \frac{2C \sqrt{d}}{(N+t)^{1/(2 \lceil d/2 \rceil)}-1}$, where $\PP^*$ denotes the true but unknown distribution of the $d$-dimensional return $\omega_{t+1} \in \Omloc$.
\end{lem}
\begin{proof}
Let $t \in \{0,1,\dots,T-1\}$. We apply \cite[Theorem 1.1]{boissard2014mean} which gives 
\begin{align*}
\E_{\PP^*}\left[W_1(\widehat{\PP}^{\rm{ada.}}_t,\PP^*)\right] \leq  \frac{64}{3} \left(\gamma+{(N+t)}^{-1/2} \int_a^{\Delta_{\Omloc}/4} N(\Omloc,\delta)^{1/2}\, \D \delta \right)
\end{align*}
for all $\gamma > 0 $, where $\Delta_{\Omloc} = 2C\sqrt{d}$ denotes the diameter of $\Omloc = [-C,C]^d$ and $N(\Omloc,\delta) \leq \left(\frac{2C\sqrt{d}}{\delta}+1\right)^d$ denotes the covering number of $\Omloc$, i.e., the minimal $n \in \N$ such that there exist $x_1,\dots,x_n \in \Omloc$ with $\Omloc \subset \cup_{i=1}^n \{x \in \R^d~|~ \|x-x_i|| < \delta\}$. Using the explicit expressions for covering number and diameter gives 
\begin{align*}
\E_{\PP^*}\left[W_1(\widehat{\PP}^{\rm{ada.}}_t,\PP^*))\right] &\leq  \frac{64}{3} \left(\gamma+{(N+t)}^{-1/2} \int_\gamma^{\frac{{C \sqrt{d}}}{2}} \left(\tfrac{2C\sqrt{d}}{\delta}+1\right)^{d/2} \D \delta \right) \\
 &\leq  \frac{64}{3} \left(\gamma+{(N+t)}^{-1/2} \int_\gamma^{\frac{{C \sqrt{d}}}{2}} \left(\tfrac{2C\sqrt{d}}{\delta}+1\right)^{\lceil d/2 \rceil } \D \delta \right) \\
  &=  \frac{64}{3} \left(\gamma+{(N+t)}^{-1/2} \sum_{k=0}^{\lceil d/2 \rceil } \int_\gamma^{\frac{{C \sqrt{d}}}{2}} {{\lceil d/2 \rceil}  \choose k } \cdot  \left(\tfrac{2C\sqrt{d}}{\delta}\right)^{k} \D \delta \right) \\
    &=  \frac{64}{3} \bigg[\gamma+{(N+t)}^{-1/2} \bigg(\Big(C\tfrac{\sqrt{d}}{2}-\gamma\Big)+\log\left(\tfrac{C\sqrt{d}}{2\gamma}\right)2C \sqrt{d} {\lceil d/2 \rceil} \\
    &\hspace{2cm}+\sum_{k=2}^{\lceil d/2 \rceil} {{\lceil d/2 \rceil}  \choose k } (2C \sqrt{d})^k \bigg(\frac{\left(C\tfrac{\sqrt{d}}{2}\right)^{1-k}-\gamma^{1-k}}{1-k}\bigg)  \bigg)\bigg] =: f(\gamma).
\end{align*}
To obtain the smallest possible upper bound, we minimize $f$ w.r.t.\,$\gamma$ which leads to the first order condition 
\[
f'(\gamma)=\frac{64}{3}\left(1-(N+t)^{-1/2}\left(\tfrac{2C\sqrt{d}}{\gamma}+1\right)^{\lceil d/2 \rceil} \right)= 0
\]
with solution $\gamma^*_t = \frac{2C \sqrt{d}}{(N+t)^{\frac{1}{2 \lceil d/2 \rceil}}-1}$ satisfying the second order condition $f''(\gamma^*_t)>0$. Setting $\widehat{\varepsilon_t}:= \frac{f(\gamma^*_t)}{\alpha}$ and applying the Markov inequality then concludes the proof by
\[
\PP^* \left(  W_1(\widehat{\PP}^{\rm{ada.}}_t,\PP^*)\geq \widehat{\varepsilon}_t \right) \leq \frac{\E_{\PP^*}\left[W_1(\widehat{\PP}^{\rm{ada.}}_t,\PP^*)\right]}{\frac{f(\gamma^*_t)}{\alpha}} \leq \frac{f(\gamma^*_t)}{ \frac{f(\gamma^*_t)}{\alpha}} = \alpha.
\]
\end{proof}
}
\section*{Acknowledgments}
\noindent
A. Neufeld gratefully acknowledges financial support by 
the MOE AcRF Tier 2 Grant \textit{MOE-T2EP20222-0013}
and
the Nanyang Assistant Professorship Grant (NAP Grant) \textit{Machine Learning based Algorithms in Finance and Insurance}, J. Sester is grateful for financial support
 by the NUS Start-Up Grant \emph{Tackling model uncertainty in Finance with machine learning}. 
\bibliographystyle{plain} 
\bibliography{literature}
\end{document}